%% file: main.tex
\documentclass[11pt]{amsart}
\usepackage{amsthm, amssymb, amsmath, amsfonts, amsrefs}
\usepackage{graphicx} 
\usepackage{mathtools}
\usepackage{soul}
\usepackage{ytableau,tikz,varwidth}
\usepackage{latexsym,hyperref, caption,tikz,comment,subcaption}
\usetikzlibrary{calc}
\usepackage{tikz-cd}
\usepackage{enumitem}
\usetikzlibrary{graphs}
\usetikzlibrary{graphs.standard}
\usetikzlibrary{shapes.misc}
\usetikzlibrary {shapes.geometric} 
\usepackage{nicematrix}
\usepackage{blkarray, bigstrut}
\usepackage{forest}
\usetikzlibrary{trees, positioning}
\usetikzlibrary{fit}

\numberwithin{equation}{section}
\tikzset{
  curarrow/.style={
  rounded corners=8pt,
  execute at begin to={every node/.style={fill=red}},
    to path={-- ([xshift=-50pt]\tikztostart.center)
    |- (#1) node[fill=white] {$\scriptstyle \delta$}
    -| ([xshift=50pt]\tikztotarget.center)
    -- (\tikztotarget)}\
    }
}

\tikzcdset{diagrams={arrows={shorten >=-0.5ex,shorten <=-0.5ex}}}

\forestset{
  mytree/.style={
    for tree={
      draw,
      circle,
      minimum size=6mm,
      inner sep=1pt,
      l=1.2cm,
      s sep=8mm,
      edge={-stealth},
      align=center
    }
  }
}

\newtheorem{theorem}{Theorem}[section]
\newtheorem{corollary}[theorem]{Corollary}
\newtheorem{lemma}[theorem]{Lemma}
\newtheorem{proposition}[theorem]{Proposition}

\theoremstyle{definition}

\newtheorem{remark}[theorem]{Remark}

\def\opn#1#2{\def#1{\operatorname{#2}}}
\opn\Cl{Cl} 
\opn\pdim{pdim} 
\opn\Im{Im} 
\opn\Ker{Ker} 
\opn\ini{in} 
\opn\typ{type} 
\opn\Par{par}  
\opn\root{root }  
\opn\son{son}  
\opn\p{par} 
\opn\st{SPT} 
\opn\lapla{$L$}
\opn\rlapla{$\widetilde{L}$}

\newcommand{\abs}[1]{\left\lvert #1 \right\rvert}

\DeclareOldFontCommand{\rm}{\normalfont\rmfamily}{\mathrm}

\newcommand{\N}{\mathbb{N}}

\newcommand{\Knn}{K_{n+1}}
\newcommand{\pr}{\mathrm{Pr}}
\newcommand{\PF}{\mathrm{PF}}
\newcommand{\SPF}{\mathrm{sPF}}
\newcommand{\SPT}{\mathrm{SPT}}
\newcommand{\U}{\mathcal{U}}
\newcommand{\Un}{\mathcal{U}_n}
\newcommand{\UnFl}{\mathcal{U}_n^{1\not\sim F_\ell}}
\newcommand{\Bn}{\mathcal{B}_n}
\newcommand{\Gl}{G_{\ell}}
\newcommand{\Cnl}{\mathcal{C}_{n,\ell}}
\newcommand{\Cnlp}{\mathcal{C}_{n,\ell}^{p}}
\newcommand{\Cnlminusonep}{\mathcal{C}_{n,\ell-1}^{p}}
\newcommand{\Cnpplustwop}{\mathcal{C}_{n,p+2}^{p}}
\newcommand{\Cnllminustwo}{\mathcal{C}_{n,\ell}^{\ell-2}}
\newcommand{\Pmax}{\mathrm{P}_{\max}^{\ell}}
\newcommand{\Pmaximum}{\mathrm{P}_{\max}}
\newcommand{\F}{\mathcal{F}}
\newcommand{\Tnl}{\mathcal{T}_{n,\ell}}
\newcommand{\Tnlminusonep}{\mathcal{T}_{n,\ell-1}^{p}}
\newcommand{\Tnpplustwop}{\mathcal{T}_{n,p+2}^{p}}
\newcommand{\TildeTnpplustwop}{\widetilde{\mathcal{T}}_{n,p+2}^{p}}
\newcommand{\hatCnl}{\widehat{\mathcal{C}}_{n,\ell-1}^{p}}
\newcommand{\hatCnp}{\widehat{\mathcal{C}}_{n,p+2}^{p}}

\DeclareOldFontCommand{\rm}{\normalfont\rmfamily}{\mathrm}

\begin{document}

\title[Enumeration of uprooted trees and spherical parking functions]{Enumeration of certain subsets of uprooted trees and spherical parking functions}
\author{Nayana Shibu Deepthi}
\address{Department of Mathematical Sciences, IISER Mohali, Knowledge City, Sector 81, SAS Nagar, Punjab, 140-306, India}
\email{nayanashibu@iisermohali.ac.in}
\author{Chanchal Kumar}
\address{Department of Mathematical Sciences, IISER Mohali, Knowledge City, Sector 81, SAS Nagar, Punjab, 140-306, India}
\email{chanchal@iisermohali.ac.in}
\author{Gargi Lather}
\address{Department of Mathematics, Chennai Mathematical Institute, Chennai, Tamil Nadu, 603-103, India} 
\email{gargilather@gmail.com}

\keywords{Spherical parking functions, uprooted trees, skeleton ideals, matrix tree theorem}
\subjclass[2020]{Primary: 05C30, 05C05, 05C50; Secondary: 05E40}

\begin{abstract}
Spherical $G$-parking functions are a distinguished subset of standard monomials, arising from the skeleton ideals of the $G$-parking function ideal. Explicit enumeration formulas for spherical $G$-parking functions are known only for a few classes of graphs. In this paper, we consider a family of graphs $G_{\ell}$ ($1\leq \ell \leq n-2$), obtained from the complete graph $K_{n+1}$ by deleting the $\ell$ edges joining vertex $1$ to the vertices in $F_\ell= \{n-\ell+1, \ldots, n\}$. The uprooted spanning trees of $G_{\ell}-\{0\}$ correspond to the set $\mathcal{U}_n^{1\not\sim F_\ell}$ of uprooted trees with vertex set $[n]$ in which vertex $1$ is not adjacent to any vertex in $F_\ell$, and we establish that $|\mathcal{U}_n^{1\not\sim F_\ell}| = (n-1)^{n-\ell-2}(n-2)^{\ell}(n-\ell-1)$. We derive this formula combinatorially and independently recover it as an application of the matrix tree theorem, obtaining some combinatorial identities as consequences. Finally, we determine the number of spherical $G_{\ell}$-parking functions as $|\mathrm{sPF}(G_{\ell})| = (n-1)^{n-3}(n-\ell-1)^2$.
\end{abstract}

\maketitle

\input{sec01}

\input{sec02}

\input{sec03}
\input{sec04}

\input{sec05}


\section*{Acknowledgments}
The first author expresses gratitude for the financial assistance received under the Institute Postdoctoral Fellowship offered by the Indian Institute of Science Education and Research (IISER) Mohali.
The third author gratefully acknowledges Chennai Mathematical Institute for the postdoctoral fellowship and the Infosys Foundation for partial financial support.


\bibliography{main}

\end{document}

%% file: sec01.tex
\section{Introduction}\label{sec:intro}

Parking functions, introduced by Konheim and Weiss~\cite{KW66} in the context of computer storage allocation, have since become central objects in enumerative and algebraic combinatorics.
Recall that a sequence $(p_1,\ldots,p_n)\in\N^n$ is a (\emph{classical}) \emph{parking function of length $n$} if its non-decreasing rearrangement $p_{i_1}\le\cdots\le p_{i_n}$ satisfies $p_{i_j}<j$ for all $1\leq j\leq n$.
The number of parking functions of length $n$ is $(n+1)^{n-1}$. 
A clever proof of this formula, based on a circular street argument, was given by Pollak \cite{Rio69}. 
By Cayley's formula, $(n+1)^{n-1}$ is also the number of labeled trees with $n+1$ vertices. 
This equality in cardinality led to many bijective constructions.
Kreweras~\cite{Kre80} gave a recursively defined bijection, and Stanley~\cite{Sta96} formulated the problem of finding a non-recursive bijective proof. 
Further, several solutions have been found, one of which can be found in~\cite{GuiVer11}.

Parking functions have many generalizations. 
An interesting generalization was given by Postnikov and Shapiro~\cite{PS04} in 2004.
Let $G$ be a connected graph with vertex set $V(G)=\{0,1,\ldots,n\}$ and designated root $0$.
Let $A(G)=[a_{ij}]_{0\leq i,j\leq n}$ denote the adjacency matrix of $G$, where $a_{ij}$ denote the number of edges between vertices $i$ and $j$.
We write $[n]=\{1,\ldots,n\}$ for the set of non-root vertices, and define $d_A(i)=\sum_{j\in V(G)\setminus A}a_{ij}$, to be the number of edges from $i$ to vertices outside $A$, for $i\in A\subseteq[n]$. 
A function $\mathcal{P}\colon[n]\to\N$ is called a \emph{$G$-parking function} (with respect to root $0$) if, for every non-empty subset $A\subseteq[n]$, there exists $i\in A$ such that $\mathcal{P}(i)<d_A(i)$.  
The set of all $G$-parking functions is denoted $\PF(G)$.
Further, Postnikov and Shapiro~\cite{PS04} associated the \emph{$G$-parking function ideal}
\[
  \mathcal{M}_G \;=\; \biggl\langle m_A=\prod_{i\in A}x_i^{d_A(i)}
             \;:\;\emptyset\ne A\subseteq[n]\biggr\rangle
  \;\subseteq\;R=\mathbb{K}[x_1,\ldots,x_n],
\]
to the graph $G$ and showed that its standard monomials correspond bijectively to the $G$-parking functions.
When $G$ is the complete graph $K_{n+1}$, the $G$-parking functions reduce precisely to the classical parking functions of length $n$.

Chebikin and Pylyavskyy~\cite{CP05} constructed bijections between $G$-parking functions and spanning trees of $G$. 
Later, Perkinson, Yang, and Yu~\cite{PYY17} introduced a \emph{Depth-First-Search} variant of Dhar's burning algorithm~\cite{Dhar90} to produce an explicit bijection $\phi\colon\PF(G)\to\SPT(G)$ between $G$-parking functions and the set of all spanning trees $\SPT(G)$ of $G$.
Gaydarov and Hopkins~\cite{GH16} extended this bijection to multigraphs.

Motivated by the combinatorics of chip-firing on graphs, Dochtermann~\cite{Doc17} introduced the following subideals of $\mathcal{M}_G $.
For $0\le k\le n-1$, the \emph{$k$-skeleton ideal} of $G$ is
\[
  \mathcal{M}_{G}^{(k)} \;=\; \bigl\langle m_A\;:\;\emptyset\ne A\subseteq[n],\;
             \abs{A}\le k+1\bigr\rangle \;\subseteq\;\mathcal{M}_G .
\]
The standard monomials of $\mathcal{M}_{G}^{(k)}$ form a family of sets that changes with the value of $k$. 
In particular, $\mathcal{M}_{G}^{(n-1)}=\mathcal{M}_G $, and the standard monomials coincide with the $G$-parking functions. 
The filtration $\mathcal{M}_G^{(n-2)}\subset \mathcal{M}_G$ naturally determines the class of \emph{spherical $G$-parking functions} introduced in~\cite{Doc17}. 
A function $\mathcal{P}\colon [n]\to \N$ is called a \emph{spherical $G$-parking function} if
\[
\mathbf{x}^\mathcal{P}=\prod_{i\in [n]} x_i^{\mathcal{P}(i)}\in \mathcal{M}_G \setminus \mathcal{M}_G^{(n-2)}.
\]
We denote the set of spherical $G$-parking functions by $\SPF(G)$.
For further studies on skeleton ideals and skeletal generalizations of parking functions, see~\cites{KLR22, BCLMOW25}.
While the enumeration of spherical $G$-parking functions is known for certain families of graphs, explicit formulas are available only in a few cases.

Throughout this paper, unless stated otherwise, we consider only finite simple vertex-labeled graphs $G=(V(G),E(G))$, with vertex set $V(G)=\{0\}\cup [n]$ and designated root $0$.
If there exists an edge between $u$ and $v$ in $G$, then we say that $u$ and $v$ are \emph{adjacent} to each other and write $u\sim v$, and we denote this edge by $e_{u,v}$. 
If there is no edge between them, then we say that $u$ and $v$ are \emph{not adjacent} to each other and write $u\nsim v$.
Let $G-\{v\}$ denote the graph obtained from $G$ by deleting the vertex $v\in V(G)$ along with all edges incident to it.
For $E'\subseteq E(G)$, let $G-E'$ denote the graph obtained from $G$ by deleting all edges in $E'$.
In particular, we write $G-\{e_{u,v}\}$ when $E'=\{e_{u,v}\}$ is a single edge.
A \emph{rooted tree} is a tree with a distinguished vertex called the \emph{root}, where each edge is directed away from the root, defining a natural orientation for the graph. 
For a rooted tree $T$, we use $\son_T(v)$ and $\Par_T(u)$ to denote the set of sons of a vertex $v \in V(T)$ and the parent of a non-root vertex $u \in V(T)$ in $T$, respectively. 
For any undefined terms 
and notations in this article, see \cites{Aigner, Bapat}.

An \emph{uprooted} (\emph{spanning}) \emph{tree} is a rooted (spanning) tree in which the root is strictly larger than all of its children.
We denote the set of all uprooted spanning trees of a graph $G$ by $\U(G)$.
In particular, let $\Un$ denote the set of all uprooted trees with vertex set $[n]$.
It is known that $|\Un|=(n-1)^{n-1}$ (see~\cites{CDG, DK25}).
Dochtermann conjectured the existence of a natural bijection between $\SPF(\Knn)$ and $\Un$.
This was proved by Kumar, Lather, and Sonica~\cite{CGS} using a modified Depth-First-Search burning algorithm, and independently by Dochtermann and King~\cite{DK21} (using a Breadth--First--Search variant of Dhar's burning algorithm~\cite{Dhar90}).
Thus, $|\SPF(\Knn)|=|\Un|=(n-1)^{n-1}$.
Moreover, adding or deleting edges between the root $0$ and the other vertices of a graph does not affect the number of spherical parking functions~\cite{CGS}.
Since vertex $1$ is the smallest non-root vertex, it is natural to study how deleting edges incident to vertex $1$ affects the enumeration of spherical parking functions.
Let $\Un^{1\nsim n}$ denote the set of all uprooted trees with vertex set $[n]$ in which vertices $1$ and $n$ are not adjacent.
The cardinality of $\Un^{1\nsim n}$ is given by $|\Un^{1\nsim n}| =(n-1)^{n-3}~(n-2)^2$ (see~\cites{CGS,DK25}). 
For graph $G=\Knn-\{e_{1,n}\}$, Kumar et al.~\cite{CGS} constructed a bijection between $\SPF(G)$ and $\Un^{1\nsim n}$, thereby enumerating $\SPF(G)$.
In this paper, we continue this line of investigation.
We consider graphs obtained from $K_{n+1}$ by deleting additional sets of edges, and derive explicit formulas for $|\SPF(G)|$ for these families.

The paper is organized as follows. 
In Section~\ref{sec:Counting Uprooted Trees}, we introduce the graph $\Gl$ and, using the bijection of Chauve, Dulucq, and Guibert~\cite{CDG} between $\mathcal{U}_{n}$ and a certain subset of labeled rooted trees with vertex set $[n]$, we enumerate the uprooted spanning trees of $\Gl-\{0\}$.
In Section~\ref{sec:Enumeration by MTT}, we recover this enumeration as an application of the matrix tree theorem and obtain several combinatorial identities as consequences. 
In Section~\ref{sec:SPF(Gl)}, we recall the necessary background on spherical $\Gl$-parking functions and outline our approach to their enumeration. 
Finally, in Section~\ref{sec:Counting SPF(Gl)}, we determine $|\SPF(\Gl)|$.

%% file: sec02.tex
\section{\texorpdfstring{Enumeration of uprooted trees in $\UnFl$}{Counting Uprooted Trees Using C-D-G}}\label{sec:Counting Uprooted Trees}

Our aim is to study the spherical parking functions of a family of graphs obtained by deleting certain edges from $\Knn$.
For integers $n\ge3$ and $1\le\ell\le n-2$, we define 
\[
  F_\ell = \{n-\ell+1,\,n-\ell+2, \ldots, n\} \;\subseteq\;[n],
\]
and let $\Gl$ be the graph with vertex set $[n]\cup\{0\}$ and root $0$, obtained from $\Knn$ by removing the $\ell$ edges joining vertex $1$ to each vertex in $F_\ell$. That is, 
\[ 
\Gl = K_{n+1}-\{ e_{1,v} \ \colon v\in F_{\ell}\}.
\]
We write $\Gl'=\Gl-\{0\}$ for the induced subgraph of $\Gl$ with vertex set $[n]$.
Let $\Un$ denote the set of all uprooted trees with vertex set $[n]$, and for $n\geq 3$ and $1 \leq \ell \leq n-2$, we define the subset
\[
\UnFl = \{T\in \Un\ \colon 1\nsim v \text{ for any } v\in F_{\ell}\}\subseteq \Un.
\]
The set of uprooted spanning trees of $\Gl'$ is precisely $\U(\Gl')= \UnFl$.

One of the key results of this section is a closed formula for the cardinality of $\UnFl$.

\begin{theorem}\label{thm:|Un|}
For $n\geq 3$ and $1\leq \ell \leq n-2$, 
\[
|\UnFl| = (n-1)^{n-\ell -2}~(n-2)^{\ell}~(n-\ell-1).
\]    
\end{theorem}

For $\ell=1$, our formula  recovers the count $(n-1)^{n-3}(n-2)^2$, obtained in \cite{CGS}.
Our proof of Theorem~\ref{thm:|Un|} is based on the bijection $\Psi\colon\Un\to\Bn$ of Chauve, Dulucq, and Guibert \cite{CDG}, where $\Bn$ denotes the set of labeled rooted trees with vertex set $[n]$ in which vertex $1$ is a non-root leaf.
We now recall the construction of $\Psi$ from \cite{CDG} (see also \cite{CGS}).

\medskip
\noindent\textbf{Construction of $\Psi$.}
Let $T \in \Un$ with root $r$. Since $T$ is uprooted, we have $r \neq 1$. Consider the maximal increasing subtree $T_0$ of $T$ containing the vertex $1$. 
Deleting the edges of $T_0$ decomposes $T$ into rooted subtrees $T_1,\ldots,T_k$, each with at least one edge, and let $r_i$ denote the root of $T_i$ for $1 \le i \leq k$. 
The root $r$ of $T$ is the root of exactly one of these subtrees, say $T_j$, where $r_j=r$. 
Since $1$ is the root of $T_0$, it will be a leaf of $T_j$.

\noindent
Suppose that $T_0$ has $m$ vertices. Then $T_0$ can be encoded by an increasing tree $\overline{T}_0$ with vertex set $[m]$ together with the set $S_0$ of labels appearing in $T_0$.
We write $T_0=(\overline{T}_0,S_0)$. 
Let $\widetilde{S}_0=(S_0\setminus\{1\})\cup\{r\}$.
The pair $(\overline{T}_0,\widetilde{S}_0)$ determines an increasing tree $\widetilde{T}_0$ whose root is $r_0=\min(\widetilde{S}_0)$. 
We attach the subtree $T_j$ to $\widetilde{T}_0$ at vertex $r$, and then attach each of the remaining subtrees $T_i$ (for $i\neq j$) at the vertex $r_i$.
The resulting rooted tree $T'=\Psi(T)$ has root $r_0$ and vertex $1$ as a non-root leaf.
Hence $\Psi(T)\in \Bn$.

\medskip
\noindent\textbf{Inverse construction.}
We now describe $\Psi^{-1}\colon \Bn\to\Un$.
Let $T' \in \Bn$ with root $r'$. 
Since vertex $1$ is a non-root leaf of $T'$, we have $r' \neq 1$. Let $T'_0$ be the maximal increasing subtree of $T'$ rooted at $r'$. 
Deleting the edges of $T'_0$ decomposes $T'$ into rooted subtrees $T'_1,\ldots,T'_k$, each containing at least one edge, and let $r'_i$ denote the root of $T'_i$ for $1 \le i \le k$. 
Let $T'_j$ be the unique subtree containing the vertex $1$ as a leaf.

\noindent
Replace the label $r'_j$ in $T'_0$ by $1$, and relabel the vertices so that their relative order is preserved. 
This produces an increasing subtree $\widetilde{T}'_0$ rooted at $1$. 
Attach $\widetilde{T}'_0$ to $T'_j$ at the vertex $1$, and then attach each of the remaining subtrees $T'_i$ (for $i\neq j$) at its root $r'_i$. 
The resulting rooted tree $T$ belongs to $\Un$, and by construction satisfies $\Psi(T)=T'$.

As noted earlier, $\UnFl$ is a (proper) subset of $\Un$. 
To enumerate $\UnFl$, we study its image under the bijection $\Psi\colon \Un \longrightarrow \Bn$.
In this section, we show that $\Psi(\UnFl)$ can be decomposed into explicitly described disjoint subsets of $\Bn$, and we determine the cardinality of each subset.
We begin with a classical enumeration result.

\begin{theorem}[\cite{Moon}*{Theorem 3.3}]\label{Thm:forest}
Let $f(n,k)$ denote the number of forests with $n$ labeled vertices consisting of $k$ disjoint trees, where $k$ specified vertices belongs to distinct components. Then, for $1 \le k \le n$, we have  $f(n,k)= k\ n^{n-k-1}.$
\end{theorem}

For $n\geq 3$ and $1\leq \ell \leq n-2$, we define 
\[
\Bn^{1\nsim F_{\ell}} = \{T\in \Bn\colon 1\nsim v \text{ for any } v\in F_{\ell}\}\subseteq \Bn.
\]
Let $T^{\prime} \in \Bn$, and let $T \in \Un$ be the unique tree such that $T^{\prime}=\Psi(T)$. From the construction of $\Psi$, the leaf $1$ is adjacent to a vertex $j$ in $T'$ if and only if $j=\Par_T(1)$. 
Hence,
$$\Psi(\UnFl)\subseteq \Bn^{1\nsim F_{\ell}}.$$
To describe this image more precisely, we introduce two disjoint subsets of $\Bn^{1\nsim F_{\ell}}$ as follows.
\[\mathbf{B} = \{T^{\prime}\in \Bn^{1\nsim F_{\ell}} \colon \root(T^{\prime})= r^{\prime}\in F_{\ell} \text{ and } r^{\prime} \nsim j,  \text{ for all } j > r^{\prime}\},\]
and
\[\mathbf{B}^{\prime} = \{T^{\prime}\in \Bn^{1\nsim F_{\ell}} \colon \root(T^{\prime})= r^{\prime}\in \{2,3,\dots , n-\ell\} \text{ and } r^{\prime} \nsim v, \text{ for all } v\in F_{\ell}\}.\]
We then have the following.

\begin{lemma}\label{lemma:BuB'}
Let $\mathbf{B}$ and $\mathbf{B}^{\prime}$ be the subsets of $\Bn^{1\nsim F_{\ell}}$ defined above.
Then we have \[\mathbf{B}\ \coprod \ \mathbf{B}^{\prime} \subseteq \Psi (\UnFl) . \]
\end{lemma}

\begin{proof}
Recall that for each $T^{\prime}\in \Bn$, by the bijection $\Psi$, there exists a unique $T\in \Un$ such that $T^{\prime}=\Psi(T)$. 
Let $r$ and $r^{\prime}$ denote the roots of $T$ and $T^{\prime}$, respectively.  
Further, we have seen that $\Psi(\UnFl)\subseteq \Bn^{1\nsim F_{\ell}}$.
By construction, the sets $\mathbf{B}$ and $\mathbf{B}^{\prime}$ are disjoint, i.e., $\mathbf{B}\cap \mathbf{B}^{\prime}=\emptyset$.  
We now show that each of these sets is contained in $\Psi(\UnFl)$.

Let $T^{\prime}\in \mathbf{B}$.  
Then $T^{\prime}$ is a tree in $\Bn^{1\nsim F_{\ell}}$ whose root $r^{\prime}$ lies in $F_{\ell}$, and for all $v\in \son_{T^{\prime}}{(r^{\prime})}$, we have $v<r^{\prime}$.  
Therefore, $T^{\prime}\in \UnFl$, and we have $\mathbf{B}\subseteq \UnFl$.
Moreover, the maximal increasing subtree $T^{\prime}_{0}$ of $T^{\prime}$ with root $r^{\prime}$ is precisely $\{r^{\prime}\}$.  
By the construction of $\Psi$, the maximal increasing subtree $T_{0}$ of $T$ rooted at $1$ is $\{1\}$.  
Thus, in this case, $T^{\prime}=\Psi(T)=T$.  
This shows that $\mathbf{B}\subseteq \Psi(\UnFl)$, and $\Psi(T)=T$ for all $T\in \mathbf{B}$.

Now, let $T^{\prime}\in \mathbf{B}^{\prime}$.  
Then $r^{\prime}\in \{2,3,\dots,n-\ell\}$, and by definition, $r^{\prime}$ is not adjacent to any vertex in $F_{\ell}$.  
Since the leaf $1$ in $T^{\prime}$ is also not adjacent to any vertex in $F_{\ell}$, it follows that in the corresponding tree $T$, the parent $\Par_{T}(1)\notin F_{\ell}$.
We next show that no vertex in $F_{\ell}$ can be a son of the vertex $1$ in the tree $T$.  
Consider the maximal increasing subtree $T^{\prime}_{0}$ of $T^{\prime}$ rooted at $r^{\prime}$, and let $\operatorname{son}_{T^{\prime}_{0}}(r^{\prime})$ denote the set of sons of $r^{\prime}$ in $T^{\prime}_{0}$, with $|\operatorname{son}_{T^{\prime}_{0}}(r^{\prime})| = s$.  
In the construction of the inverse of $\Psi$, let $r_{j}^{\prime}$ denote the root of the subtree $T^{\prime}_{j}$ of $T^{\prime}$ containing the leaf $1$.  
Suppose first that $r_{j}^{\prime}=r^{\prime}$.  
Then the maximal increasing subtree $T_{0}$ of $T$ rooted at $1$ is obtained from $T^{\prime}_{0}$ by relabeling $r^{\prime}$ with $1$, while keeping all other vertex labels unchanged.  
Hence, in this case, $\son_{T_{0}}{(1)}=\son_{T^{\prime}_{0}}{(r^{\prime})}$, and therefore none of the vertices in $F_{\ell}$ is a son of vertex $1$ in $T$.
Now suppose $r_{j}^{\prime}\neq r^{\prime}$.  
Then we have $r_{j}^{\prime}>r^{\prime}$, and $T_{0}$ is obtained by replacing the label $r_{j}^{\prime}$ by $1$ in $T^{\prime}_{0}$ and relabeling the vertices so that their relative order is preserved.
Accordingly, the least labeled vertex in $\son_{T^{\prime}_{0}}{(r^{\prime})}$ will be relabeled as $r^{\prime}$.
Regardless of whether $r_{j}^{\prime}$ belongs to $\son_{T^{\prime}_{0}}{(r^{\prime})}$ or not, since $r^{\prime}\notin F_{\ell}$ and $r^{\prime}\nsim v$ for all $v\in F_{\ell}$, we have that at least $s$ vertices in $(\son_{T^{\prime}_{0}}{(r^{\prime})}\setminus\{r_{j}^{\prime}\})\cup \{r^{\prime}\}$ are not in $F_{\ell}$.
It then follows from the construction that $\son_{T_{0}}{(1)}\subseteq\{2,3,\dots , n-\ell\}$.
Hence, in both cases, none of the vertices in $F_{\ell}$ can be either the parent or a son of vertex $1$ in the tree $T$ satisfying $\Psi(T)=T^{\prime}$. 
Hence, $T\in \UnFl$, and therefore $T^{\prime}\in \Psi(\UnFl)$.
Combining this with the earlier part of the proof, we conclude that $\mathbf{B}\ \coprod \ \mathbf{B}^{\prime} \subseteq \Psi (\UnFl)$.
\end{proof}

We have already observed that $\Psi(\UnFl)\subseteq \Bn^{1\nsim F_{\ell}}$, and that the sets $\mathbf{B}$ and $\mathbf{B}^{\prime}$ are disjoint subsets of $\Psi(\UnFl)$.  
To complete the description of $\Psi(\UnFl)$, we now turn our attention to the remaining trees in $\Bn^{1\nsim F_{\ell}}\setminus(\mathbf{B}\coprod\mathbf{B}^{\prime})$, and investigate those whose corresponding preimage under $\Psi$ lies in $\UnFl$.

\begin{proposition}\label{prop:B"}
Let $\mathbf{B}'' = \Psi(\UnFl)\setminus(\mathbf{B}\coprod\mathbf{B}^{\prime})$.
Then the trees in $\mathbf{B}''$ are precisely those for which the root $r^{\prime}$ belongs to $\{2,3,\dots,n-\ell\}$ and is adjacent to exactly one vertex in $F_{\ell}$; that is,
\[
r^{\prime}\sim v \text{ for some } v\in F_{\ell}, \quad\text{and}\quad 
r^{\prime}\nsim w \text{ for all } w\in F_{\ell}\setminus\{v\}.
\]
Moreover, there does not exist an increasing path 
\[r^{\prime}\rightarrow b_{1}\rightarrow \cdots \rightarrow b_{t}\rightarrow \gamma, \quad (t\geq1),\]
with $b_{i}\in \{2,3,\dots,n-\ell\}\setminus\{r^{\prime}\}$ and $\gamma\in F_{\ell}\setminus\{v\}$ satisfying $\gamma < v$, such that either 
\begin{itemize}
    \item $1$ is a descendant of $v$, or
    \item $1$ is a descendant of a vertex $\rho>v$, where there exists an increasing path 
    \[
    r^{\prime}\rightarrow b_{1}\rightarrow \cdots \rightarrow b_{t}\rightarrow \rho
    \]
    with $b_{i}\in \{2,3,\dots,n-\ell\}\setminus\{r^{\prime}\}$.
\end{itemize}
\end{proposition}

\begin{proof}
Let $T^{\prime}\in\Bn^{1\nsim F_{\ell}}$.  
Since $\Bn^{1\nsim F_{\ell}}\subseteq \Bn$, there exists a unique $T\in \Un$ such that $T^{\prime}=\Psi(T)$.  
Let $r$ and $r^{\prime}$ denote the roots of $T$ and $T^{\prime}$, respectively.  
We now examine the trees $T^{\prime}\in \Bn^{1\nsim F_{\ell}}\setminus(\mathbf{B}\coprod\mathbf{B}^{\prime})$ for which the corresponding $T\in\UnFl$.

If $r'=n$, then $T^{\prime}\in \mathbf{B}$, that contradicts our assumption.
Now, suppose that $r^{\prime}\in F_{\ell}\setminus\{n\}$.
As $T^{\prime}\notin \mathbf{B}$, we have that $r^{\prime}$ is adjacent to some $j\in F_{\ell}\setminus\{r^{\prime}\}$ with $j>r^{\prime}$.  
Considering the maximal increasing subtree $T^{\prime}_0$ of $T^{\prime}$ rooted at $r^{\prime}$, we observe that $V(T^{\prime}_0)\subseteq F_{\ell}$, since $r^{\prime}\in F_{\ell}\setminus\{n\}$.  
Therefore, under the inverse construction of $\Psi$, the vertex $r^{\prime}$ is relabeled as $1$, and we obtain $\son_{T_{0}}(1)\subseteq F_{\ell}$ in the maximal increasing subtree $T_{0}$ of $T$.  
Thus, the vertex $1$ is adjacent to some vertices in $F_{\ell}$ in $T$.  
Hence $T\notin\UnFl$, and in this case we have $T^{\prime}\notin\Psi(\UnFl)$.

Let $r^{\prime}\in \{2,3,\dots,n-\ell\}$.  
If the root $r^{\prime}$ is not adjacent to any vertex in $F_{\ell}$, then $T^{\prime}\in\mathbf{B}^{\prime}$.  
Since we are considering $T^{\prime}\notin \mathbf{B}^{\prime}$, it follows that $r^{\prime}$ must be adjacent to at least one vertex in $F_{\ell}$.  
Suppose that $r^{\prime}$ is adjacent to exactly $m$ vertices ($m\geq 2$) in $F_{\ell}$, say $\{w_1 , w_2 ,\dots , w_m\}\subseteq F_{\ell}, \quad \text{with } w_1 < w_2 < \cdots < w_m$.
Consider the maximal increasing subtree $T_{0}^{\prime}$ of $T^{\prime}$ with root $r^{\prime}$.  
Recall that $r_{j}^{\prime}$ denotes the root of the subtree $T_{j}^{\prime}$, which belongs to the collection of subtrees obtained by deleting the edges of $T_{0}^{\prime}$ from $T^{\prime}$, and that contains the leaf $1$.  
We have that $r_{j}^{\prime}$ is either equal to $r^{\prime}$ or a non-root vertex of $T_{0}^{\prime}$.  
Suppose first that $r_{j}^{\prime}=r^{\prime}$.  
Then, by the inverse construction of $T$ from $T^{\prime}$, the maximal increasing subtree $T_{0}$ of $T$ is obtained from $T_{0}^{\prime}$ by relabeling $r^{\prime}$ as $1$, while all other vertices retain their labels.  
Therefore, in this case, $T\notin\UnFl$, since in $T$ the vertex $1$ is adjacent to all the vertices in $\{w_1, w_2, \dots, w_m\}\subseteq F_{\ell}$.
Now, suppose that $r_{j}^{\prime}\neq r^{\prime}$.  
Then, in the inverse construction of $T$ from $T^{\prime}$, the maximal increasing subtree $T_{0}$ of $T$ rooted at $1$ is obtained by replacing $r_{j}^{\prime}$ with $1$ in the label set of $T_{0}^{\prime}$ and relabeling the vertices accordingly.  
If $r_{j}^{\prime}< w_{1}$, then all vertices $w_i$ remain unchanged during the relabeling, and hence in $T$ we have $1\sim w_i$ for all $1\leq i \leq m$.  
If $r_{j}^{\prime}\geq w_{m}$, then after relabeling, we observe that $1\sim w_i$ for all $1\leq i\leq m-1$.  
Furthermore, if $w_1< \cdots < w_t \leq r_{j}^{\prime} < w_{t+1}<\cdots < w_m$, for some $1\leq t < m$, then in the subtree $T_{0}$ obtained from $T_{0}^{\prime}$, the vertex $1$ becomes adjacent to all the vertices in $\{w_{t+1}, w_{t+2}, \dots, w_m\}$. 
Thus, in every case where $r_{j}^{\prime}\neq r^{\prime}$, the leaf $1$ in $T$ is adjacent to at least one vertex in $F_{\ell}$. %
Therefore, $T\notin \UnFl$, and hence $T^{\prime}\notin \Psi(\UnFl)$ whenever the root $r^{\prime}$ of $T^{\prime}$ is adjacent to at least two vertices in $F_{\ell}$.

Now, let us consider the case where $r^{\prime}\in \{2,3,\dots,n-\ell\}$ and $r^{\prime}\sim v$ for some $v\in F_{\ell}$, but $r^{\prime}\nsim w$ for any $w\in F_{\ell}\setminus \{v\}$.  
We claim that if $T^{\prime}$ contains an increasing path of the form
\begin{equation}\label{eq:increasing_path}
r^{\prime}\rightarrow b_{1}\rightarrow \cdots \rightarrow b_{t}\rightarrow \gamma, \quad (t\geq 1),
\end{equation}
where $b_i\in \{2,3, \dots, n-\ell\}\setminus\{r^{\prime}\}$ and $\gamma\in F_{\ell}\setminus \{v\}$ with $\gamma < v$, then $T\notin \UnFl$.

Suppose that such increasing paths exist in $T^{\prime}$.  
Without loss of generality, assume that $\gamma$ is the minimal vertex smaller than $v$ among all such paths in $T^{\prime}$.  
Consider the maximal increasing subtree $T_{0}^{\prime}$ of $T^{\prime}$ rooted at $r^{\prime}$, where $1$ is a descendant of some $r_{j}^{\prime}\in V(T_{0}^{\prime})$ such that the path $r_{j}^{\prime}\rightarrow a_1 \rightarrow \cdots \rightarrow a_{k}\rightarrow 1
$ satisfies $a_1 < r_{j}^{\prime}$.
Suppose first that $r_{j}^{\prime}=r^{\prime}$ or $r_{j}^{\prime}\leq \gamma$.  
Since $\gamma < v$ and $r^{\prime}\sim v$, the inverse construction relabels $r^{\prime}$ as $1$, while the vertex $v$ retains its label.  
Hence, in $T$, the vertex $1$ is adjacent to a vertex in $F_{\ell}$, and therefore $T\notin \UnFl$.
If instead $r_{j}^{\prime}> \gamma$, then the inverse construction replaces $r^{\prime}$ by $1$ and relabels $v$ with the largest element in $V(T_{0}^{\prime})$ that is smaller than $v$.  
Moreover, the existence of an increasing path of the form \eqref{eq:increasing_path} ensures that the vertex $v$ in $T_{0}^{\prime}$ is relabeled by a value at least $\gamma$.  
As a result, in $T$, the vertex $1$ is adjacent to a vertex in $F_{\ell}$, and hence $T\notin \UnFl$.
Therefore, whenever an increasing path of the form \eqref{eq:increasing_path} exists in $T^{\prime}$, we have $T^{\prime}\notin \Psi(\UnFl)$.

Thus, let us now consider those trees $T^{\prime}\in \Bn^{1\nsim F_{\ell}}\setminus(\mathbf{B}\coprod\mathbf{B}^{\prime})$
for which $r^{\prime}\in \{2,3,\dots,n-\ell\}$ and $r^{\prime}\sim v$ for some $v\in F_{\ell}$, while $r^{\prime}\nsim w$ for any $w\in F_{\ell}\setminus\{v\}$, and moreover, no increasing path of the form~\eqref{eq:increasing_path} exists in $T^{\prime}$.  
We now examine this situation by distinguishing two subcases depending on whether the vertex $1$ is a descendant of $v$ or not.
\medskip

\noindent 
$\textbf{Case 1.}$ 
Suppose that $1$ is a descendant of $v$.  
In this case, within $T_{0}^{\prime}$, we necessarily have $r_{j}^{\prime}\geq v$.  
Since there are no increasing paths of the form~\eqref{eq:increasing_path}, the inverse construction relabels $v$ with the largest vertex in $V(T_{0}^{\prime})\cap\{2,3,\ldots,n-\ell\}$.  
It follows that, in $T$, the vertex $1$ is not adjacent to any vertex in $F_{\ell}$, and therefore $T\in \UnFl$.
\medskip

\noindent
$\textbf{Case 2.}$ 
Suppose that $1$ is not a descendant of $v$.  
We consider the maximal increasing subtree $T_{0}^{\prime}$ of $T^{\prime}$ rooted at $r^{\prime}$.  
There are two possibilities, depending on the position of $r_{j}^{\prime}$.
If $r_{j}^{\prime}=r^{\prime}$, then the maximal subtree $T_{0}$ of $T$ is obtained by replacing $r^{\prime}$ with $1$ in $T_{0}^{\prime}$, while all other vertex labels remain unchanged.  
Hence, in $T$, we have $1\sim v$, and therefore $T\notin \UnFl$.
Next, suppose that $r_{j}^{\prime}\neq r^{\prime}$.  
If $r_{j}^{\prime}<v$, then the inverse construction again yields $1\sim v$ in $T$, and hence $T\notin \UnFl$.  
Finally, suppose that $r_{j}^{\prime}>v$.  
Since $r^{\prime}\sim v$ and $r^{\prime}$ is not adjacent to any other vertex in $F_{\ell}$, the increasing path from $r^{\prime}$ to $r_{j}^{\prime}$ in $T^{\prime}$ must be of the form
\[
r^{\prime}\rightarrow b_{1}\rightarrow \cdots \rightarrow b_{t} \rightarrow \rho_{1}\rightarrow \cdots \rightarrow \rho_{m}\rightarrow r_{j}^{\prime}, \quad (t\geq1),
\]
where $b_{i}\in \{2, 3, \dots, n-\ell\}\setminus \{r^{\prime}\}$ and $\rho_{i}\in F_{\ell}\setminus \{v\}$ such that
$b_{1}< \cdots <b_t < \rho_1 < \cdots < \rho_m$.
Since $T^{\prime}$ contains no increasing path of the form~\eqref{eq:increasing_path}, it follows that $\rho_1 > v$.  
That is, there exists a path $r^{\prime}\rightarrow b_{1}\rightarrow \cdots \rightarrow b_{t} \rightarrow \rho$ in $T^{\prime}$, where $\rho > v$ and the vertex $1$ is a descendant of $\rho$.  
Under the inverse construction, $r^{\prime}$ in $T_{0}^{\prime}$ is relabeled as $1$, and since no path of the form~\eqref{eq:increasing_path} exists and $r_{j}^{\prime}>v$, the vertex $v$ is relabeled with the largest vertex in $V(T_{0}^{\prime})\cap\{2,3,\ldots,n-\ell\}$.  
Hence, in $T$, the vertex $1$ is not adjacent to any vertex in $F_{\ell}$, implying $T\in \UnFl$.
Therefore, we conclude that if there exists an increasing path
\[
r^{\prime}\rightarrow b_{1}\rightarrow \cdots \rightarrow b_{t} \rightarrow \rho, \quad (t\geq1),
\]
in $T^{\prime}$ with $b_{i}\in \{2, 3, \dots , n-\ell\}\setminus \{r^{\prime}\}$ for $1\leq i\leq t$, $\rho\in F_{\ell}\setminus \{v\}$ satisfying $\rho > v$, and vertex $1$ is a descendant of $\rho$, then $T^{\prime}\in \Psi(\UnFl).$
\medskip

\noindent
Collecting the above cases, we obtain the stated characterization of the trees in $\mathbf{B}''$.
\end{proof}

\begin{remark}\label{rmrk:Psi}
From Lemma~\ref{lemma:BuB'} and Proposition~\ref{prop:B"}, we have
\[
\Psi(\UnFl) = \mathbf{B}\coprod\mathbf{B}^{\prime} \coprod \mathbf{B}''.
\]
Thus, the cardinality of the image set $\Psi(\UnFl)$ is obtained by summing the cardinalities of the sets $\mathbf{B}$, $\mathbf{B}^{\prime}$, and $\mathbf{B}''$.   
\end{remark}

\begin{proposition}\label{prop:identity001}
For $n \geq 3$ and $1 \leq \ell \leq n-2$, we have
\begin{equation}\label{eq:|B|.2}
\sum_{k=0}^{\ell-1} (n-k-2)(n-2)^{k-1}(n-1)^{n-k-3} = (n-1)^{n-\ell-2}(n-2)^{\ell-1}\ell.
\end{equation}
\end{proposition}

\begin{proof}
The identity follows by a straightforward induction on $\ell$.
\end{proof}

Note that Proposition~\ref{prop:identity001} may be of independent interest and will be used in subsequent sections.

We now compute the cardinality of each of the subsets $\mathbf{B}$, $\mathbf{B}^{\prime}$, and $\mathbf{B}''$ of $\Psi(\UnFl)$.

\begin{lemma}\label{lemma:|B|}
For $n\geq 3$ and $1\leq \ell \leq n-2$, we have
\[
|\mathbf{B}| = (n-\ell - 1)\,(n-1)^{\,n-\ell -2}\,(n-2)^{\,\ell -1}\,\ell.
\]
\end{lemma}

\begin{proof}
By definition, $\mathbf{B} = \{T^{\prime}\in \Bn^{1\nsim F_{\ell}} \colon \root(T^{\prime})= r^{\prime}\in F_{\ell} \text{ and } r^{\prime} \nsim j,  \forall j > r^{\prime}\}$.
Recall that $\mathbf{B}\subseteq \UnFl$, and that if $T^{\prime}\in\mathbf{B}$ then its preimage under $\Psi$ coincides with $T^{\prime}$.

Fix $T^{\prime}\in \mathbf{B}$ with root $r^{\prime} = n-k$, where $0\le k\le \ell-1$.  
By definition, the root $r^{\prime}$ is not adjacent to any vertex in $\{n-k+1,\, n-k+2,\, \dots,\, n\}$.
We now count the number of such trees $T^{\prime}$ for a fixed root $r^{\prime} = n-k$.
Let $\son_{T^{\prime}}(r^{\prime})$ denote the set of all the sons of $r^{\prime}$ in $T^{\prime}$, and let $|\son_{T^{\prime}}(r^{\prime})| = s$.
By definition, $\son_{T^{\prime}}(r^{\prime})\subseteq \{2, 3, \dots, n-k-1\}$.
Therefore, the set $\son_{T^{\prime}}(r^{\prime})$ can be chosen in $\binom{n-k-2}{s}$ ways.
Deleting the vertices $r^{\prime}$ and $1$ from $T^{\prime}$ yields a forest with vertex set $\{2, 3, \dots, n\} \setminus \{r^{\prime}\}$ and exactly $s$ connected components, in which the vertices in $\son_{T^{\prime}}(r^{\prime})$ belong to distinct components.
By Theorem~\ref{Thm:forest}, the number of such forests 
is $s (n-2)^{n-s-3}$.
Finally, the vertex $1$ may be attached as a leaf to any vertex in $\{2, 3, \dots, n-\ell\}$, giving $n-\ell-1$ choices.
Hence, for a fixed $k$, the number of uprooted trees $T^{\prime}$ with vertex set $[n]$ and root $r^{\prime}=n-k$,
in which the leaf $1$ is not adjacent to any vertex in $F_{\ell}$, is given by 
\[
\sum_{s=1}^{n-k-2} \binom{n-k-2}{s}~ s ~(n-2)^{n-s-3} ~ (n-\ell-1).
\]
Since $0 \le k \le \ell-1$, summing over all such $k$ yields
\begin{align*}
|\mathbf{B}| &  = \sum_{k=0}^{\ell -1}\sum_{s=1}^{n-k-2}\binom{n-k-2}{s}~s~(n-2)^{n-s-3}~(n-\ell -1)\\
& = (n- \ell -1) \sum_{k=0}^{\ell -1}(n-k-2)\left(\sum_{s=1}^{n-k-2}\binom{n-k-3}{s-1}~(n-2)^{n-s-3}\right)\\
& \hspace{9cm} (\text{By setting } s-1=j)\\
& = (n- \ell -1) \sum_{k=0}^{\ell -1}(n-2)^{k-1}~(n-k-2)\left(\sum_{j=0}^{n-k-3}\binom{n-k-3}{j}~(n-2)^{n-k-3-j}\right).
\end{align*}
Therefore,
\begin{equation}\label{eq:|B|}
 |\mathbf{B}|  = (n- \ell -1) \left(\sum_{k=0}^{\ell -1}(n-2)^{k-1}~(n-k-2)~(n-1)^{n-k-3}\right).
\end{equation}
To evaluate the above sum, we apply Proposition~\ref{prop:identity001}. 
Combining the identity~\eqref{eq:|B|.2} and \eqref{eq:|B|}, we obtain 
\[
|\mathbf{B}| = (n-\ell - 1)\,(n-1)^{\,n-\ell -2}\,(n-2)^{\,\ell -1}\,\ell.
\]
This completes the proof.
\end{proof}

\begin{lemma}\label{lemma:|B'|}
For $n\geq 3$ and $1\leq \ell\leq n-2$, we have 
\[
|\mathbf{B}^{\prime}| = (n-\ell-1)^{2}~(n-\ell -2)~(n-2)^{\ell -1}(n-1)^{n-\ell -3}.
\]
\end{lemma}

\begin{proof}
Let $\widetilde{\mathbf{B}}^{\prime}$ be the set of rooted trees with vertex set $\{2,3,\dots,n\}$ and root $r^{\prime}\in \{2,3,\dots,n-\ell\}$ such that $r^{\prime}\nsim v$ for all $v\in F_{\ell}$.
Observe that any tree in $\mathbf{B}^{\prime}$ can be obtained from a tree in 
$\widetilde{\mathbf{B}}^{\prime}$ by adjoining the vertex $1$ as a leaf to a vertex 
$i \in \{2, 3, \dots , n-\ell\}$ of the chosen tree.  
Since there are exactly $(n-\ell-1)$ choices for such a vertex $i$, we obtain
\begin{equation}\label{eq:|B'|.1}
|\mathbf{B}^{\prime}| = (n-\ell-1)\,|\widetilde{\mathbf{B}}^{\prime}|.
\end{equation}

Therefore, to determine the cardinality of $\mathbf{B}^{\prime}$, it suffices to count the elements of $\widetilde{\mathbf{B}}^{\prime}$.
Let $\widetilde{T}^{\prime}\in\widetilde{\mathbf{B}}^{\prime}$ with $\root{(\widetilde{T}^{\prime})} = r^{\prime}$, and
let $\son_{\widetilde{T}^{\prime}}{(r^{\prime})} = \{a_1 , a_2 , \dots , a_s\}$ be the set of sons of $r^{\prime}$ in $\widetilde{T}^{\prime}$.
Then $\widetilde{T}^{\prime}$ has the structure shown in Figure~\ref{fig:tildeT'}, where for each $1 \le i \le s$, $\F_{a_i}$ denotes the rooted forest consisting of all descendants of $a_i$, with the roots of its component trees adjacent to $a_i$.

\begin{figure}[ht]
\centering
\begin{tikzpicture}[scale=1]

\node[circle, fill=black, inner sep=1pt] (r) at (0,0) {};
\node[above] at (r) {$r'$};

\node[circle, fill=black, inner sep=1pt] (a1) at (-4,-1.5) {};
\node[above] at (a1) {$a_1$};

\node[circle, fill=black, inner sep=1pt] (a2) at (-2,-1.5) {};
\node[above] at (a2) {$a_2$};

\node at (0,-1.5) {$\dots$};

\node[circle, fill=black, inner sep=1pt] (as) at (2,-1.5) {};
\node[above] at (as) {$a_s$};

\draw (r) -- (a1);
\draw (r) -- (a2);
\draw (r) -- (as);

\node[draw, rectangle, minimum width=1cm, minimum height=0.5cm] (F1) at (-4,-3) {$\F_{a_1}$};
\node[draw, rectangle, minimum width=1cm, minimum height=0.5cm] (F2) at (-2,-3) {$\F_{a_2}$};
\node[draw, rectangle, minimum width=1cm, minimum height=0.5cm] (Fs) at (2,-3) {$\F_{a_s}$};

\draw (a1) -- (F1);
\draw (a2) -- (F2);
\draw (as) -- (Fs);

\end{tikzpicture}
\caption{General structure of a tree $\widetilde{T}^{\prime}\in \widetilde{\mathbf{B}}^{\prime}$ with 
$\son_{\widetilde{T}^{\prime}}(r^{\prime}) = \{a_1, a_2, \dots, a_s\}$.  
Each rectangle $\F_{a_i}$ represents the rooted forest of descendants of $a_i$.}
\label{fig:tildeT'}
\end{figure}

By definition, $\son_{\widetilde{T}^{\prime}}{(r^{\prime})} \subseteq \{2,3, \dots , n-\ell\}\setminus\{r^{\prime}\}$. 
Hence, the set $\son_{\widetilde{T}^{\prime}}(r^{\prime})$ can be chosen in $\binom{n-\ell-2}{s}$ ways.
Since $r^{\prime}\in\{2,3,\dots,n-\ell\}$, there are $n-\ell-1$ choices for the root $r^{\prime}$.
Moreover, $\widetilde{T}^{\prime}-\{r^{\prime}\}$ is a rooted forest with vertex set $\{2,3,\dots,n\}\setminus\{r^{\prime}\}$ and $s$ connected components (see Figure~\ref{fig:tildeT'.1}).
By Theorem~\ref{Thm:forest}, the number of such forests is $s(n-2)^{n-s-3}$.
Therefore,
\[
|\widetilde{\mathbf{B}}^{\prime}| = \sum_{s=1}^{n-\ell -2}\binom{n-\ell-2}{s}~s~(n-2)^{n-s-3}~(n-\ell-1).
\]
As $\binom{n-\ell-2}{s}~s=(n-\ell-2)~\binom{n-\ell -3}{s-1}$, we have 
\begin{equation}\label{eq:|tildeB'|}
|\widetilde{\mathbf{B}}^{\prime}| = (n-\ell-1)~(n-\ell-2)\sum_{s=1}^{n-\ell -2}\binom{n-\ell-3}{s-1}~(n-2)^{n-s-3}.
\end{equation}

\begin{figure}[ht]
\centering
\begin{tikzpicture}[scale=1]

\node[circle, fill=black, inner sep=1pt] (a1) at (-4,-1.5) {};
\node[above] at (a1) {$a_1$};

\node[circle, fill=black, inner sep=1pt] (a2) at (-2,-1.5) {};
\node[above] at (a2) {$a_2$};

\node at (0,-1.5) {$\dots$};

\node[circle, fill=black, inner sep=1pt] (as) at (2,-1.5) {};
\node[above] at (as) {$a_s$};

\node[draw, rectangle, minimum width=1cm, minimum height=0.5cm] (F1) at (-4,-3) {$\F_{a_1}$};
\node[draw, rectangle, minimum width=1cm, minimum height=0.5cm] (F2) at (-2,-3) {$\F_{a_2}$};
\node[draw, rectangle, minimum width=1cm, minimum height=0.5cm] (Fs) at (2,-3) {$\F_{a_s}$};

\draw (a1) -- (F1);
\draw (a2) -- (F2);
\draw (as) -- (Fs);

\end{tikzpicture}
\caption{The forest $\widetilde{T}^{\prime}-\{r^{\prime}\}$, where  $\widetilde{T}^{\prime}\in \widetilde{\mathbf{B}}^{\prime}$ and
$\son_{\widetilde{T}^{\prime}}(r^{\prime}) = \{a_1, a_2, \dots, a_s\}$.} 
\label{fig:tildeT'.1}
\end{figure}

Thus, by combining \eqref{eq:|B'|.1} and \eqref{eq:|tildeB'|},  we have
\begin{align*}
|\mathbf{B}^{\prime}|&  = (n-\ell-1)^{2}~(n-\ell-2) \left(\sum_{s=1}^{n-\ell -2}\binom{n-\ell-3}{s-1}~(n-2)^{n-s-3}\right)\\
& \hspace{9cm} (\text{By setting } s-1=j)\\
& = (n-\ell-1)^{2}~ (n-\ell-2)~(n-2)^{\ell -1}\left(\sum_{j=0}^{n-\ell -3}\binom{n-\ell-3}{j}~(n-2)^{n-\ell -3-j} \right)\\
& = (n-\ell-1)^{2}~(n-\ell -2)~(n-2)^{\ell -1}(n-1)^{n-\ell -3}.
\end{align*}
This concludes the proof.
\end{proof}

\begin{lemma}\label{lemma:|B"|}
For $n\geq 3$ and $1\leq \ell \leq n-2$, we have 
\[
|\mathbf{B}''|= (n-\ell-1)~(n-\ell-2)~(n-2)^{\ell-1}~(n-1)^{n-\ell -3}~\ell.
\]  
\end{lemma}

\begin{proof}
As a first step in computing the cardinality of $\mathbf{B}''$, for each $v\in F_{\ell}$ we define a subset $\mathcal{T}_v$ of $\Bn$.
The set $\mathcal{T}_v$ consists of all trees in $\Bn$ whose root lies in $\{2,3,\dots,n-\ell\}$ and is adjacent to $v$, where $v$ is not adjacent to any other vertex in $F_{\ell}$, and the vertex $1$ is a descendant of $v$ but not adjacent to it.
That is, for $v \in F_{\ell}$, let
%
\[
\mathcal{T}_{v}
=
\left\{\, T'' \in \Bn \;\colon\;
\begin{aligned}
&\root(T'') = r'' \in \{2,3,\ldots,n-\ell\},\ 
r'' \sim v,\\
& v \nsim w \text{ for all } w \in F_{\ell}\setminus\{v\}, 1 \text{ is a descendant of } v,
\text{ and } 1 \nsim v
\end{aligned}
\,\right\}.
\]

We recall that there exists a bijection $\Psi:\Un \to \Bn$.
We now define a map
\begin{equation}\label{map:B''}
\tau:\coprod_{v\in F_{\ell}} \mathcal{T}_v \longrightarrow \mathbf{B}'',
\end{equation}
and prove that $\tau$ is a bijection. 
Since $\mathbf{B}''\subseteq \Psi(\UnFl)$, we show that every $T''\in  \coprod_{v\in F_{\ell}}\mathcal{T}_{v}$ is mapped by $\tau$ to a unique tree
$T^{\prime}\in \mathbf{B}''$, where $T^{\prime}=\Psi(T)$ for some $T\in \UnFl$.

The construction of the map~\eqref{map:B''} is carried out sequentially.
Let $T'' \in \coprod_{v\in F_{\ell}} \mathcal{T}_{v}$.
If $T'' \in \mathbf{B}''$, then $\tau$ fixes $T''$, i.e., $\tau(T'') = T''$.
Equivalently, if $T'' = \Psi(T)$ for some $T \in \UnFl$, then we set $\tau(T'') = T''$.
Now suppose that $T'' \neq \Psi(T)$ for any $T \in \UnFl$.
Let $T'' \in \mathcal{T}_{v}$ with $\root(T'') = r''$.
We proceed through the following cases in order.
\begin{enumerate}
    \item \label{a} If $1 \nsim w$ for all $w \in F_{\ell} \setminus \{v\}$, then we move directly to case~\eqref{b}.
    Otherwise, suppose that $1$ is adjacent to a vertex $\rho \in F_{\ell} \setminus \{v\}$.
    Consider the unique path from $v$ to $1$ in the tree $T''$, and let $\widetilde{a}$ be the son of $v$ lying on this path.
    Then necessarily $\rho = \Par_{T''}(1)$.
    We obtain a tree $T''_{1}$ from $T''$ by interchanging the labels of the vertices $\widetilde{a}$ and $\rho$ (see Figure~\ref{fig:T1''}).
    If $T''_{1} \in \Psi(\UnFl)$, then we define $\tau(T'') = T''_{1}$.
    Otherwise, replace $T''$ by $T''_{1}$ (keeping the same root $r''$), and proceed to case~\eqref{b}.

\begin{figure}[ht]
\centering
\begin{tikzpicture}[scale=1]

\node[circle, fill=black, inner sep=1pt] (r) at (-1,0) {};
\node[above] at (r) {$r''$};

\node[draw, rectangle, minimum width=1cm, minimum height=0.5cm] (Fr) at (1,-1.5) {$\F_{r''}$};

\draw (r) -- (Fr);

\node[circle, fill=black, inner sep=1pt] (v) at (-3,-1.5) {};
\node[above] at (v) {$v$};

\draw (r) -- (v);

\node[draw, rectangle, minimum width=1cm, minimum height=0.5cm] (Fv) at (-4.5,-2) {$\F_{v}$};
\draw (v) -- (Fv);

\node[circle, fill=black, inner sep=1pt] (a) at (-3,-3 ) {};
\node[right] at (a) {$\widetilde{a}$};
\draw (v) -- (a);

\node[draw, rectangle, minimum width=1cm, minimum height=0.5cm] (Fa) at (-4.5,-4) {$\F_{\widetilde{a}}$};
\draw (a) -- (Fa);

\node[circle, fill=black, inner sep=1pt] (b) at (-3,-4 ) {};
\draw (a) -- (b);


\node[circle, fill=black, inner sep=1pt] (rho) at (-3,-5 ) {};
\node[right] at (rho) {$\rho$};

\node[circle, fill=black, inner sep=1pt] (1) at (-4,-6 ) {};
\node[left] at (1) {$1$};
\draw (rho) -- (1);

\node[draw, rectangle, minimum width=1cm, minimum height=0.5cm] (Frho) at (-2,-6) {$\F_{\rho}$};
\draw (rho) -- (Frho);

\draw[dashed] (b)--(rho);

\node at (3,-3) {$\longrightarrow$};

\node[right] at (-1,-7) {$T''$};

\node[right] at (8,-7) {$T_{1}''$};

\node[circle, fill=black, inner sep=1pt] (r1) at (8,0) {};
\node[above] at (r1) {$r''$};

\node[draw, rectangle, minimum width=1cm, minimum height=0.5cm] (Fr1) at (10,-1.5) {$\F_{r''}$};

\draw (r1) -- (Fr1);

\node[circle, fill=black, inner sep=1pt] (v1) at (6,-1.5) {};
\node[above] at (v1) {$v$};

\draw (r1) -- (v1);

\node[draw, rectangle, minimum width=1cm, minimum height=0.5cm] (Fv1) at (4.5,-2) {$\F_{v}$};
\draw (v1) -- (Fv1);

\node[circle, fill=black, inner sep=1pt] (a1) at (6,-3 ) {};
\node[right] at (a1) {$\rho$};
\draw (v1) -- (a1);

\node[draw, rectangle, minimum width=1cm, minimum height=0.5cm] (Fa1) at (4.6,-4) {$\F_{\widetilde{a}}$};
\draw (a1) -- (Fa1);

\node[circle, fill=black, inner sep=1pt] (b1) at (6,-4 ) {};
\draw (a1) -- (b1);

\node[circle, fill=black, inner sep=1pt] (rho1) at (6,-5 ) {};
\node[right] at (rho1) {$\widetilde{a}$};

\node[circle, fill=black, inner sep=1pt] (11) at (5,-6 ) {};
\node[left] at (11) {$1$};
\draw (rho1) -- (11);

\node[draw, rectangle, minimum width=1cm, minimum height=0.5cm] (Frho1) at (7,-6) {$\F_{\rho}$};
\draw (rho1) -- (Frho1);

\draw[dashed] (b1)--(rho1);

\end{tikzpicture}
\caption{Construction of $T_{1}''$ from the tree $T''$.} 
\label{fig:T1''}
\end{figure}
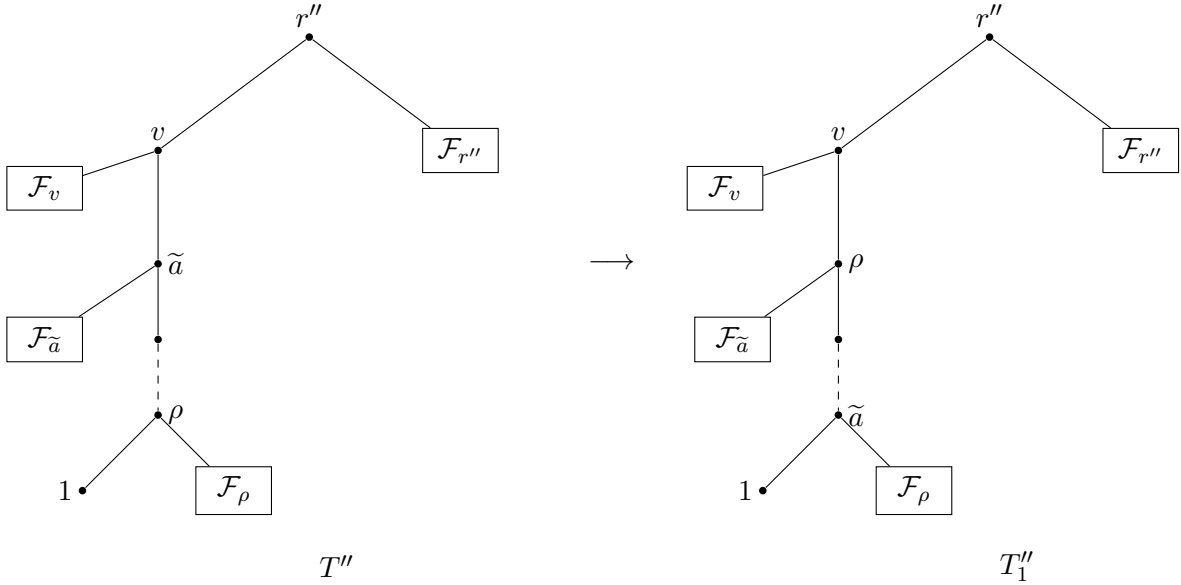

    \item \label{b} If $r'' \nsim w$ for all $w \in F_{\ell} \setminus \{v\}$, then proceed to case~\eqref{c}.
    Otherwise, suppose that $r''$ is adjacent to a non-empty set of vertices $\{w_1, w_2, \dots, w_m\} \subseteq F_{\ell} \setminus \{v\}$.
    For each $1 \le i \le m$, remove the vertex $w_i$ together with its descendant forest $\F_{w_i}$ from $r''$, and reattach $w_i$ (together with $\F_{w_i}$) as a son of the vertex $v$.
    We denote the resulting tree by $T''_{2}$ (see Figure~\ref{fig:T2''}).
    If $T''_{2} \in \Psi(\UnFl)$, then we define $\tau(T'') = T''_{2}$.
    Otherwise, replace $T''$ by $T''_{2}$ (keeping the same root $r''$), and proceed to case~\eqref{c}.

\begin{figure}[ht]
\centering
\begin{tikzpicture}[scale=0.9]

\node[circle, fill=black, inner sep=1pt] (r) at (-1,0) {};
\node[above] at (r) {$r''$};

\node[draw, rectangle, minimum width=1cm, minimum height=0.5cm] (Fr) at (1.5,-1.5) {$\F_{r''}$};

\draw (r) -- (Fr);

\node[circle, fill=black, inner sep=1pt] (v) at (-5,-1.5) {};
\node[left] at (v) {$v$};

\draw (r) -- (v);

\node[draw, rectangle, minimum width=1cm, minimum height=0.5cm] (Fv) at (-5,-3 ) {$\F_{v}$};
\draw (v) -- (Fv);

\node[circle, fill=black, inner sep=1pt] (w1) at (-3,-1.5) {};
\node[left] at (w1) {$w_1$};

\draw (r) -- (w1);

\node[draw, rectangle, minimum width=1cm, minimum height=0.5cm] (Fw1) at (-3,-3 ) {$\F_{w_1}$};
\draw (w1) -- (Fw1);

\node at (-1.5,-1.5) {$\cdots$};

\node[circle, fill=black, inner sep=1pt] (wm) at (-0,-1.5) {};
\node[left] at (wm) {$w_m$};

\draw (r) -- (wm);

\node[draw, rectangle, minimum width=1cm, minimum height=0.5cm] (Fwm) at (0,-3 ) {$\F_{w_m}$};
\draw (wm) -- (Fwm);

\node at (2.5,-3) {$\longrightarrow$};

\node[right] at (-2,-5) {$T''$};

\node[right] at (8,-5) {$T_{2}''$};

\node[circle, fill=black, inner sep=1pt] (r1) at (8,0) {};
\node[above] at (r1) {$r''$};

\node[draw, rectangle, minimum width=1cm, minimum height=0.5cm] (Fr1) at (10,-1.5) {$\F_{r''}$};

\draw (r1) -- (Fr1);

\node[circle, fill=black, inner sep=1pt] (v1) at (6,-1.5) {};
\node[above] at (v1) {$v$};

\draw (r1) -- (v1);

\node[draw, rectangle, minimum width=1cm, minimum height=0.5cm] (Fv1) at (4.5,-2) {$\F_{v}$};
\draw (v1) -- (Fv1);

\node[circle, fill=black, inner sep=1pt] (w11) at (5.5,-3 ) {};
\node[left] at (w11) {$w_1$};
\draw (v1) -- (w11);

\node[draw, rectangle, minimum width=1cm, minimum height=0.5cm] (Fw11) at (5.5,-4 ) {$\F_{w_1}$};
\draw (w11) -- (Fw11);

\node at (6.7,-3) {$\cdots$};

\node[circle, fill=black, inner sep=1pt] (wm1) at (8,-3 ) {};
\node[right] at (wm1) {$w_{m}$};
\draw (v1) -- (wm1);

\node[draw, rectangle, minimum width=1cm, minimum height=0.5cm] (Fwm1) at (8,-4) {$\F_{w_m}$};
\draw (wm1) -- (Fwm1);

\end{tikzpicture}
\caption{Construction of $T_{2}''$ from the tree $T''$.} 
\label{fig:T2''}
\end{figure}
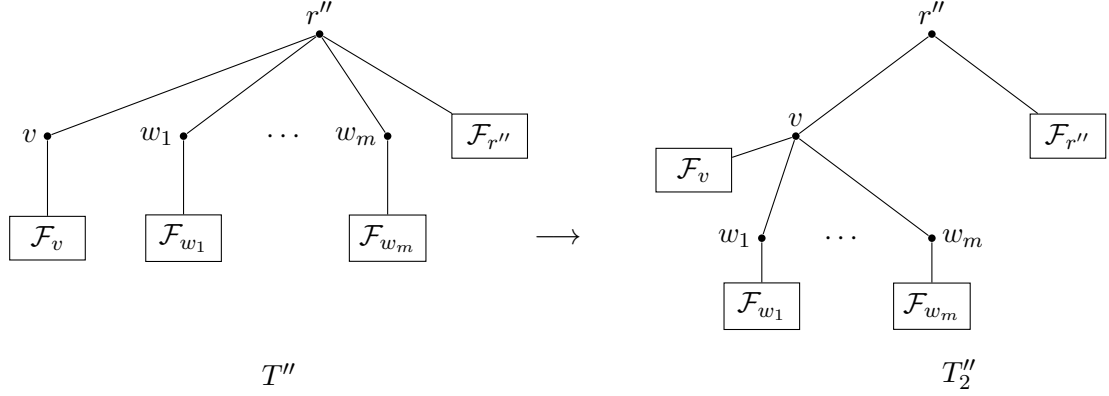

    \item \label{c} Consider the subtree of $T''$ obtained by deleting the vertex $v$ together with its descendant forest $\F_v$.
    Suppose that the resulting subtree contains an increasing path of the form
    \begin{equation*}\label{T''_case_c}
    r'' \longrightarrow c_1 \longrightarrow c_2 \longrightarrow \cdots \longrightarrow c_t \longrightarrow w,
    \end{equation*}
    where $c_j \in \{2,3,\dots,n-\ell\}\setminus\{r''\}$ for all $1 \le j \le t$, and
    $w \in F_{\ell}\setminus\{v\}$ satisfies $w < v$.
    Without loss of generality, choose such a path for which $w$ is minimal.
    We now construct a tree $T_{3}''$ from $T''$ as follows.
    Remove the vertices $w$ and $v$ together with their descendant forests $\F_w$ and $\F_v$, respectively.
    Then reattach $w$ (together with $\F_w$) as a son of the root $r''$ and $v$ (together with $\F_v$) as a son of the vertex $c_t$ (see Figure~\ref{fig:T3''}).
    Since $T_{3}'' \in \mathbf{B}'' \subseteq \Psi(\UnFl)$, we define $\tau(T'') = T_{3}''$.

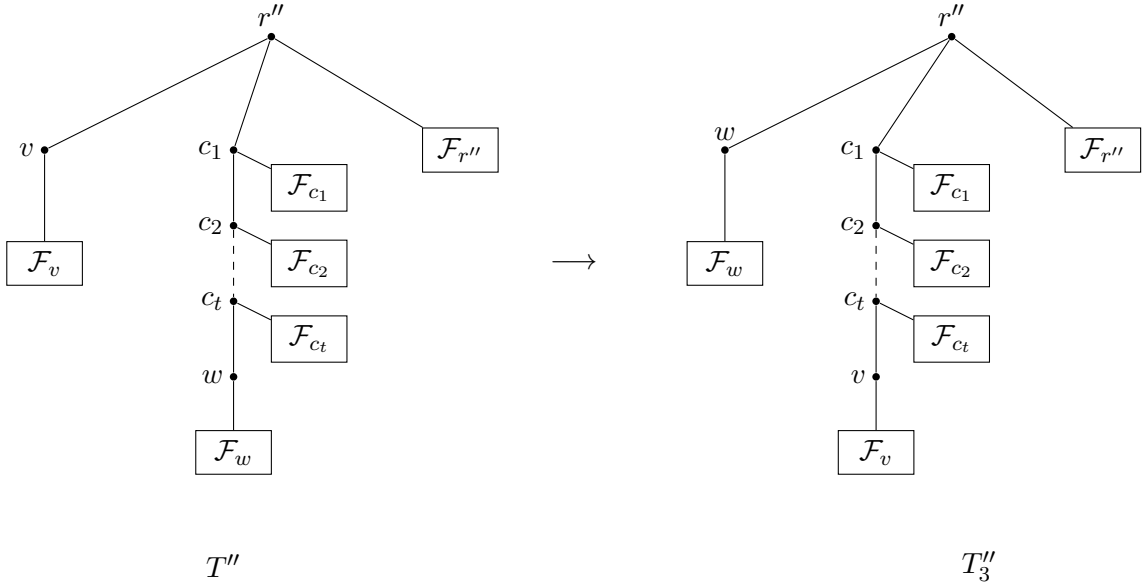
\begin{figure}[ht]
\centering
\begin{tikzpicture}[scale=1]

\node[circle, fill=black, inner sep=1pt] (r) at (-1,0) {};
\node[above] at (r) {$r''$};

\node[draw, rectangle, minimum width=1cm, minimum height=0.5cm] (Fr) at (1.5,-1.5) {$\F_{r''}$};

\draw (r) -- (Fr);

\node[circle, fill=black, inner sep=1pt] (v) at (-4,-1.5) {};
\node[left] at (v) {$v$};

\draw (r) -- (v);

\node[draw, rectangle, minimum width=1cm, minimum height=0.5cm] (Fv) at (-4,-3 ) {$\F_{v}$};
\draw (v) -- (Fv);

\node[circle, fill=black, inner sep=1pt] (c1) at (-1.5,-1.5) {};
\node[left] at (c1) {$c_1$};

\draw (r) -- (c1);

\node[draw, rectangle, minimum width=1cm, minimum height=0.5cm] (Fc1) at (-0.5,-2 ) {$\F_{c_1}$};
\draw (c1) -- (Fc1);

\node[circle, fill=black, inner sep=1pt] (c2) at (-1.5,-2.5) {};
\node[left] at (c2) {$c_2$};

\draw (c1) -- (c2);

\node[draw, rectangle, minimum width=1cm, minimum height=0.5cm] (Fc2) at (-0.5,-3 ) {$\F_{c_2}$};
\draw (c2) -- (Fc2);


\node[circle, fill=black, inner sep=1pt] (ct) at (-1.5,-3.5) {};
\node[left] at (ct) {$c_t$};

\node[draw, rectangle, minimum width=1cm, minimum height=0.5cm] (Fct) at (-0.5,-4 ) {$\F_{c_t}$};
\draw (ct) -- (Fct);

\node[circle, fill=black, inner sep=1pt] (w) at (-1.5,-4.5) {};
\node[left] at (w) {$w$};

\draw (ct) -- (w);

\node[draw, rectangle, minimum width=1cm, minimum height=0.5cm] (Fw) at (-1.5,-5.5 ) {$\F_{w}$};
\draw (w) -- (Fw);

\draw[dashed] (c2)--(ct);

\node at (3,-3) {$\longrightarrow$};

\node[right] at (-2,-7) {$T''$};


\node[circle, fill=black, inner sep=1pt] (r1) at (8,0) {};
\node[above] at (r1) {$r''$};

\node[draw, rectangle, minimum width=1cm, minimum height=0.5cm] (Fr1) at (10,-1.5) {$\F_{r''}$};

\draw (r1) -- (Fr1);

\node[circle, fill=black, inner sep=1pt] (w1) at (5,-1.5) {};
\node[above] at (w1) {$w$};

\draw (r1) -- (w1);

\node[draw, rectangle, minimum width=1cm, minimum height=0.5cm] (Fw1) at (5,-3) {$\F_{w}$};
\draw (w1) -- (Fw1);

\node[circle, fill=black, inner sep=1pt] (c11) at (7,-1.5) {};
\node[left] at (c11) {$c_1$};

\draw (r1) -- (c11);

\node[draw, rectangle, minimum width=1cm, minimum height=0.5cm] (Fc11) at (8,-2 ) {$\F_{c_1}$};
\draw (c11) -- (Fc11);

\node[circle, fill=black, inner sep=1pt] (c21) at (7,-2.5) {};
\node[left] at (c21) {$c_2$};

\draw (c11) -- (c21);

\node[draw, rectangle, minimum width=1cm, minimum height=0.5cm] (Fc21) at (8,-3 ) {$\F_{c_2}$};
\draw (c21) -- (Fc21);

\node[circle, fill=black, inner sep=1pt] (ct1) at (7,-3.5) {};
\node[left] at (ct1) {$c_t$};

\node[draw, rectangle, minimum width=1cm, minimum height=0.5cm] (Fct1) at (8,-4 ) {$\F_{c_t}$};
\draw (ct1) -- (Fct1);

\node[circle, fill=black, inner sep=1pt] (v1) at (7,-4.5) {};
\node[left] at (v1) {$v$};

\draw (ct1) -- (v1);

\node[draw, rectangle, minimum width=1cm, minimum height=0.5cm] (Fv1) at (7,-5.5 ) {$\F_{v}$};
\draw (v1) -- (Fv1);

\draw[dashed] (c21)--(ct1);

\node[right] at (8,-7) {$T_{3}''$};

\end{tikzpicture}
\caption{Construction of $T_{3}''$ from the tree $T''$.} 
\label{fig:T3''}
\end{figure}

\end{enumerate}

We now show that the above procedure is reversible.
Let $T^{\prime} \in \mathbf{B}''$ with root $r^{\prime}$.
We claim that there exists a tree 
$T'' \in \coprod_{v \in F_{\ell}} \mathcal{T}_{v}$ such that $\tau(T'') = T^{\prime}$.

Suppose that $r^{\prime}$ is adjacent to a vertex $v \in F_{\ell}$.
Then, by Proposition~\ref{prop:B"}, the root $r^{\prime}$ is not adjacent to any other vertex of $F_{\ell}$.
Starting from $T'$, we reconstruct a tree $T''$ with the required properties by considering the following cases in order.

\begin{enumerate}[label=(\roman*)]
    \item \label{1a}We first check whether $T'$ contains an increasing path of the form
    \begin{equation}\label{path:T'}
    r^{\prime} \longrightarrow d_1 \longrightarrow d_2 \longrightarrow \cdots \longrightarrow d_t \longrightarrow w,
    \end{equation}
    where $d_i \in \{2,3,\dots,n-\ell\}\setminus\{r^{\prime}\}$ for $1\leq i\leq t$, and 
    $w\in F_{\ell}\setminus\{v\}$ with $v<w$, such that the vertex $1$ is a descendant of $w$.
    If no such path exists in $T^{\prime}$, we proceed directly to case~\ref{2a}.  
    Otherwise, suppose that a path of the form \eqref{path:T'} exists.
    Then, we construct a tree $T^{\prime}_{1}$ from $T^{\prime}$ as follows.
    Remove the vertices $v$ and $w$ together with their descendant forests $\F_v$ and $\F_w$, respectively.
    Then reattach $v$ (together with $\F_v$) as a son of the vertex $d_t$, and $w$ (together with $\F_w$) as a son of the root $r^{\prime}$, obtaining a tree $T^{\prime}_{1}$.
    If $T^{\prime}_{1}\in\mathcal{T}_{w}$, then we set $T''=T^{\prime}_{1}$.  
    Otherwise, if $T^{\prime}_{1}\notin\coprod_{v\in F_{\ell}}\mathcal{T}_{v}$, replace $T^{\prime}$ by $T^{\prime}_1$ (keeping the same root $r'$) and proceed to case~\ref{2a}.

    \item \label{2a} We continue with the tree $T^{\prime}$ obtained after case~\ref{1a}.  
    In this tree, the root $r^{\prime}$ is adjacent to a vertex $v_0\in F_{\ell}$, where $v_0=v$ or $v_0=w$, according to whether an increasing path of the form~\eqref{path:T'} existed in the tree considered in case~\ref{1a}.
    Moreover, the vertex $1$ is a descendant of $v_0$.
    Consider the subtree of $T^{\prime}$ obtained by deleting the son of $v_0$ lying on the unique path from $v_0$ to the vertex $1$, together with all of its descendants.
    In the resulting subtree, if the vertex $v_0$ is not adjacent to any vertex in $F_{\ell}$, then proceed directly to case~\ref{3a}.
    Otherwise, suppose that $v_0$ is adjacent to a non-empty set of vertices $\{w_1,w_2,\dots,w_m\}\subseteq F_{\ell}\setminus\{v_0\}$.
    Then, in $T^{\prime}$, for each $1\leq i\leq m$, remove the vertex $w_i$ together with its descendant forest $\F_{w_i}$, and reattach $w_i$ (together with $\F_{w_i}$) as a son of the root $r^{\prime}$.
    Let us denote the resulting tree by $T^{\prime}_{2}$.
    If $T^{\prime}_{2}\in \coprod_{v\in F_{\ell}}\mathcal{T}_{v}$, then we set $T''=T^{\prime}_{2}$.
    Otherwise, replace $T^{\prime}$ by $T^{\prime}_{2}$ (keeping the same root $r'$) and proceed to case~\ref{3a}.

    \item\label{3a} Let $\rho$ be the son of $v_0$ lying on the unique path from $v_0$ to the vertex $1$.
    If $\rho\notin F_{\ell}$, then the tree $T^{\prime}$ already belongs to $\coprod_{v\in F_{\ell}}\mathcal{T}_{v},$ and we set $T''=T^{\prime}$.
    Otherwise, suppose that $\rho\in F_{\ell}$.
    Let $\widetilde{a}$ be the parent of the vertex $1$ in $T^{\prime}$.
    Clearly, $\widetilde{a}\notin F_{\ell}$.
    We obtain a tree $T_{3}^{\prime}$ from $T^{\prime}$ by interchanging the labels of the vertices $\widetilde{a}$ and $\rho$.
    The tree $T^{\prime}_{3}$ thus obtained belongs to $\coprod_{v\in F_{\ell}}\mathcal{T}_{v}$, and we set $T''=T_{3}^{\prime}$.
    
\end{enumerate}

Since all the steps in the construction of the map $\tau\colon \coprod_{v\in F_{\ell}} \mathcal{T}_{v} \longrightarrow \mathbf{B}''$ can be reversed systematically, the map $\tau$ is a bijection.
Note that, for any $v, v^{\prime}\in F_{\ell}$, the set $\mathcal{T}_{v^{\prime}}$ can be obtained from $\mathcal{T}_{v}$ by interchanging the labels of the vertices $v$ and $v^{\prime}$ in each tree of $\mathcal{T}_{v}$.
Hence, the cardinality of $\mathcal{T}_{v}$ is independent of the choice of $v\in F_{\ell}$.
In particular, $|\mathcal{T}_{v}|=|\mathcal{T}_{n}|$ for all $v\in F_{\ell}$.
Therefore, 
\begin{equation}\label{eq:B''}
|\mathbf{B}''| = \bigg |\coprod_{v\in F_{\ell}}\mathcal{T}_{v}\bigg| =\sum_{v\in F_{\ell}}|\mathcal{T}_{v}|=|F_{\ell}|~|\mathcal{T}_{n}| = \ell ~ |\mathcal{T}_{n}|.   
\end{equation}

We now proceed to compute the cardinality of the set $\mathcal{T}_{v}$.
Let $\overline{T}$ be a tree with vertex set $\{2,3,\dots,n\}$ having $\root(\overline{T})=v$, where $v\in F_{\ell}$ and $v$ is not adjacent to any vertex in $F_{\ell}$.
Let $\son_{\overline{T}}{(v)}=\{a_1 , a_2 , \dots , a_s\}$ be the set of sons of $v$ in $\overline{T}$, with $|\son_{\overline{T}}{(v)}|=s\geq 2$ and $\son_{\overline{T}}{(v)}\subseteq \{2,3, \dots , n-\ell\}$.
From $\overline{T}$, we construct a tree $\widetilde{T}$ with vertex set $[n]$ having $\root(\widetilde{T})=v$ by attaching the vertex $1$ as a leaf, making it a descendant of one of the vertices $a_i$ for some $1\le i\leq s$.
Note that, $\son_{\overline{T}}{(v)} = \son_{\widetilde{T}}{(v)}$.
If the leaf $1$ is a descendant of $a_{j}$ in $\widetilde{T}$, then we obtain tree a $T''\in \mathcal{T}_{v}$ from $\widetilde{T}$ by choosing any vertex in $\son_{\widetilde{T}}{(v)}\setminus\{a_j\}$ as the root (see Figure~\ref{fig:Ttilde}).
Thus, each tree $\widetilde{T}$ gives rise to exactly $s-1$ distinct trees in $\mathcal{T}_{v}$.
Conversely, every tree $T''\in \mathcal{T}_{v}$ yields a tree $\widetilde{T}$ by designating the vertex $v$ of $T''$ as the root.
If the vertex $v$ has $s$ sons in $T''$, then precisely $s-1$ trees in $\mathcal{T}_{v}$ correspond to the same tree $\widetilde{T}$.

\begin{figure}[ht]
\centering
\begin{tikzpicture}[scale=0.8,
    every node/.style={font=\small}]

\node[circle, fill=black, inner sep=1pt] (v) at (-1,0) {};
\node[above] at (v) {$v$};

\node[circle, fill=black, inner sep=1pt] (a1) at (-5,-1.5) {};
\node[left] at (a1) {$a_1$};

\draw (v) -- (a1);

\node[draw, rectangle, minimum width=1cm, minimum height=0.5cm] (Fa1) at (-5,-3 ) {$\F_{a_1}$};
\draw (a1) -- (Fa1);

\node[circle, fill=black, inner sep=1pt] (a2) at (-3,-1.5) {};
\node[left] at (a2) {$a_2$};

\draw (v) -- (a2);

\node[draw, rectangle, minimum width=1cm, minimum height=0.5cm] (Fa2) at (-3,-3 ) {$\F_{a_2}$};
\draw (a2) -- (Fa2);

\node at (-2.2,-1.5) {$\cdots$};

\node[circle, fill=black, inner sep=1pt] (aj) at (-1,-1.5) {};
\node[left] at (aj) {$a_j$};

\draw (v) -- (aj);


\node[circle, fill=black, inner sep=1pt] (par1) at (-1,-2.4) {};
\node[right] at (par1) {$a^{\prime}$};

\node[circle, fill=black, inner sep=1pt] (1) at (-1.5,-3.5) {};
\node[left] at (1) {$1$};

\draw (par1) -- (1);

\node[draw, rectangle, minimum width=1cm, minimum height=0.5cm] (Fa') at (-1,-4.5) {$\F_{a^{\prime}}$};
\draw (par1) -- (Fa');

\node at (0,-1.5) {$\cdots$};

\node[circle, fill=black, inner sep=1pt] (as) at (1,-1.5) {};
\node[right] at (as) {$a_s$};

\draw (v) -- (as);

\node[draw, rectangle, minimum width=1cm, minimum height=0.5cm] (Fas) at (1,-3) {$\F_{a_s}$};

\draw (as) -- (Fas);

\draw[dashed] (aj)--(par1);

\node[right, align=left] at (-6,-7.2) {%
$\widetilde{T}$ with $\root(\widetilde{T})=v$ and \\
$\son_{\widetilde{T}}(v)=\{a_1,a_2,\dots,a_s\}$, where \\
the leaf $1$ is a descendant of $a_j$.};

\node at (2.5,-3) {$\longrightarrow$};


\node[circle, fill=black, inner sep=1pt] (ak) at (8,0) {};
\node[above] at (ak) {$a_k$};

\node[draw, rectangle, minimum width=1cm, minimum height=0.5cm] (Fak) at (11,-1.3) {$\F_{a_k}$};

\draw (ak) -- (Fak);

\node[circle, fill=black, inner sep=1pt] (v1) at (6,-1.5) {};
\node[above] at (v1) {$v$};

\draw (ak) -- (v1);

\node[circle, fill=black, inner sep=1pt] (ai1) at (3.4,-2.5) {};
\node[left] at (ai1) {$a_{1}$};

\draw (v1) -- (ai1);

\node[draw, rectangle, minimum width=1cm, minimum height=0.5cm] (Fai1) at (3.4,-4 ) {$\F_{a_{1}}$};
\draw (ai1) -- (Fai1);

\node at (4.1,-2.5) {$\cdots$};

\node[circle, fill=black, inner sep=1pt] (ai2) at (5.6,-2.5 ) {};
\node[left] at (ai2) {$a_{k-1}$};
\draw (v1) -- (ai2);

\node[draw, rectangle, minimum width=1cm, minimum height=0.5cm] (Fai2) at (5.6,-4 ) {$\F_{a_{k-1}}$};
\draw (ai2) -- (Fai2);

\node[circle, fill=black, inner sep=1pt] (ak1) at (7.5,-2.5) {};
\node[left] at (ak1) {$a_{k+1}$};

\draw (v1) -- (ak1);

\node[draw, rectangle, minimum width=1cm, minimum height=0.5cm] (Fak1) at (7.5,-4 ) {$\F_{a_{k+1}}$};
\draw (ak1) -- (Fak1);

\node at (8.2,-2.5) {$\cdots$};

\node[circle, fill=black, inner sep=1pt] (ait) at (9.2,-2.5) {};
\node[left] at (ait) {$a_{j}$};

\draw (v1) -- (ait);

\node[circle, fill=black, inner sep=1pt] (pari1) at (9.2,-3.3) {};
\node[right] at (pari1) {$a^{\prime}$};

\node[circle, fill=black, inner sep=1pt] (i1) at (8.7,-4.3) {};
\node[left] at (i1) {$1$};

\draw (pari1) -- (i1);

\node[draw, rectangle, minimum width=1cm, minimum height=0.5cm] (Fia') at (9.2,-5.5) {$\F_{a^{\prime}}$};
\draw (pari1) -- (Fia');

\node at (9.8,-2.5) {$\cdots$};

\node[circle, fill=black, inner sep=1pt] (ais-1) at (10.5,-2.5 ) {};
\node[right] at (ais-1) {$a_{s}$};
\draw (v1) -- (ais-1);

\node[draw, rectangle, minimum width=1cm, minimum height=0.5cm] (Fais-1) at (10.5,-4) {$\F_{a_{s}}$};
\draw (ais-1) -- (Fais-1);

\draw[dashed] (ait)--(pari1);

\node[right, align=left] at (4.5,-7.5) {%
$T''$ with $\root(T'') = a_k$, where \\
$k\in[s]\setminus\{j\}$.};

\end{tikzpicture}
\caption{Construction of $T''$ from the tree $\widetilde{T}$, where $\root{(T'')}\in \son_{\widetilde{T}}{(v)}\setminus\{a_j\}$.}
\label{fig:Ttilde}
\end{figure}

Therefore, by Theorem~\ref{Thm:forest}, the number of trees $\overline{T}$ with $\son_{\overline{T}}{(v)}=\{a_1 , a_2 , \dots , a_s\}$, as described above, is $s(n-2)^{n-3-s}$.
The number of ways to choose a subset $\son_{\overline{T}}{(v)}\subseteq \{2,3, \dots , n-\ell\}$ with $s$ elements is $\binom{n-\ell -1 }{s}$.
For a fixed tree $\overline{T}$, we construct a tree $\widetilde{T}$ by attaching the vertex $1$ as a leaf.
Since $1$ may be attached to any vertex in $\{2,3,\dots,n\}\setminus\{v\}$, there are exactly $n-2$ possible choices for this attachment.
Further, as each $\widetilde{T}$ yields $s-1$ distinct trees in
$\mathcal{T}_v$, we obtain
\[
|\mathcal{T}_{v}| = \sum_{s=2}^{n-\ell -1} \binom{n-\ell -1}{s}~s~(n-2)^{n-s-3}~(s-1)~(n-2).
\]
Since $\binom{n-\ell -1}{s}~s~(s-1)= (n-\ell -1)~(n-\ell -2)~\binom{n-\ell -3}{s-2}$, we have
\begin{equation*}
\begin{split}
|\mathcal{T}_{v}| & = (n-\ell -1)~(n-\ell -2)~(n-2) \left(\sum_{s=2}^{n-\ell -1} \binom{n-\ell -3}{s-2}~(n-2)^{n-s-3} \right)\\
& \hspace{9cm} (\text{By setting } s-2 =j)\\
&  =  (n-\ell -1)~(n-\ell -2)~(n-2)\left(\sum_{j=0}^{n-\ell -3} \binom{n-\ell -3}{j}~(n-2)^{n-\ell - 3-j}~(n-2)^{\ell -2}\right)\\
&  =  (n-\ell -1)~(n-\ell -2)~(n-2)^{\ell-1}~(n-1)^{n-\ell -3}.
\end{split}
\end{equation*}

Thus, for all $v\in F_{\ell}$, we have
\begin{equation}\label{eq:|Tv|}
    |\mathcal{T}_{v}| =  (n-\ell -1)~(n-\ell -2)~(n-2)^{\ell-1}~(n-1)^{n-\ell -3}.
\end{equation}
Combining \eqref{eq:|Tv|} with \eqref{eq:B''}, we obtain 
\[
|\mathbf{B}''| = \ell~|\mathcal{T}_{v}| = (n-\ell -1)~(n-\ell -2)~(n-2)^{\ell-1}~(n-1)^{n-\ell -3}~\ell.
\]
This completes the proof of the lemma.
\end{proof}

We are now ready to prove Theorem~\ref{thm:|Un|}.

\begin{proof}[\textbf{Proof of Theorem~\ref{thm:|Un|}}]
By Remark~\ref{rmrk:Psi}, we have 
\[
|\UnFl| = |\Psi (\UnFl)| = |\mathbf{B}| ~ + ~ |\mathbf{B}^{\prime}| ~ + ~ |\mathbf{B}''|.
\]
Hence, using Lemma \ref{lemma:|B|}, \ref{lemma:|B'|} and \ref{lemma:|B"|}, we obtain
\begin{equation*}
\begin{split}
|\UnFl|  & =(n-\ell - 1)~(n-1)^{n-\ell -2}~(n-2)^{\ell -1}~\ell \\
& \hspace{0.5cm} + (n-\ell-1)^{2}~(n-\ell -2)~(n-2)^{\ell -1}(n-1)^{n-\ell -3} \\
& \hspace{0.5cm}  + (n-\ell -1)~(n-\ell -2)~(n-2)^{\ell-1}~(n-1)^{n-\ell -3}~\ell \\
&  =  (n-\ell - 1)~(n-1)^{n-\ell -3}~(n-2)^{\ell -1} \Bigl( (n-1)~\ell + (n-\ell - 2)~(n-1) \Bigr)\\
&  =  (n-1)^{n-\ell -2}~(n-2)^{\ell}~(n-\ell-1).
\end{split}
\end{equation*}
This completes the proof.
\end{proof}

%% file: sec03.tex
\section{\texorpdfstring{Enumeration of $\UnFl$ via the matrix tree theorem}{Counting Uprooted Trees Using MTT}}\label{sec:Enumeration by MTT}

For each $0\le k\le n-2$, let $\Gl^{n-k}$ be the rooted subgraph of $\Gl^{\prime}$ with root $n-k$, obtained from $\Gl^{\prime}$  by deleting the edges $e_{n-k,i}\in E(\Gl^{\prime})$ for all $n-k+1\leq i\leq n$.
That is, for $0\leq k\leq n-2$, we define
\[
\Gl^{n-k} \coloneqq \Gl^{\prime}-\{ e_{n-k,i} \ \colon n-k+1 \leq i \leq n\}.
\]
Clearly, any spanning tree of $\Gl^{n-k}$ is an uprooted spanning tree of $\Gl^{\prime}$ with root $n-k$.
Further, recall that $\mathcal{U}(\Gl^{\prime}) = \UnFl$.
Therefore, we can express 
$\UnFl
= \coprod_{k=0}^{n-2}\SPT{(\Gl^{n-k})}$,
where $\SPT{(\Gl^{n-k})}$ denotes the set of all spanning trees of $\Gl^{n-k}$.
Hence, we have
\[
\left|\UnFl\right|
= \sum_{k=0}^{n-2}\left|\SPT{(\Gl^{n-k})}\right|.
\]
Further, we can split this sum as follows.
\begin{equation}\label{eq:Uprooted sum MTT}
\left|\UnFl\right|
=
\left|\SPT{(\Gl^{n})}\right|
+
\sum_{k=1}^{\ell-1}\left|\SPT{(\Gl^{n-k})}\right|
+
\sum_{k=\ell}^{n-2}\left|\SPT{(\Gl^{n-k})}\right|.
\end{equation}

By the matrix tree theorem~\cite{Kirchhoff}, we know that the cardinality of the set of all spanning trees of $\Gl^{n-k}$ is given by the determinant of the reduced Laplacian matrix $\widetilde{L}(\Gl^{n-k})$ of $\Gl^{n-k}$, i.e.,
\begin{equation}\label{eq:ST(G)_MTT}
|\SPT{(\Gl^{n-k})}|= \det \bigl( \widetilde{L}(\Gl^{n-k}) \bigr), \quad  0\leq k \leq n-2.
\end{equation}

For the graph $\Gl^{n}$ with root $n$, we obtain the reduced Laplacian matrix $\widetilde{L}(\Gl^{n})$ by deleting the $n$-th row and the $n$-th column of the Laplacian matrix $L(\Gl^{n})$.
The resulting matrix $\widetilde{L}(\Gl^{n})$ is an $(n-1)\times (n-1)$ matrix in which the first diagonal entry is $n-\ell-1$, the $i$-th diagonal entry equals $n-1$ for $2\le i\le n-\ell$ and equals $n-2$ for $n-\ell+1\le i\le n-1$.
Moreover, all non-diagonal entries are $-1$ except for the $(1,j)$-th and $(j,1)$-th entries, where $n-\ell+1\le j\le n-1$, which are $0$.
Equivalently, $\widetilde{L}(\Gl^{n})$ is given by
\begin{equation*}
\widetilde{L}(\Gl^{n}) =
\begin{blockarray}{*{7}{c}r}
 & &  &  \substack{C_{n-\ell} \\ \downarrow} &  &  &  & \\
\begin{block}{[*{7}{c}]r}
 n-\ell-1 & -1 & \cdots & -1 & 0 &\cdots  & 0  & \\
 -1 & n-1 & \cdots & -1   & -1  & \cdots & -1  & \\
 \vdots & \vdots  & \ddots & \vdots & \vdots  &  \ddots & \vdots  & \\
 -1  & -1   & \cdots & n-1  & -1   &  \cdots & -1  & \leftarrow \substack{R_{n-\ell} \\ }  \\
 0  & -1   & \cdots & -1  & n-2   & \cdots& -1 & \\
 \vdots & \vdots    &\ddots  &\vdots & \vdots & \ddots &\vdots & \\
 0 & -1  & \cdots & -1  & -1  & \cdots & n-2 & {\scriptstyle (n-1) \times (n-1)}& \\
\end{block}
\end{blockarray}.
\end{equation*}

Let $R_{i}$ denote the $i$-th row and $C_{j}$ denote the $j$-th column of the matrix $\widetilde{L}(\Gl^{n})$.
Let us perform the following series of row and column operations on $\widetilde{L}(\Gl^{n})$.
\begin{enumerate}
    \item $C_1 \longrightarrow C_1 + \sum_{i=2}^{n-2}C_{i}$,
    \item $C_i \longrightarrow C_i + C_{1}$, for all $ 2\leq i\leq  n-1$,
    \item $R_{1} \longrightarrow R_{1} + \sum_{i=2}^{n-\ell}\frac{1}{n} R_i$.
\end{enumerate}
After applying these operations in the given order, the matrix $\widetilde{L}(\Gl^{n})$ is reduced to a lower triangular matrix whose first diagonal entry is $\frac{n-\ell-1}{n}$, the $i$-th diagonal entry is $n$ for $2\leq i \leq n-\ell$ and $n-1$ for $n-\ell +1\leq i \leq n-1$.
Since the determinant of a lower triangular matrix is the product of its diagonal entries, and all the above operations preserve the determinant, we have
\[\det \bigl(\widetilde{L}(\Gl^{n}) \bigr)=n^{n-\ell-1}~(n-1)^{\ell-1}~\left(\frac{n-\ell-1}{n}\right)=n^{n-\ell-2}~(n-1)^{\ell-1}~(n-\ell-1).\]
By the matrix tree theorem, we have
\begin{equation}\label{eq:MTT_root_n}
     |\SPT{(\Gl^{n})}| = \det \bigl(\widetilde{L}(\Gl^{n}) \bigr)=n^{n-\ell-2}~(n-1)^{\ell-1}~(n-\ell-1).
\end{equation}

Now consider the graph $\Gl^{n-k}$ with root $n-k$, where $1\leq k \leq \ell -1$.
By construction, in $\Gl^{n-k}$, the vertex $1$ is not adjacent to any vertex in $F_{\ell}$, and the root $n-k$ is not adjacent to any vertex $i$ with $ n-k+1\leq i\leq n$.
The reduced Laplacian matrix of $\Gl^{n-k}$ is obtained by deleting the $(n-k)$-th row and column of the Laplacian matrix $L(\Gl^{n-k})$.
The resulting reduced Laplacian matrix $\widetilde{L}(\Gl^{n-k})$ is an $(n-1)\times(n-1)$ matrix in which the first diagonal entry is $n-\ell-1$,
the $i$-th diagonal entry equals $n-1$ for $2 \leq i \leq n-\ell$, equals $n-2$ for $n-\ell+1 \leq i \leq n-k-1$, and equals $n-3$ for $n-k \leq i \leq n-1$.
Moreover, all non-diagonal entries are $-1$ except for the $(1,j)$-th and $(j,1)$-th entries, where $n-\ell +1\leq j \leq n-1$, which are $0$.

Let $R_{i}$ and $C_{j}$ denote the $i$-th row and $j$-th column of the matrix $\widetilde{L}(\Gl^{n-k})$, respectively.
We now perform the following sequence of row and column operations on $\widetilde{L}(\Gl^{n-k})$.
\begin{enumerate}
    \item $C_1 \longrightarrow C_1 + \sum_{i=2}^{n-1}C_{i}$,
    \item $C_{i} \longrightarrow C_{i} + C_{1}$, for all $2\leq i\leq  n-1$,
    \item $R_{i} \longrightarrow R_{i} - R_{1}$, for all $n-k\leq i\leq  n-1$,
    \item $R_1 \longrightarrow R_1 + \sum_{i=2}^{n-\ell}\frac{1}{n} R_i$, 
    \item $C_{n-k} \longrightarrow C_{n-k} + \sum_{i=n-k+1}^{n-1}C_{i}$, 
    \item $C_{i} \longrightarrow C_i + \frac{1}{n-k-2} C_{n-k}$, for all $ n-k+1\leq i\leq  n-1$.
\end{enumerate}
After applying the above operations in the given order, the matrix $\widetilde{L}(\Gl^{n-k})$ reduces to a lower triangular matrix whose first diagonal entry is $\frac{n-\ell-1}{n}$, 
the $(n-k)$-th diagonal entry is $n-k-2$,
the $i$-th diagonal entry equals $n$ for $2 \leq i \leq n-\ell$,
equals $n-1$ for $n-\ell +1 \leq i \leq n-k-1$,
and equals $n-2$ for $n-k+1 \leq i \leq n-1$.
Since the determinant of a lower triangular matrix is the product of its diagonal entries, and all the above operations preserve the determinant, we obtain
\[\det \bigl(\widetilde{L}(\Gl^{n-k})\bigr)=n^{n-\ell-2}~(n-1)^{\ell-k-1}~(n-2)^{k-1}~(n-k-2)~(n-\ell-1).\]
By the matrix tree theorem, for $1\leq k\leq \ell -1$, we have
\begin{equation}\label{eq:MTT_root_n-k (k<l)}
     |\SPT(\Gl^{n-k})| =n^{n-\ell-2}~(n-1)^{\ell-k-1}~(n-2)^{k-1}~(n-k-2)~(n-\ell-1).
\end{equation}

For $\ell \leq k\leq n-2$, consider the graph $\Gl^{n-k}$ with root $n-k$.
Let $\widetilde{L}(\Gl^{n-k})$ be the reduced Laplacian matrix of $\Gl^{n-k}$ obtained by deleting the $(n-k)$-th row and column of the Laplacian matrix $L(\Gl^{n-k})$. 
Then $\widetilde{L}(\Gl^{n-k})$ is an $(n-1)\times (n-1)$ matrix whose diagonal entries are given as follows:
\[
d_{ii}=
\begin{cases}
n-\ell-1, & i=1,\\
n-1, & 2 \leq i \leq n-k-1,\\
n-2, & n-k \leq i \leq n-\ell-1,\\
n-3, & n-\ell \leq i \leq n-1.
\end{cases}
\]
All non-diagonal entries are equal to $-1$, except the $(1,j)$-th and $(j,1)$-th entries, where $n-\ell \leq j \leq n-1$, which are equal to $0$.

Let $R_i$ and $C_j$ denote the $i$-th row and $j$-th column of $\widetilde{L}(\Gl^{n-k})$, respectively. 
We now perform the following column operations on $\widetilde{L}(\Gl^{n-k})$.
\begin{enumerate}
    \item $C_1 \longrightarrow C_1 + \sum_{i=2}^{n-1} C_{i}$,
    \item $C_i \longrightarrow C_i + C_{1}$, for all $ 2\leq i\leq  n-\ell-1$.
\end{enumerate}
Applying these operations in the stated order transforms the matrix $\widetilde{L}(\Gl^{n-k})$ into the $(n-1)\times(n-1)$ matrix $\widetilde{L}(\Gl^{n-k})'$ given by
\[
\scalebox{0.84}{$
\widetilde{L}(\Gl^{n-k})'=
\begin{blockarray}{cccccccccccccccc}
& & & & & \substack{C_{n-k-1}\\ \downarrow}
& & & &   \substack{C_{n-\ell-1}\\ \downarrow}
& & & & & \substack{C_{n-1}\\ \downarrow} \\
\begin{block}{[ccccccccccccccc]c}
& 1& 0 & 0 & \cdots  & 0 & 0 & 0 & \cdots  & 0 & 0 & 0 & \cdots & 0 & 0 & \\
& 1& n & 0 & \cdots  & 0 & 0 & 0 & \cdots  & 0 & -1 & -1 & \cdots & -1 & -1 & \\
& 1 & 0 & n & \cdots  & 0 & 0 & 0 & \cdots & 0 & -1  & -1 & \cdots & -1 & -1 & \\
& \vdots & \vdots & \vdots & \ddots & \vdots & \vdots & \vdots  & \ddots & \vdots & \vdots & \vdots & \ddots & \vdots & \vdots & \\
& 1 & 0 & 0 & \cdots & n  & 0 & 0 & \cdots & 0 & -1  & -1 & \cdots & -1 & -1 &  \longleftarrow  \substack{R_{n-k-1}} \\
& 0 & -1 & -1 & \cdots  & -1 & n-2 & -1 & \cdots & -1 & -1  & -1 & \cdots & -1 & -1 & \\
& 0 & -1 & -1 & \cdots  & -1 & -1 & n-2 & \cdots & -1  & -1 & -1 & \cdots & -1 & -1 & \\
& \vdots & \vdots & \vdots & \ddots & \vdots & \vdots & \vdots & \ddots & \vdots & \vdots & \vdots & \ddots & \vdots & \vdots & \\
& 0 & -1 & -1 & \cdots & -1 & -1 & -1 & \cdots  & n-2 & -1 & -1 & \cdots & -1 & -1 &  \longleftarrow  \substack{R_{n-\ell-1}} \\
& 0 & -1 & -1 & \cdots & -1 & -1 & -1 & \cdots & -1 & n-3 & -1 & \cdots & -1 & -1 & \\
& 0 & -1 & -1 & \cdots & -1 & -1 & -1 & \cdots & -1 & -1 & n-3 & \cdots & -1 & -1 & \\
& \vdots & \vdots & \vdots & \ddots & \vdots & \vdots & \vdots & \ddots & \vdots & \vdots & \vdots & \ddots & \vdots & \vdots & \\
& 0 & -1 & -1 & \cdots & -1 & -1 & -1 & -1 & \cdots & -1 & -1 & \cdots & n-3 & -1 & \\
& 0 & -1 & -1 & \cdots & -1 & -1 & -1 & -1 & \cdots & -1 & -1 & \cdots & -1 & n-3 & \longleftarrow  \substack{R_{n-1}} \\
\end{block}
\end{blockarray}
$}
\]
Since all the above column operations preserve the determinant, we have $\det\bigl(\widetilde{L}(\Gl^{n-k})\bigr)=\det\bigl(\widetilde{L}(\Gl^{n-k})'\bigr)$.
We compute the determinant of the matrix $\widetilde{L}(\Gl^{n-k})'$ by expanding along the first row, whose only nonzero entry is the first entry, equal to $1$. Hence, 
$$\det \bigl(\widetilde{L}(\Gl^{n-k})'\bigr) = \det\bigl(\widetilde{L}(\Gl^{n-k})''\bigr),$$
where $\widetilde{L}(\Gl^{n-k})''$ is the $(n-2)\times (n-2)$ matrix obtained from $\widetilde{L}(\Gl^{n-k})'$ by deleting its first row and first column. 
Applying the row operations $R_i \longrightarrow R_i - R_{n-2}$, which preserve the determinant, for all $n-k-1\leq i\leq  n-3$,  transforms the matrix $\widetilde{L}(\Gl^{n-k})''$ into the block matrix 
\[
\widetilde{\mathbf{L}} = 
\begin{bNiceArray}{c|c}[margin,columns-width=auto]
  A & B \\
  \hline
 C & D
\end{bNiceArray}_{(n-2)\times(n-2)},
\]  
where $A$ is an $(n-\ell-2)\times (n-\ell-2)$ diagonal matrix whose $i$-th diagonal entry equals $n$ for $1\leq i \leq n-k-2$, and equals $n-1$ for $n-k-1\leq i\leq n-\ell-2$.
Further, $B=[b_{ij}]$ is an $(n-\ell-2)\times \ell$ matrix with entries
\[
b_{ij}=
\begin{cases}
-1, & 1 \leq i \leq n-k-2,\ 1 \leq j \leq \ell,\\
0, & n-k-1 \leq i \leq n-\ell-2,\ 1 \leq j \leq \ell-1,\\
-(n-2), & n-k-1 \leq i \leq n-\ell-2,\ j=\ell.
\end{cases}
\]
Moreover, $C$ is an $\ell\times (n-\ell-2)$ matrix whose entries are all $0$, except those in the last row, which are all equal to $-1$. 
Finally, $D$ is an $\ell\times \ell$ matrix whose diagonal entries are all equal to $n-2$, except the last diagonal entry, which is equal to $n-3$, and whose non-diagonal entries are all $0$, except the $(i,\ell)$-th entries, which are equal to $-(n-2)$ for $1 \leq i \leq \ell-1$, and the $(\ell,j)$-th entries, which are equal to $-1$ for $1 \leq j \leq \ell-1$.

As $A$ is invertible, using the Schur's formula for determinants of block matrices, we have
\[
\det\bigl(\widetilde{L}(\Gl^{n-k})\bigr)=\det\bigl(\widetilde{\mathbf{L}}\bigr)= \det (A)\cdot \det (D-CA^{-1}B).
\]
As $A$ is diagonal, we obtain $\det(A) = n^{n-k-2}\, (n-1)^{k-\ell} $, and the matrix $D-CA^{-1}B$ is given by
\begin{equation*}
D-CA^{-1}B=
\begin{blockarray}{*{6}{c}r}
\begin{block}{[*{6}{c}]r}
n-2 & 0 & 0 & \cdots & 0 & -(n-2) & \\
0 & n-2 & 0 & \ddots & 0 & -(n-2) & \\
0 & 0 & n-2 & \cdots & 0 & -(n-2) & \\
\vdots & \vdots & \vdots & \ddots & \vdots & \vdots & \\
0 & 0 & 0 & \cdots & n-2 & -(n-2) & \\
\frac{-2n+k+2}{n} & \frac{-2n+k+2}{n} & \frac{-2n+k+2}{n} & \cdots & \frac{-2n+k+2}{n} & n-3-\lambda & {\scriptstyle \ell\times\ell} \\
\end{block}
\end{blockarray},
\end{equation*}
where $\lambda= \frac{n-k-2}{n} + \frac{(k-\ell)(n-2)}{n-1}$.
Now apply the column operation $C_{\ell} \longrightarrow C_{\ell} +\sum_{i=1}^{\ell -1} C_{i}$.
Then the matrix $D-CA^{-1}B$ is transformed into a lower triangular matrix whose diagonal entries are all equal to $n-2$, except the last diagonal entry, which is equal to $(n-k-\ell-2)+\frac{k-\ell}{n-1}+\frac{\ell(k+2)}{n}$.
Since the above column operation preserves the determinant, and the determinant of a triangular matrix is the product of its diagonal entries, it follows that 
\[
\det(D-CA^{-1}B)= (n-2)^{\ell- 1}~ \bigg((n-k-\ell-2)+\frac{k-\ell}{n-1}+\frac{\ell(k+2)}{n}\bigg).
\]
Hence, we obtain
\[
\det\bigl(\widetilde{L}(\Gl^{n-k})\bigr)= n^{n-k-2}~ (n-1)^{k-\ell}~ (n-2)^{\ell- 1}~ \bigg((n-k-\ell-2)+\frac{k-\ell}{n-1}+\frac{\ell(k+2)}{n}\bigg).
\]
For $\ell\leq k\leq n-2$, again by the matrix tree theorem, we have 
\begin{equation}\label{eq:MTT_root_n-k (k>l)}
     |\SPT(\Gl^{n-k})| = n^{n-k-2}~ (n-1)^{k-\ell}~(n-2)^{\ell- 1}~ \bigg((n-k-\ell-2)+\frac{k-\ell}{n-1}+\frac{\ell~(k+2)}{n}\bigg).
\end{equation}

Now, from \eqref{eq:Uprooted sum MTT}, \eqref{eq:MTT_root_n}, \eqref{eq:MTT_root_n-k (k<l)} and \eqref{eq:MTT_root_n-k (k>l)}, we obtain
\begin{equation*}
\begin{split}
|\UnFl|  & =n^{n-\ell -2}~(n-1)^{\ell -1}~(n-\ell - 1) \\
& \hspace{0.5cm} + \sum_{k=1}^{\ell -1}n^{n-\ell-2}~(n-1)^{\ell -k -1}~(n-2)^{k-1}~(n-k-2)~(n-\ell-1)\\
& \hspace{0.5cm}  + \sum_{k=\ell}^{n -2}n^{n-k-2}~(n-1)^{k-\ell}~(n-2)^{\ell -1}\left ((n-k-\ell-2)+\frac{k-\ell}{n-1}+\frac{\ell~(k+2)}{n} \right ).
\end{split}
\end{equation*}
This expression can be rewritten as
\begin{equation}\label{eq:Un_MTTsum}
\begin{split}
|\UnFl|  & = n^{n-\ell-2}~(n-\ell-1)\sum_{k=0}^{\ell -1}(n-1)^{\ell -k -1}~(n-2)^{k-1}~(n-k-2) \\
& \hspace{0.2cm}  +(n-2)^{\ell -1} \sum_{k=\ell}^{n -2}n^{n-k-2}~(n-1)^{k-\ell}\left((n-k-\ell-2) + \frac{(\ell+1)~nk + \ell~(n-k-2)}{n~(n-1)}\right).
\end{split}
\end{equation}

We now present a few propositions that will enable us to combine the preceding results and determine the cardinality of the set $\UnFl$ as an application of the matrix tree theorem.

\begin{proposition}\label{prop:(n-2)}
Let $\ell \geq 1$. Then, for all $n \geq \ell+2$, we have
\begin{equation*}
\begin{split}
\ell\left(\frac{n}{n-1}\right)^{n-\ell-2} + \sum_{k=\ell}^{n-2}\left(\frac{n}{n-1}\right)^{n-k-2} \left(\frac{n-k-\ell-2}{n-\ell-1}+\frac{(\ell+1)~nk~+\ell~(n-k-2)}{n~(n-1)~(n-\ell-1)}\right)  = n-2.
\end{split}
\end{equation*}  
\end{proposition}

\begin{proof}
Let us introduce an auxiliary real parameter $a\in \mathbb{R}\setminus\{0,1,\ell+1\}$ and consider the more general identity 
\begin{equation}\label{eq:generalidentity}
\begin{split}
&\ell\left(\frac{a}{a-1}\right)^{n-\ell-2} 
+ \sum_{k=\ell}^{n-2}\left(\frac{a}{a-1}\right)^{n-k-2} \left(\frac{a-k-\ell-2}{a-\ell-1}+\frac{(\ell+1)~ak~+\ell~(a-k-2)}{a~(a-1)~(a-\ell-1)}\right)  \\
& \hspace{8cm}= (n-2) +\left(1+ \frac{\ell n-(n-1)~a}{a~(a-\ell -1)}\right).
\end{split}
\end{equation}
Identity \eqref{eq:generalidentity} can be verified by induction on $n$.
By substituting $a=n$ into \eqref{eq:generalidentity}, we obtain the required result for all integers $n\geq \ell+2$.
\end{proof}

\begin{proposition}\label{prop:(n-2)l}
Let $\ell\geq 1$. Then, for all $n\geq \ell+2$, we have
\begin{equation}\label{eq:(n-2)l_identity}
\begin{split}
(n-2)^{\ell} &=\left(\frac{n}{n-1}\right)^{n-\ell-2} ~ \sum_{k=0}^{\ell-1}(n-1)^{\ell-k-1}~(n-2)^{k-1}~(n-k-2) \\
& \hspace{0.5cm} + \frac{(n-2)^{\ell-1}}{n-\ell -1} ~ \sum_{k=\ell}^{n-2}\left(\frac{n}{n-1}\right)^{n-k-2}\left((n-k-\ell-2) + \frac{(\ell+1)~nk + \ell~(n-k-2)}{n~(n-1)}\right).
\end{split}
\end{equation}
\end{proposition}

\begin{proof}
Dividing both sides of \eqref{eq:|B|.2} by $(n-1)^{n-\ell-2}$, we obtain
\[
\sum_{k=0}^{\ell-1}(n-1)^{\ell-k-1}~(n-2)^{k-1}~(n-k-2) = \ell~(n-2)^{\ell-1}.
\]
Hence,
\[
\left(\frac{n}{n-1}\right)^{n-\ell-2} \left(\sum_{k=0}^{\ell-1}(n-1)^{\ell-k-1}~(n-2)^{k-1}~(n-k-2)\right)  
= \left(\frac{n}{n-1}\right)^{n-\ell-2}\ell~(n-2)^{\ell-1}.
\]
Therefore,
\begin{align*}
&\scalebox{0.9}{$\displaystyle
\left(\frac{n}{n-1}\right)^{n-\ell-2}\left(\sum_{k=0}^{\ell-1}(n-1)^{\ell-k-1}~(n-2)^{k-1}~(n-k-2)\right)$} \\
&\scalebox{0.9}{$\displaystyle\quad
+\frac{(n-2)^{\ell-1}}{n-\ell-1}\left(\sum_{k=\ell}^{n-2}
\left(\frac{n}{n-1}\right)^{n-k-2}
\left(
(n-k-\ell-2)+\frac{(\ell+1)~nk+\ell~(n-k-2)}{n~(n-1)}
\right)\right)$} \\
&\scalebox{0.9}{$\displaystyle
=(n-2)^{\ell-1}
\Biggl[
\ell\left(\frac{n}{n-1}\right)^{n-\ell-2}
+
\sum_{k=\ell}^{n-2}
\left(\frac{n}{n-1}\right)^{n-k-2}
\left(
\frac{n-k-\ell-2}{n-\ell-1}
+
\frac{(\ell+1)~nk+\ell~(n-k-2)}{n~(n-1)~(n-\ell-1)}
\right)
\Biggr]$}.
\end{align*}
By Proposition~\ref{prop:(n-2)}, the quantity inside the brackets is equal to $n-2$. 
Hence,
\begin{align*}
&\scalebox{0.9}{$\displaystyle
(n-2)^{\ell-1}
\Biggl[
\ell\left(\frac{n}{n-1}\right)^{n-\ell-2}
+
\sum_{k=\ell}^{n-2}
\left(\frac{n}{n-1}\right)^{n-k-2}
\left(
\frac{n-k-\ell-2}{n-\ell-1}
+
\frac{(\ell+1)~nk+\ell~(n-k-2)}{n~(n-1)~(n-\ell-1)}
\right)
\Biggr]$}\\
& = (n-2)^{\ell-1}\,(n-2)\\
& = (n-2)^{\ell}.
\end{align*}
This concludes the proof.
\end{proof}

Multiplying both sides of the identity \eqref{eq:(n-2)l_identity} by $(n-1)^{n-\ell-2}~(n-\ell-1)$, we obtain the equivalent form
\begin{equation}\label{eq:Un_MTTsum01}
\begin{split}
& (n-1)^{n-\ell-2}~(n-2)^{\ell}~(n-\ell-1) \\
& = n^{n-\ell-2}~(n-\ell-1)\sum_{k=0}^{\ell -1}(n-1)^{\ell -k -1}~(n-2)^{k-1}~(n-k-2) \\
& \hspace{0.2cm}  +(n-2)^{\ell -1} \sum_{k=\ell}^{n -2}n^{n-k-2}~(n-1)^{k-\ell}\left((n-k-\ell-2) + \frac{(\ell+1)~nk + \ell~(n-k-2)}{n~(n-1)}\right),
\end{split}
\end{equation}
valid for $\ell \geq 1$ and $n \geq \ell+2$.

Combining \eqref{eq:Un_MTTsum01} with \eqref{eq:Un_MTTsum}, we recover the formula 
\[
|\UnFl| = (n-1)^{n-\ell-2}~(n-2)^{\ell}~(n-\ell-1).
\]
This provides an alternative proof of Theorem~\ref{thm:|Un|}, as an application of the matrix tree theorem.

%% file: sec04.tex
\section{\texorpdfstring{Spherical $\Gl$-parking function}{Spherical Gl parking function}}\label{sec:SPF(Gl)}

Having enumerated the uprooted spanning trees of $\Gl'$  in the preceding sections, we now turn to the study of spherical $\Gl$-parking functions. 
In this section, we recall the necessary background, summarize the known results, and establish several results that will be used in Section~\ref{sec:Counting SPF(Gl)}.

Recall that $\Knn$ denotes the complete graph with vertex set $\{0\}\cup[n]$ and designated root $0$.
The graph $\Gl$ is obtained from $\Knn$ by deleting the $\ell$ edges joining vertex $1$ to the vertices in $F_{\ell} =\{n-\ell+1,\ldots,n\}$, and we set $\Gl'=\Gl-\{0\}$.
For a vertex $r\in [n]$, we write $(\Gl';r)$ to denote the graph $\Gl'$ with root $r$.

We recall the Depth-First-Search burning algorithm of Perkinson, Yang, and Yu~\cite{PYY17}. 
Applied to the rooted graph $(\Gl';r)$ with input function $P\colon [n]\setminus\{r\}\to\mathbb{N}$, the algorithm simulates a fire starting at $r$ and spreading through $\Gl'$ according to the depth-first-search rule, where $P(j)$ represents the number of water droplets at vertex $j$ available to resist the fire. 
A vertex is said to be \emph{burnt} once it has been reached by the fire. 
At each step, if the largest unburnt vertex $j$ adjacent to a burnt vertex $i$ has $P(j)=0$, the fire spreads to $j$, and the edge $e_{i,j}$ is recorded. 
Otherwise, one water droplet at $j$ is consumed, the edge $e_{i,j}$ is declared \emph{dampened} and removed, and the fire continues to spread from $i$. 
When all edges from a burnt vertex $i$ to unburnt vertices have been dampened, the algorithm \emph{backtracks} to the previous burnt vertex in the depth-first-search order.
If all vertices of $\Gl'$ are eventually burnt, then $P$ is a $(\Gl';r)$-parking function and the tree edges form a spanning tree of $(\Gl';r)$, yielding a bijection $\phi\colon\PF(\Gl';r)\to\SPT(\Gl';r)$.

Let $\mathcal{M}_{\Gl}=\bigl\langle m_A : \emptyset\ne A\subseteq[n] \bigr\rangle\subseteq\mathbb{K}[x_1,\ldots,x_n]$ be the $\Gl$-parking function ideal, and note that $m_{[n]}=x_1 x_2\cdots x_n$. 
For $\mathcal{P}\colon[n]\to\mathbb{N}$, we write 
$\mathbf{x}^{\mathcal{P}}=\prod_{i\in[n]}x_i^{\mathcal{P}(i)}$ for the associated monomial.
Following \cite{CGS}, for $\mathcal{P}\in\SPF(\Gl)$, we define the \emph{reduced spherical $\Gl$-parking function} $\widetilde{\mathcal{P}}\colon[n]\to\mathbb{N}$ by $\widetilde{\mathcal{P}}(i)=\mathcal{P}(i)-1$ for all $i\in[n]$, and denote the set of all such functions by $\widetilde{\SPF}(\Gl)$.

\begin{theorem}[{\cite{CGS}*{Lemma~19}}]\label{thm:reducedSPF}
Let $\Gl$ be as above, with root $0$ connected to all other vertices. 
Then,
\begin{enumerate}
    \item[\rm{(1)}] $\widetilde{\SPF}(\Gl)\subseteq\PF(\Gl)$.
    \item[\rm{(2)}]  Let $\mathcal{P}\in\SPF(\Gl)$ and let $r\in[n]$ be the unique vertex such that $\widetilde{\mathcal{P}}(r)=0$ but $\widetilde{\mathcal{P}}(j)\geq1$ for all $j>r$. 
    Consider the graph $(\Gl';r)$ with vertex set $[n]$ and root $r$. 
    Then 
    $\widehat{\mathcal{P}}=\widetilde{\mathcal{P}}|_{[n]\setminus\{r\}}$ 
    is a $(\Gl';r)$-parking function.
\end{enumerate}
\end{theorem}

This enables us to use a modified Depth-First-Search burning algorithm~\cite{CGS} to associate uprooted spanning trees of $\Gl'$ with spherical $\Gl$-parking functions. 
We briefly recall this algorithm following~\cite{CGS}.
Let $\mathcal{P}\in\SPF(\Gl)$ and let $\widetilde{\mathcal{P}}$ be the associated reduced spherical $\Gl$-parking function. 
Set $r=\max\{i\in[n] : \widetilde{\mathcal{P}}(i)=0\}$ and consider the graph $\Gl'$ with root $r$. 
Since $\widehat{\mathcal{P}}=\widetilde{\mathcal{P}}|_{[n]\setminus\{r\}}$ is a $(\Gl';r)$-parking function by Theorem~\ref{thm:reducedSPF}, the Depth-First-Search burning algorithm yields a spanning tree $\phi(\widehat{\mathcal{P}})\in\SPT(\Gl';r)$. 
Moreover, since $\widehat{\mathcal{P}}(j)\geq 1$ for all $j>r$, all edges $e_{r,j}$ for $j>r$ get dampened during the algorithm, and hence 
$\phi(\widehat{\mathcal{P}})$ is an uprooted spanning tree of $\Gl'$. 
Recall that $\U(\Gl')=\UnFl$.
The modified Depth-First-Search burning algorithm induces a map
\[
\phi_{\Gl}\colon \SPF(\Gl)\longrightarrow \UnFl,
\qquad
\phi_{\Gl}(\mathcal{P})=\phi(\widehat{\mathcal{P}}).
\]
In particular, we have the following result.

\begin{theorem}[{\cite{CGS}*{Theorem~20}}]\label{thm:SPFGl_to_UnFl}
There exists an injective map
\[
\phi_{\Gl}\colon \SPF(\Gl)\hookrightarrow \UnFl.
\]
\end{theorem}

Now, we proceed to give the description of $\Im(\phi_{\Gl})$.
Let $T\in\UnFl$ be an uprooted spanning tree of $\Gl'$ with $\root(T)=r$.
Since the Depth-First-Search burning algorithm induces a bijection $\phi\colon\PF(\Gl';r)\to\SPT(\Gl';r)$, there exists a unique $(\Gl';r)$-parking function $\mathcal{P}_T$ such that $\phi(\mathcal{P}_T)=T$. 
By~\cite{CGS}*{Proposition~21}, we have $\Im(\phi_{\Gl})\subseteq\{T\in\UnFl : \mathcal{P}_T(j)\geq1 \text{ for all } j>r=\root(T)\}.
$
We define
\[
\Cnl=\{T\in\UnFl : \root(T)=r\in F_{\ell}\setminus\{n-\ell+1\} \text{ and } \mathcal{P}_T(k)\geq 1 \text{ for all } 2\leq k\leq n-\ell\},
\]
and express $\Cnl$ as a disjoint union $\Cnl=\coprod_{p=0}^{\ell -2}\Cnlp$, where $\Cnlp=\{T\in\Cnl : \root(T)=n-p\}$.
Note that for $T\in\Cnlp$, the condition $\mathcal{P}_T(k)\geq 1$ for all $2\leq k\leq n-\ell$ implies that $\mathcal{P}_T(j)=0$ for some $j\in\{n-\ell+1,\ldots,n-p-1\}$.
The following theorem gives a description of $\Im(\phi_{\Gl})$.

\begin{theorem}\label{thm:Im(Psi)}
For $n\geq 3$ and $1\leq\ell\leq n-2$, we have
\[
\Im(\phi_{\Gl})=\UnFl\setminus\Cnl.
\]
\end{theorem}


\begin{proof}
We first show that $\Cnl\cap\Im(\phi_{\Gl})=\emptyset$.
Suppose, to the contrary, that $T\in\Cnl\cap\Im(\phi_{\Gl})$, and let $\root(T)=r$. 
Let $\mathcal{P}_T$ be the unique $(\Gl';r)$-parking function such that $\phi(\mathcal{P}_T)=T$.
Since $T\in\Im(\phi_{\Gl})$, there exists $\mathcal{P}\in\SPF(\Gl)$ such that $\phi_{\Gl}(\mathcal{P})=T$.
By definition of $\phi_{\Gl}$, we have $\mathcal{P}_T=\widetilde{\mathcal{P}}\big|_{[n]\setminus\{r\}}$,
where $\widetilde{\mathcal{P}}$ is the reduced spherical $\Gl$-parking function associated with $\mathcal{P}$. 
Hence, $\mathbf{x}^{\mathcal{P}}= m_{[n]}\mathbf{x}^{\mathcal{P}_T}$.
Since $T\in\Cnl$, we have $\prod_{i=2}^{n-\ell}x_i$ divides $\mathbf{x}^{\mathcal{P}_T}$, i.e., $\prod_{i=2}^{n-\ell}x_i\mid \mathbf{x}^{\mathcal{P}_T}$.
For $A=\{2,3,\ldots,n\}$, we have $m_{[n]}\prod_{i=2}^{n-\ell}x_i=x_1 m_A$
and hence $m_A\mid\mathbf{x}^{\mathcal{P}}$, contradicting that $\mathcal{P}\in\SPF(\Gl)$. 
Therefore, $\Cnl\cap\Im(\phi_{\Gl})=\emptyset$, and hence $\Im(\phi_{\Gl})\subseteq \UnFl\setminus\Cnl$.

For the reverse inclusion, we partition $\UnFl\setminus\Cnl$ into 
two disjoint subsets as follows.
\begin{align*}
\mathcal{A}_1 &= \{T\in\UnFl : \root(T)=r\in\{2,\ldots,n-\ell\}\},\text{ and } \\
\mathcal{A}_2 &= \{T\in\UnFl : \root(T)=r\in F_{\ell} 
\text{ and } \mathcal{P}_T(k)=0 \text{ for some } 
k\in\{2,\ldots,n-\ell\}\}.
\end{align*}
We show that $\mathcal{A}_1\coprod\mathcal{A}_2\subseteq\Im(\phi_{\Gl})$.
Let $T\in\mathcal{A}_1\coprod\mathcal{A}_2$ with $\root(T)=r$, and let 
$\mathcal{P}_T$ be the corresponding $(\Gl';r)$-parking function such that $\phi(\mathcal{P}_T)=T$. 
Set $\mathbf{x}^{\mathcal{P}}= m_{[n]}\mathbf{x}^{\mathcal{P}_T}$. 
We claim that $\mathcal{P}\in\SPF(\Gl)$, i.e., 
$\mathbf{x}^{\mathcal{P}}\in\mathcal{M}_{\Gl}\setminus
\mathcal{M}_{\Gl}^{(n-2)}$. 
Since $m_{[n]}\mid\mathbf{x}^{\mathcal{P}}$, clearly $\mathbf{x}^{\mathcal{P}}\in\mathcal{M}_{\Gl}$. 
Suppose, for contradiction, that $\mathbf{x}^{\mathcal{P}}\in\mathcal{M}_{\Gl}^{(n-2)}$.
Then there exists a non-empty proper subset $A$ of $[n]$ such that 
$m_A\mid\mathbf{x}^{\mathcal{P}}$.
If $r\notin A$, then 
$m_A\mid\mathbf{x}^{\mathcal{P}}$ implies 
$\frac{m_A}{\gcd(m_A,m_{[n]})}\mid\mathbf{x}^{\mathcal{P}_T}$, 
contradicting that $\mathcal{P}_T\in \PF(\Gl';r)$. 
Now suppose that $r\in A$.
If $T\in\mathcal{A}_1$, then since the degree of $x_r$ in $\mathbf{x}^{\mathcal{P}}$ is $1$
and $m_A\mid\mathbf{x}^{\mathcal{P}}$,
we must have $A=[n]$, contradicting that $A$ is a proper subset of $[n]$. 
If $T\in\mathcal{A}_2$, then we have $r\in A\cap F_{\ell}$ and $1\nsim r$.
Since the degree of $x_r$ in $\mathbf{x}^{\mathcal{P}}$ is $1$,
we obtain $A=[n]\setminus\{1\}$ and 
$m_A=\prod_{2\leq i\leq n-\ell}x_i^2\prod_{j\in F_\ell}x_j$. 
As $T\in\mathcal{A}_2$, there exists $k\in\{2,\ldots,n-\ell\}$ such that the degree of $x_k$ in $\mathbf{x}^{\mathcal{P}}$ is $1$.
Thus, $m_A$ does not divide $\mathbf{x}^{\mathcal{P}}$, a contradiction.
Therefore, $\mathcal{A}_1\coprod\mathcal{A}_2\subseteq\Im(\phi_{\Gl})$, and this completes the proof.
\end{proof}

For $\ell=1$, the map $\phi_{\Gl}$ was shown to be surjective in~\cite{CGS}, and hence $\mathcal{C}_{n,1}=\emptyset$. 
Therefore, for the rest of this article, we assume $\ell\geq 2$.
We now give a more geometric description of the set $\Cnlp$ by introducing the notion of a maximum path in $F_{\ell}$ associated with an uprooted tree, and fixing notation that will be used throughout.
Recall that $\Un$ denotes the set of all uprooted trees with vertex set $[n]$.

Let $T\in\Un$ be an uprooted tree with $\root(T)=r\in F_{\ell}$, and let the largest son of $r$ in $T$ be a vertex $a_1\in F_{\ell}$ satisfying $a_1<r$. 
The \emph{maximum path of $T$ in $F_{\ell}$}, denoted by $\Pmax(T)$, is the path
\[
r=a_0 \to a_1 \to a_2 \to \cdots \to a_{\alpha-1}\to a_{\alpha}
\qquad (\alpha\geq 1),
\]
satisfying the following conditions:
\begin{enumerate}
    \item $a_i\in F_{\ell}$ for all $0\leq i\leq \alpha$,
    \item for each $1\leq i\leq \alpha$, the vertex $a_i$ is the largest son of $a_{i-1}$ in $T$,
    \item $a_1<r=a_0$, and
    \item either $a_{\alpha}$ is a leaf of $T$, or every son of $a_{\alpha}$ belongs to $\{2,\ldots,n-\ell\}$, i.e., $\son_T(a_{\alpha})\cap F_{\ell}=\emptyset$.
\end{enumerate}

For a vertex $v$ of a rooted tree $T$, let $\F_{v}(T)$, or simply $\F_{v}$ when $T$ is clear from the context, denote the rooted forest formed by all descendants of $v$, where the roots of the connected components are precisely the sons of $v$ in $T$.
More generally, for any vertex $v$ of $T$, the notation $\F_{(v)}(T)$ denotes the rooted forest of descendants of $v$ with certain subtrees removed, where the excluded subtrees will be clear from the context.
The rooted subtree of $T$ consisting of $v$ together with all of its 
descendants, rooted at $v$, is $\overline{\F}_{v}(T)=v\vee \F_{v}(T)$ (see, Figure~\ref{fig:vFv_structure}).

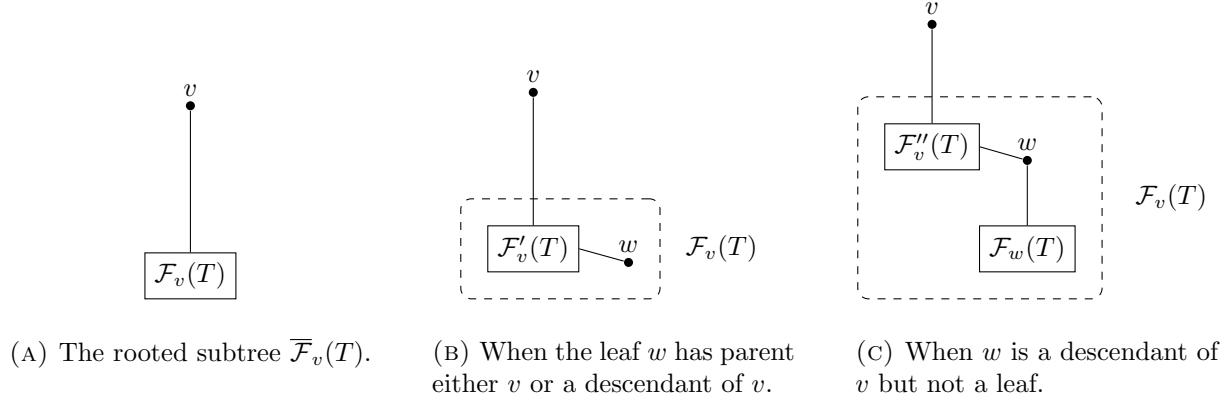
\begin{figure}[ht]
\centering

\begin{subfigure}[t]{0.3\textwidth}
\centering
\begin{tikzpicture}[scale=0.9,
    every node/.style={font=\small}]

\node[circle,fill=black,inner sep=1.2pt] (v) at (0,1.5) {};
\node[above] at (v) {$v$};

\node[draw,rectangle,minimum width=1cm,minimum height=0.5cm] (Fv) at (0,-1) {$\F_v(T)$};

\draw (v)--(Fv);

\end{tikzpicture}

\vspace{3mm}

\caption{The rooted subtree $\overline{\F}_v(T)$.}
\end{subfigure}
\hfill
\begin{subfigure}[t]{0.3\textwidth}
\centering
\begin{tikzpicture}[scale=0.9,
    every node/.style={font=\small}]

\node[circle,fill=black,inner sep=1.2pt] (v) at (0,0.5) {};
\node[above] at (v) {$v$};

\node[draw,rectangle, minimum width=1cm, minimum height=0.5cm] (Fv) at (0,-1.8) {$\F'_v(T)$};

\draw (v)--(Fv);

\node[circle,fill=black,inner sep=1.2pt] (w) at (1.4,-2) {};
\node[above] at (w) {$w$};

\draw (Fv.east)--(w);

\node[
draw,
dashed,
rounded corners,
inner sep=0.35cm,
fit=(Fv) (w)
] (Fvbox) {};

\node[right=0.2cm] at (Fvbox.east)
{$\F_v(T)$};

\end{tikzpicture}

\vspace{3mm}

\caption{When the leaf $w$ has parent either $v$ or a descendant of $v$.}
\end{subfigure}
\hfill
\begin{subfigure}[t]{0.3\textwidth}
\centering
\begin{tikzpicture}[scale=0.9,
    every node/.style={font=\small}]

\node[circle,fill=black,inner sep=1.2pt] (v) at (0,0) {};
\node[above] at (v) {$v$};

\node[draw,rectangle, minimum width=1cm, minimum height=0.5cm] (Fv) at (0,-1.8) {$\F''_v(T)$};

\draw (v)--(Fv);

\node[circle,fill=black,inner sep=1.2pt] (w) at (1.4,-2) {};
\node[above] at (w) {$w$};

\draw (Fv.east)--(w);

\node[draw,rectangle, minimum width=1cm, minimum height=0.5cm] (Fw) at (1.4,-3.3) {$\F_w(T)$};

\draw (w)--(Fw);

\node[
draw,
dashed,
rounded corners,
inner sep=0.35cm,
fit=(Fv) (w) (Fw)
] (Fvbox) {};

\node[right=0.3cm] at (Fvbox.east) {$\F_v(T)$};

\end{tikzpicture}

\vspace{3mm}

\caption{When $w$ is a descendant of $v$ but not a leaf.}
\end{subfigure}

\caption{Possible structures of the rooted subtree $\overline{\F}_v(T)$.}
\label{fig:vFv_structure}
\end{figure}

For the maximum path $\Pmax(T): a_0\to a_1\to \cdots \to a_{\alpha}$, let $\F_{(a_i)}(T)$ denote the rooted forest obtained from $\F_{a_i}(T)$ by removing the subtree rooted at $a_{i+1}$. 
Equivalently, the roots of the connected components of $\F_{(a_i)}(T)$ are precisely the sons of $a_i$ other than $a_{i+1}$.

Let $u$ and $v$ be vertices of the rooted tree $T$ such that $u$ is a leaf of $T$ and is not a descendant of $v$. 
We construct a new rooted tree $T'$ by removing the forest $\F_v(T)$ from $T$ and attaching each root of $\F_v(T)$ to the vertex $u$. 
We denote this operation by
\[
T'= T\bigl[\F_v(T)\to u\bigr],
\]
and say that $T'$ is obtained from $T$ by \emph{attaching the descendants of $v$ to $u$}.
We note that this operation is defined only when $u$ is a leaf of $T$ and is not a descendant of $v$.
This operation is illustrated in Figure~\ref{fig:T to T'}. 
By construction, $\F_u(T')=\F_v(T),$ and hence $T=T'\bigl[\F_u(T')\to v\bigr].$
Thus, the transformation $T\mapsto T'$ is reversible.

\begin{figure}[ht]
\centering
\begin{tikzpicture}[scale=0.82,
    every node/.style={font=\small}]


\node[circle,fill=black,inner sep=1.2pt] (r) at (0,0) {};
\node[above] at (r) {$r$};

\node[draw,rectangle,
      minimum width=1cm,
      minimum height=0.5cm]
(Frp) at (-2,-2)
{$\F'_{(r)}(T)$};

\draw (r)--(Frp);

\node[circle,fill=black,inner sep=1.2pt]
(u) at (-0.5,-2.3) {};
\node[above] at (u) {$u$};

\draw (Frp.east)--(u);

\node[draw,rectangle,
      minimum width=1cm,
      minimum height=0.5cm]
(Frpp) at (3,-2)
{$\F''_{(r)}(T)$};

\draw (r)--(Frpp);

\node[circle,fill=black,inner sep=1.2pt]
(v) at (4.5,-2.3) {};
\node[above] at (v) {$v$};

\draw (Frpp.east)--(v);

\node[draw,rectangle,
      minimum width=1cm,
      minimum height=0.5cm]
(Fv) at (4.5,-3.5)
{$\F_v (T)$};

\draw (v)--(Fv);

\node at (6.2,-2) {$\longrightarrow$};

\node at (1,-4.3) {$T$};


\node[circle,fill=black,inner sep=1.2pt] (r2) at (10,0) {};
\node[above] at (r2) {$r$};

\node[draw,rectangle,
      minimum width=1cm,
      minimum height=0.5cm]
(Frp2) at (8,-2)
{$\F'_{(r)}(T)$};

\draw (r2)--(Frp2);

\node[circle,fill=black,inner sep=1.2pt]
(u2) at (9.5,-2.3) {};
\node[above] at (u2) {$u$};

\draw (Frp2.east)--(u2);

\node[draw,rectangle,
      minimum width=1cm,
      minimum height=0.5cm]
(Fv2) at (9.5,-3.5)
{$\F_v (T)$};

\draw (u2)--(Fv2);

\node[draw,rectangle,
      minimum width=1cm,
      minimum height=0.5cm]
(Frpp2) at (13,-2)
{$\F''_{(r)}(T)$};

\draw (r2)--(Frpp2);

\node[circle,fill=black,inner sep=1.2pt]
(v2) at (14.5,-2.3) {};
\node[above] at (v2) {$v$};

\draw (Frpp2.east)--(v2);

\node at (11,-4.3) {$T'$};

\end{tikzpicture}
\caption{The operation $T'=T\bigl[\F_v(T)\to u\bigr]$.} 
\label{fig:T to T'}
\end{figure}

\begin{proposition}\label{prop:Cnlp_characterization}
Let $T\in \UnFl$. Then $T\in \Cnlp$ if and only if the following conditions are satisfied.
\begin{enumerate}[label={\rm(\arabic*)}]
    \item \label{cond1} $\root(T)=r=n-p$ for some integer $p$ with 
    $0\leq p\leq \ell-2$, and
    \item \label{cond2} the maximum path $\Pmax(T)$ of $T$ in $F_{\ell}$ is of the form
    \begin{equation}\label{eq:Pmax_path}
    n-p=a_0 \to a_1 \to a_2 \to \cdots \to a_{\alpha-1}\to a_{\alpha} \qquad (\alpha\geq 1),
    \end{equation}
    where $a_{\alpha}$ is a leaf of $T$.
\end{enumerate}
\end{proposition}

\begin{proof}
Let $T\in\Cnlp$. 
Then $\root(T)=r=n-p$ for some $0\leq p\leq\ell-2$, and $\mathcal{P}_T(k)\geq 1$ for all $2\leq k\leq n-\ell.$
Moreover, $\mathcal{P}_T(k)\geq 1$ for all $r<k\leq n$.
Therefore, proceeding as per the Depth-First-Search burning algorithm, we see that the largest son $a_1$ of $r$ satisfies $n-\ell+1\leq a_1<r$. 
Suppose the maximum path $\Pmax(T)$ has the form \eqref{eq:Pmax_path}.
We claim that $a_{\alpha}$ is a leaf of $T$. 
Since $\mathcal{P}_T(k)\geq 1$ for all $2\leq k\leq n-\ell$, every edge joining $a_{\alpha}$ to a vertex $k\in\{2,\dots,n-\ell\}$ is dampened during the Depth-First-Search burning algorithm. 
Thus, the algorithm backtracks from $a_{\alpha}$, and hence $a_{\alpha}$ must be a leaf.

Conversely, suppose that $T\in \UnFl$ satisfies conditions \ref{cond1} and \ref{cond2}. 
Then $\root(T)=n-p$ for some $0\leq p\leq\ell-2$.
Since $a_{\alpha}$ is a leaf, the Depth-First-Search burning algorithm backtracks from $a_{\alpha}$, which implies $\mathcal{P}_T(k)\geq 1$ for all $2\leq k\leq n-\ell$. Therefore $T\in\Cnlp$.
\end{proof}

See Figure~\ref{fig:Cnlp_structure} for an illustration of the structure of trees in $\Cnlp$.

\begin{figure}[ht]
\centering
\begin{tikzpicture}[scale=0.7,
    every node/.style={font=\small}]

\node[circle,fill=black,inner sep=1.2pt] (r) at (0,0) {};
\node[above] at (r) {$r=n-p$};

\node[circle,fill=black,inner sep=1.2pt] (a1) at (-2,-2) {};
\node[right] at (a1) {$a_1$};

\node[circle,fill=black,inner sep=1.2pt] (a2) at (-2,-4) {};
\node[right] at (a2) {$a_2$};

\node[circle,fill=black,inner sep=1.2pt] (am) at (-2,-7) {};
\node[right] at (am) {$a_{\alpha-1}$};

\node[circle,fill=black,inner sep=1.2pt] (aa) at (-2,-9) {};
\node[right] at (aa) {$a_{\alpha}$};

\draw (r)--(a1)--(a2);
\draw[dashed] (a2)--(am);
\draw (am)--(aa);

\node[draw,rectangle,
      minimum width=1cm,
      minimum height=0.5cm]
(Fr) at (2,-2)
{$\F_{(r)}$};

\draw (r)--(Fr);

\node[draw,rectangle,
      minimum width=1cm,
      minimum height=0.5cm]
(Fa1) at (-3.5,-3)
{$\F_{(a_1)}$};

\draw (a1)--(Fa1);

\node[draw,rectangle,
      minimum width=1cm,
      minimum height=0.5cm]
(Fa2) at (-3.5,-5)
{$\F_{(a_2)}$};

\draw (a2)--(Fa2);

\node[draw,rectangle,
      minimum width=1cm,
      minimum height=0.5cm]
(Fam) at (-3.5,-8)
{$\F_{(a_{\alpha-1})}$};

\draw (am)--(Fam);

\end{tikzpicture}
\caption{A tree $T\in\Cnlp$ with maximum path
$\Pmax(T): r\to a_1\to a_2\to \cdots \to a_{\alpha}$ and the associated forests.}
\label{fig:Cnlp_structure}
\end{figure}

\begin{proposition}\label{prop:T*_in_Cnlp}
Let $0\leq p\leq \ell-2$, and let $T\in \Un$ be an uprooted tree with $\root(T)=n-p$.
Suppose that $T$ has the following properties.
\begin{enumerate}
    \item[{\rm(1)}]  The maximum path $\Pmax(T)$ of $T$ in $F_{\ell}$ is of the form $n-p=a_0\to a_1\to \cdots \to a_{\alpha-1}\to a_{\alpha},$
    where $\alpha\geq 1$ and $a_1<n-p$,

    \item[{\rm(2)}]  vertex $1$ is a leaf of $T$ and is not adjacent to any vertex in $F_{\ell}$, and

    \item[{\rm(3)}]  $1$ is not a descendant of $a_{\alpha}$ in $T$.
\end{enumerate}
Let $(T)^*$ be the tree obtained from $T$ by attaching the descendants 
$\F_{a_{\alpha}}(T)$ of $a_{\alpha}$ to the leaf $1$, that is,
\[
(T)^*= T\bigl[\F_{a_{\alpha}}(T)\to 1\bigr].
\]
Then $(T)^*\in\Cnlp$.
\end{proposition}

\begin{proof}
In the tree $(T)^*=T\bigl[\F_{a_{\alpha}}(T)\to 1\bigr]$,
the vertex $a_{\alpha}$ becomes a leaf. 
Moreover, since the roots of the connected components of $\F_{a_{\alpha}}(T)$ do not belong to $F_{\ell}$, the vertex $1$ remains non-adjacent to every vertex in $F_{\ell}$.
Hence, $(T)^*\in \UnFl$.
Further, the maximum path of $(T)^*$ in $F_{\ell}$ remains $n-p=a_0\to a_1\to \cdots \to a_{\alpha},$ with $a_{\alpha}$ now a leaf of $(T)^*$. 
Therefore, all the hypotheses of Proposition~\ref{prop:Cnlp_characterization} are satisfied, and hence $(T)^*\in \Cnlp$.
\end{proof}

The results established in this section will be used in Section~\ref{sec:Counting SPF(Gl)} to enumerate $\Cnlp$, and thereby determine $|\SPF(\Gl)|$.

%% file: sec05.tex
\section{\texorpdfstring{Enumeration of Spherical $\Gl$-parking function}{Counting Spherical Parking Function}}\label{sec:Counting SPF(Gl)}

Our goal in this section is to enumerate the spherical $\Gl$-parking functions. 
By Theorems~\ref{thm:SPFGl_to_UnFl} and~\ref{thm:Im(Psi)}, this reduces to computing  $|\UnFl\setminus\Cnl|$. 
Since the value of $|\UnFl|$ was determined in the preceding sections, it remains to enumerate $\Cnl$. 
Recall that $\Cnl=\coprod_{p=0}^{\ell-2}\Cnlp$.
Thus, it suffices to determine $|\Cnlp|$ for each $0\le p\le \ell-2$.
As counting $\Cnlp$ directly is difficult, we first establish several intermediate results that will lead to an explicit formula for $|\Cnlp|$, and hence for $|\SPF(\Gl)|$.
Henceforth, we assume that $n\geq\ell+2$ and $\ell\geq 2$.

Let $\Tnl$ denote the set of trees $\widetilde{T}$ with vertex set $[n]$ satisfying the following conditions:
\begin{enumerate}
    \item $\widetilde{T}$ is uprooted with $\root(\widetilde{T})=r=n-\ell+2,$

    \item vertex $1$ is not adjacent to the root $r$, and

    \item vertex $n-\ell+1$ is a leaf adjacent to the root $r$.
\end{enumerate}

For $\widetilde{T}\in\Tnl$, let $\F_{(r)}(\widetilde{T})$ denote the rooted forest consisting of all descendants of the root $r$ in $\widetilde{T}$, excluding the leaf $n-\ell+1$, and let $S=\son_{\widetilde{T}}(r)\setminus\{n-\ell+1\}$
be the set of sons of the root $r$ other than the leaf $n-\ell+1$. 
Then the roots of the connected components of the forest $\F_{(r)}(\widetilde{T})$ are precisely the vertices in $S$, and $\emptyset\neq S\subseteq \{2,3,\ldots,n-\ell\}.$
Thus, every tree $\widetilde{T}\in \Tnl$ has the form illustrated in Figure~\ref{fig:Tnl_structure}.
\begin{figure}[ht]
\centering
\begin{tikzpicture}[scale=0.8,
    every node/.style={font=\small}]

\node[circle, fill=black, inner sep=1.2pt] (r) at (-2,0) {};
\node[above] at (r) {$r=n-\ell+2$};

\node[circle, fill=black, inner sep=1.2pt] (b) at (-3.5,-1.6) {};
\node[left] at (b) {$n-\ell+1$};

\node[draw, rectangle, minimum width=1cm, minimum height=0.5cm]
(F) at (0,-1.6) {$\F_{(r)}(\widetilde{T})$};

\draw (r)--(b);
\draw (r)--(F);

\node at (-1.5,-3.0)
{$\widetilde{T}$ with $\root(\widetilde{T})=r=n-\ell+2$};

\node[circle, fill=black, inner sep=1.2pt] (i1) at (5,0) {};
\node[above] at (i1) {$i_1$};

\node[draw, rectangle, minimum width=1cm, minimum height=0.5cm]
(F1) at (5,-1.8) {$\F_{i_1}(\widetilde{T})$};
\draw (i1)--(F1);

\node[circle, fill=black, inner sep=1.2pt] (i2) at (8,0) {};
\node[above] at (i2) {$i_2$};

\node[draw, rectangle, minimum width=1cm, minimum height=0.5cm]
(F2) at (8,-1.8) {$\F_{i_2}(\widetilde{T})$};
\draw (i2)--(F2);

\node at (10.5,-0.9) {$\cdots$};

\node[circle, fill=black, inner sep=1.2pt] (is) at (13,0) {};
\node[above] at (is) {$i_s$};

\node[draw, rectangle, minimum width=1cm, minimum height=0.5cm]
(Fs) at (13,-1.8) {$\F_{i_s}(\widetilde{T})$};
\draw (is)--(Fs);

\node at (9,-3.6)
{$\F_{(r)}(\widetilde{T})$ with
$\emptyset\neq S=\{i_1,i_2,\dots,i_s\}\subseteq\{2,3,\dots,n-\ell\}$};
\end{tikzpicture}
\caption{The structure of a tree $\widetilde{T}\in\Tnl$.}
\label{fig:Tnl_structure}
\end{figure}

\begin{proposition}\label{prop:Tnl_count}
The cardinality of $\Tnl$ is given by
\[
|\Tnl| = (n-1)^{n-\ell-2}~(n-2)^{\ell-2}~(n-\ell-1).
\]
\end{proposition}

\begin{proof}
Every tree $\widetilde{T}\in \Tnl$ is uniquely determined by the rooted forest $\mathcal{F}_{(r)}(\widetilde{T}),$
as illustrated in Figure~\ref{fig:Tnl_structure}.
Let $S=\{i_1,i_2,\ldots,i_s\}$
be the set of roots of the connected components of the forest
$\mathcal{F}_{(r)}(\widetilde{T})$. 
Then $\emptyset\neq S\subseteq \{2,3,\ldots,n-\ell\}.$ 
Therefore, the set $S$ with $s$ elements can be chosen in 
$\binom{n-\ell-1}{s}$
ways.
Now fix such a subset $S$ with $|S|=s$. By Theorem~\ref{Thm:forest}, the number of rooted forests with roots specified by the vertices of $S$ is
$s(n-2)^{n-3-s}.$
Hence,
\begin{equation*}
\begin{split}
|\Tnl| & = \sum_{s=1}^{n-\ell-1}\binom{n-\ell-1}{s}s(n-2)^{n-3-s} \\
& = \sum_{s=1}^{n-\ell-1}(n-\ell-1)\binom{n-\ell-2}{s-1}(n-2)^{n-3-s}  \\
& \hspace{8.8cm} (\text{By setting } s-1 =j)\\
& = (n-\ell-1)\sum_{j=0}^{n-\ell-2}\binom{n-\ell-2}{j}(n-2)^{n-4-j} \\
&  =  (n-\ell-1)~(n-2)^{\ell-2}\left(\sum_{j=0}^{n-\ell-2}\binom{n-\ell-2}{j}(n-2)^{n-\ell-2-j} \right)\\
&  =  (n-\ell-1)~(n-2)^{\ell-2}\bigl((n-2)+1\bigr)^{n-\ell-2}\\
&  =  (n-1)^{n-\ell-2}~(n-2)^{\ell-2}~(n-\ell-1).
\end{split}
\end{equation*}
This completes the proof.
\end{proof}

Throughout this section, the phrase \emph{attach a rooted forest to a vertex} means attaching each root of the connected components of the forest to that vertex. 
Likewise, \emph{attach a rooted subtree to a vertex} means attaching the root of the subtree to that vertex.

\begin{proposition}\label{prop:Tnl_bijection}
For $n\geq \ell+2$ and $\ell\geq 2$,
there exists a bijection $\psi\colon \Tnl \longrightarrow \Cnllminustwo$. 
Thus, we deduce that
\[
|\Cnllminustwo|=(n-1)^{n-\ell-2}~(n-2)^{\ell-2}~(n-\ell-1).
\]
\end{proposition}

\begin{proof}
Let $\widetilde{T}\in \Tnl$. Consider the unique path $Q$ in $\widetilde{T}$ from the root $r=n-\ell+2$ to the vertex $1$. 
Suppose that $Q$ is of the form
\[
Q \colon \quad 
r\to b_1\to b_2\to \cdots \to b_{\beta}\to 1, \qquad (\beta\geq 1).
\]
Since all the sons of the root other than the leaf $n-\ell+1$ belong to $\{2,3,\ldots,n-\ell\}$, we have $b_1\notin F_{\ell}.$
Let $\son_{\widetilde{T}}(1)$ be the set of sons of $1$ in $\widetilde{T}$, and $\Par_{\widetilde{T}}(1)$ be the parent of $1$ in $\widetilde{T}$. 
Clearly, $\Par_{\widetilde{T}}(1)=b_{\beta}$.

We now consider the following cases.
In each of the cases, we also describe the reverse construction, thereby showing that the corresponding transformations are reversible.

\medskip

\noindent 
$\textbf{Case I.}$ 
Suppose that in $\widetilde{T}$, the vertex $1$ is not adjacent to any vertex in $F_{\ell}$.
That is, $b_{\beta}=\Par_{\widetilde{T}}(1)\notin F_{\ell}$ and $\son_{\widetilde{T}}(1)\cap F_{\ell} =\emptyset.$
Hence, we have  $\widetilde{T}\in \Cnllminustwo.$
In this case, we define $\psi(\widetilde{T})=\widetilde{T}.$

\medskip
\noindent
Conversely, if $T\in \Cnllminustwo$, then $\root(T)=r=n-\ell+2$, and the vertex $1$ is not adjacent to any vertex in $F_{\ell}$.
Moreover, if the maximum path $\Pmax(T)$ in $F_{\ell}$ is of the form $r=n-\ell+2\to n-\ell+1,$ with $n-\ell+1$ a leaf of $T$, then $T\in \Tnl$ satisfying Case~I, and $\psi(T)=T.$

\medskip
\noindent 
$\textbf{Case II.}$ 
Suppose that $b_{\beta}=\Par_{\widetilde{T}}(1)\notin F_{\ell}$, but $\son_{\widetilde{T}}(1)\cap F_{\ell}\neq \emptyset$.
In this case, consider the tree $\widetilde{T}'= \widetilde{T}\bigl[\F_1(\widetilde{T})\to (n-\ell+1)\bigr]$
obtained by attaching descendants of $1$ to the leaf $n-\ell+1$ in $\widetilde{T}$. 
Then $1$ becomes a leaf of $\widetilde{T}'$, and $\widetilde{T}'$ satisfies the hypotheses of Proposition~\ref{prop:T*_in_Cnlp}. 
Applying Proposition~\ref{prop:T*_in_Cnlp}, we obtain a tree $(\widetilde{T}')^*\in \Cnllminustwo$, and define $\psi(\widetilde{T})=(\widetilde{T}')^*$.

\medskip
\noindent 
We now identify the trees in $\Cnllminustwo$ arising from trees in $\Tnl$ satisfying Case~II. 
Let $T\in\Cnllminustwo$ whose maximum path $\Pmax(T)$ has the form
\begin{equation}\label{eq:caseII_pmax}
r=n-\ell+2\to \mathrel{\underset{\substack{\rotatebox{90}{$=$} \\ n-\ell+1}}{a_1}}\to a_2\to \cdots \to a_{\alpha},
\qquad (\alpha>1),
\end{equation}
where $a_{\alpha}$ is a leaf of $T$, and $1$ is not a descendant of $a_1=n-\ell+1$.
Consider the trees $T'=T\bigl[\F_1(T)\to a_{\alpha}\bigr]$,
obtained by attaching descendants of $1$ to the leaf $a_{\alpha}$ in $T$, and $T^2=T'\bigl[\F_{a_1}(T')\to 1\bigr]$,
obtained by attaching all descendants of $a_1=n-\ell+1$ to the leaf $1$ in $T'$.
Clearly, $T^2\in \Tnl$, and $T^2=\widetilde{T}$ is a tree satisfying the conditions of Case~II. 
Reversing the above constructions, we obtain $\psi(\widetilde{T})=\psi(T^2)=T$.

\medskip
\noindent 
$\textbf{Case III.}$ 
Suppose that $b_{\beta}=\Par_{\widetilde{T}}(1)\in F_{\ell}.$
Since $b_1\notin F_{\ell}$, we necessarily have $\beta>1$.
Interchanging the vertices $b_1$ and $b_{\beta}$ in $\widetilde{T}$, we obtain a tree, denoted by $\widetilde{T}^{2}$. 
Under this interchange, the path $Q$ becomes a path $Q'$ of the form
\begin{equation}\label{eq:Qprime}
Q': \quad r\to b_1'\to b_2'\to \cdots \to b_{\beta}'\to 1,
\end{equation}
where $b_1'=b_{\beta}$, $b_{\beta}'=b_1$, and $b_j'=b_j$, for all $j\notin\{1,\beta\}$.
For each $b_j'$ on the path $Q'$, let $\mathcal{F}_{(b_j')}(\widetilde{T}^{2})$ denote the rooted forest consisting of all descendants of $b_j'$ in $\widetilde{T}^{2}$ excluding those lying on the path $Q'$. 
Clearly, $\mathcal{F}_{(b_j')}(\widetilde{T}^{2})=\mathcal{F}_{(b_j)}(\widetilde{T})$ for all $j$.
Suppose that $b_1',b_2',\ldots,b_t'\in F_{\ell}$, but  $b_{t+1}'\notin F_{\ell}$.
Since $b_{\beta}'=b_1\notin F_{\ell}$, we have $t<\beta$.
The path $Q'$ in $\widetilde{T}^{2}$ together with its descendants is illustrated in Figure~\ref{fig:caseIII_structure}.

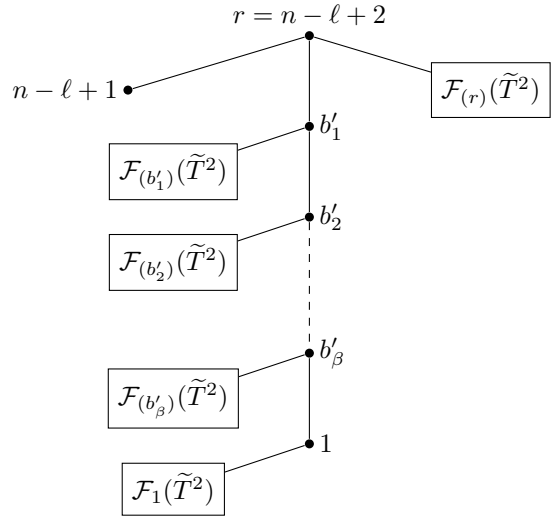
\begin{figure}[ht]
\centering
\begin{tikzpicture}[scale=0.6,
    every node/.style={font=\small}]

\node[circle,fill=black,inner sep=1.2pt] (r) at (0,0) {};
\node[above] at (r) {$r=n-\ell+2$};

\node[circle,fill=black,inner sep=1.2pt] (np1) at (-4,-1.2) {};
\node[left] at (np1) {$n-\ell+1$};

\node[circle,fill=black,inner sep=1.2pt] (b1) at (0,-2) {};
\node[right] at (b1) {$b_1'$};

\node[circle,fill=black,inner sep=1.2pt] (b2) at (0,-4) {};
\node[right] at (b2) {$b_2'$};

\node[circle,fill=black,inner sep=1.2pt] (bb) at (0,-7) {};
\node[right] at (bb) {$b_\beta'$};

\node[circle,fill=black,inner sep=1.2pt] (one) at (0,-9) {};
\node[right] at (one) {$1$};

\draw (r)--(np1);
\draw (r)--(b1);
\draw (b1)--(b2);
\draw[dashed] (b2)--(bb);
\draw (bb)--(one);

\node[draw,rectangle,
      minimum width=1cm,
      minimum height=0.5cm]
(Fr) at (4,-1.2)
{$\F_{(r)}(\widetilde T^2)$};
\draw (r)--(Fr);

\node[draw,rectangle,
      minimum width=1cm,
      minimum height=0.5cm]
(Fb1) at (-3,-3)
{$\F_{(b_1')}(\widetilde T^2)$};
\draw (b1)--(Fb1);

\node[draw,rectangle,
      minimum width=1cm,
      minimum height=0.5cm]
(Fb2) at (-3,-5)
{$\F_{(b_2')}(\widetilde T^2)$};
\draw (b2)--(Fb2);

\node[draw,rectangle,
      minimum width=1cm,
      minimum height=0.5cm]
(Fbb) at (-3,-8)
{$\F_{(b_\beta')}(\widetilde T^2)$};
\draw (bb)--(Fbb);

\node[draw,rectangle,
      minimum width=1cm,
      minimum height=0.5cm]
(F1) at (-3,-10)
{$\F_{1}(\widetilde T^2)$};
\draw (one)--(F1);

\end{tikzpicture}
\caption{The tree $\widetilde{T}^{2}$ with path $Q'$ and the associated descendant forests.}
\label{fig:caseIII_structure}
\end{figure}

\medskip
\noindent
We now divide Case~III into two subcases.

\medskip
\noindent 
$\textbf{Subcase III(a).}$ 
Suppose that for every $1\leq j\leq t$, all sons of $b_j'$ in the rooted forest $\mathcal{F}_{(b_j')}(\widetilde{T}^{2})$ are smaller than $b_j'$.
Let $\overline{\mathcal{F}}_{b_1'}(\widetilde{T}^{2})=b_1'\vee \mathcal{F}_{b_1'}(\widetilde{T}^{2})$ denote the rooted subtree of $\widetilde{T}^{2}$ consisting of $b_1'$ together with all of its descendants, rooted at $b_1'$. 
Consider the tree $\widetilde{T}^{3}$ obtained from $\widetilde{T}^2$, by removing the rooted subtree $\overline{\mathcal{F}}_{b_1'}(\widetilde{T}^{2})$ and attaching it to $n-\ell+1$, as illustrated in Figure~\ref{fig:T3_structure}.

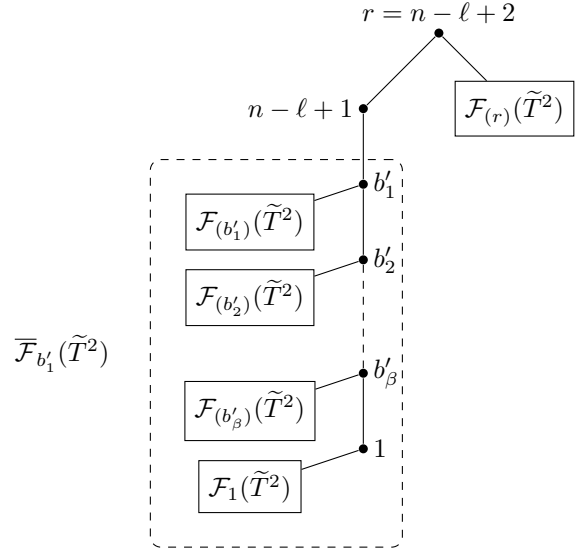
\begin{figure}[ht]
\centering
\begin{tikzpicture}[scale=0.5,
    every node/.style={font=\small}]

\node[circle,fill=black,inner sep=1.2pt] (r) at (0,0) {};
\node[above] at (r) {$r=n-\ell+2$};

\node[circle,fill=black,inner sep=1.2pt] (np1) at (-2,-2) {};
\node[left] at (np1) {$n-\ell+1$};

\node[circle,fill=black,inner sep=1.2pt] (b1) at (-2,-4) {};
\node[right] at (b1) {$b_1'$};

\node[circle,fill=black,inner sep=1.2pt] (b2) at (-2,-6) {};
\node[right] at (b2) {$b_2'$};

\node[circle,fill=black,inner sep=1.2pt] (bb) at (-2,-9) {};
\node[right] at (bb) {$b_\beta'$};

\node[circle,fill=black,inner sep=1.2pt] (one) at (-2,-11) {};
\node[right] at (one) {$1$};

\draw (r)--(np1);
\draw (np1)--(b1);

\draw (b1)--(b2);
\draw[dashed] (b2)--(bb);
\draw (bb)--(one);

\node[draw,rectangle,
      minimum width=1cm,
      minimum height=0.5cm]
(Fr) at (2,-2)
{$\F_{(r)}(\widetilde T^2)$};

\draw (r)--(Fr);

\node[draw,rectangle,
      minimum width=1cm,
      minimum height=0.5cm]
(Fb1) at (-5,-5)
{$\F_{(b_1')}(\widetilde T^2)$};

\draw (b1)--(Fb1);

\node[draw,rectangle,
      minimum width=1cm,
      minimum height=0.5cm]
(Fb2) at (-5,-7)
{$\F_{(b_2')}(\widetilde T^2)$};

\draw (b2)--(Fb2);

\node[draw,rectangle,
      minimum width=1cm,
      minimum height=0.5cm]
(Fbb) at (-5,-10)
{$\F_{(b_\beta')}(\widetilde T^2)$};

\draw (bb)--(Fbb);

\node[draw,rectangle,
      minimum width=1cm,
      minimum height=0.5cm]
(F1) at (-5,-12)
{$\F_{1}(\widetilde T^2)$};

\draw (one)--(F1);

\node[
draw,
dashed,
rounded corners,
inner sep=0.45cm,
fit=(one) (Fb1) (Fb2) (Fbb) (F1)
] (Qbox) {};

\node[left=0.4cm] at (Qbox.west)
{$\overline{\mathcal{F}}_{b_1'}(\widetilde{T}^{2})$};

\end{tikzpicture}
\caption{Illustration of the tree $\widetilde{T}^{3}$.}
\label{fig:T3_structure}
\end{figure}

\noindent

Next, remove the forests $\F_{(b_j')}(\widetilde{T}^2)$ for $1\leq j\leq t$ from $\widetilde{T}^3$. 
Then attach $\F_{(b_1')}(\widetilde{T}^2)$ to $n-\ell+1$, and, for each $1<j\leq t$, attach $\F_{(b_j')}(\widetilde{T}^2)$ to $b_{j-1}'$, thereby obtaining $\widetilde{T}^4$ (see Figure~\ref{fig:T4_structure}).

\begin{figure}[ht]
\centering

\begin{tikzpicture}[scale=0.4,
    every node/.style={font=\small}]

\node[circle,fill=black,inner sep=1.2pt] (r) at (0,0) {};
\node[above] at (r) {$r$};

\node[circle,fill=black,inner sep=1.2pt] (nl1) at (-3,-1) {};
\node[above left] at (nl1) {$n-\ell+1$};

\node[circle,fill=black,inner sep=1.2pt] (b1) at (-3,-3) {};
\node[right] at (b1) {$b_1'$};

\node[circle,fill=black,inner sep=1.2pt] (btm1) at (-3,-6) {};
\node[right] at (btm1) {$b_{t-1}'$};

\node[circle,fill=black,inner sep=1.2pt] (bt) at (-3,-8) {};
\node[right] at (bt) {$b_t'$};

\node[circle,fill=black,inner sep=1.2pt] (btp1) at (-3,-10) {};
\node[right] at (btp1) {$b_{t+1}'$};

\node[circle,fill=black,inner sep=1.2pt] (beta) at (-3,-13) {};
\node[right] at (beta) {$b_\beta'$};

\node[circle,fill=black,inner sep=1.2pt] (one) at (-3,-15) {};
\node[right] at (one) {$1$};

\draw (r)--(nl1)--(b1);

\draw[dashed] (b1)--(btm1);

\draw (btm1)--(bt)--(btp1);

\draw[dashed] (btp1)--(beta);

\draw (beta)--(one);

\node[draw,rectangle,
      minimum width=1cm,
      minimum height=0.5cm]
(Fr) at (4.2,-2)
{$\F_{(r)}(\widetilde T^2)$};

\draw (r)--(Fr);

\node[draw,rectangle,
      minimum width=1cm,
      minimum height=0.5cm]
(Fb1) at (-7,-2.3)
{$\F_{(b_1')}(\widetilde T^2)$};

\draw (nl1)--(Fb1);

\node[draw,rectangle, minimum width=1cm, minimum height=0.5cm] (Fbt) at (-7,-4.5) {$\F_{(b_2')}(\widetilde T^2)$};

\draw (b1)--(Fbt);

\node[draw,rectangle, minimum width=1cm, minimum height=0.5cm] (Fbtm1) at (-7,-7.3) {$\F_{(b_t')}(\widetilde T^2)$};

\draw (btm1)--(Fbtm1);

\node[draw,rectangle, minimum width=1cm, minimum height=0.5cm] (Fbtp1) at (-7,-11.2) {$\F_{(b_{t+1}')}(\widetilde T^2)$};

\draw (btp1)--(Fbtp1);

\node[draw,rectangle, minimum width=1cm, minimum height=0.5cm] (Fbeta) at (-7,-14) {$\F_{(b_\beta')}(\widetilde T^2)$};

\draw (beta)--(Fbeta);

\node[draw,rectangle, minimum width=1cm, minimum height=0.5cm] (Fone) at (-7,-16.3) {$\F_{1}(\widetilde T^2)$};

\draw (one)--(Fone);

\end{tikzpicture}

\caption{The tree $\widetilde{T}^{4}$.}
\label{fig:T4_structure}
\end{figure}
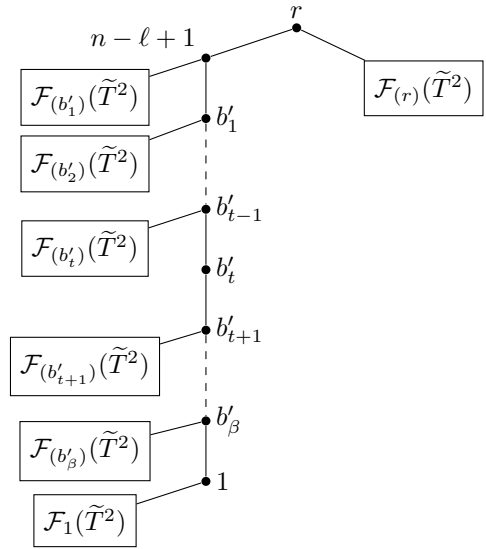

\noindent
Let $\overline{\F}_{b_{t+1}'}(\widetilde{T}^{4})=
b_{t+1}'\vee\F_{b_{t+1}'}(\widetilde{T}^{4})$ denote the rooted subtree consisting of the vertex $b_{t+1}'$ together with all of its descendants. 
Observe that $\overline{\F}_{b_{t+1}'}(\widetilde{T}^{4})=\overline{\F}_{b_{t+1}'}(\widetilde{T}^{2})$.
By removing $\overline{\F}_{b_{t+1}'}(\widetilde{T}^4)$ and attaching it to $b_{t-1}'$ in $\widetilde{T}^{4}$, we obtain $\widetilde{T}^{5}$, as illustrated in Figure~\ref{fig:T5_structure}.

\begin{figure}[ht]
\centering

\begin{subfigure}[t]{0.48\textwidth}
\centering
\begin{tikzpicture}[scale=0.7,
    every node/.style={font=\small}]


\node[circle,fill=black,inner sep=1.2pt] (r) at (0,0) {};
\node[above] at (r) {$r$};

\node[circle,fill=black,inner sep=1.2pt] (nl1) at (-2,-1) {};
\node[left] at (nl1) {$n-\ell+1$};

\node[circle,fill=black,inner sep=1.2pt] (b1) at (-2,-3) {};
\node[right] at (b1) {$b_1'$};

\node[circle,fill=black,inner sep=1.2pt] (btm1) at (-2,-6) {};
\node[right] at (btm1) {$b_{t-1}'$};

\node[circle,fill=black,inner sep=1.2pt] (bt) at (-2,-8) {};
\node[right] at (bt) {$b_t'$};

\draw (r)--(nl1)--(b1);
\draw[dashed] (b1)--(btm1);
\draw (btm1)--(bt);

\node[draw,rectangle, minimum width=1cm, minimum height=0.5cm] (Fr) at (3,-1.3) {$\F_{(r)}(\widetilde T^2)$};
\draw (r)--(Fr);

\node[draw,rectangle, minimum width=1cm, minimum height=0.5cm] (Fnl1) at (-4.5,-2.5) {$\F_{(b_1')}(\widetilde T^2)$};
\draw (nl1)--(Fnl1);

\node[draw,rectangle, minimum width=1cm, minimum height=0.5cm] (Fb1) at (-4.5,-4.5) {$\F_{(b_2')}(\widetilde T^2)$};
\draw (b1)--(Fb1);

\node[draw,rectangle, minimum width=1cm, minimum height=0.5cm] (Fbtm1) at (-4.5,-7) {$\F_{(b_t')}(\widetilde T^2)$};
\draw (btm1)--(Fbtm1);

\node[draw,rectangle, minimum width=1cm, minimum height=0.5cm] (oF) at (1,-7.5) {$\overline{\F}_{(b_{t+1}')}(\widetilde T^2)$};
\draw (btm1)--(oF);

\end{tikzpicture}
\caption{The tree $\widetilde T^5$.}
\end{subfigure}
\hfill
\begin{subfigure}[t]{0.48\textwidth}
\centering
\begin{tikzpicture}[scale=0.6,
    every node/.style={font=\small}]


\node[circle,fill=black,inner sep=1.2pt] (btp1) at (0,1) {};
\node[above] at (btp1) {$b_{t+1}'$};

\node[circle,fill=black,inner sep=1.2pt] (btp2) at (0,-1) {};
\node[right] at (btp2) {$b_{t+2}'$};

\node[circle,fill=black,inner sep=1.2pt] (beta) at (0,-4) {};
\node[right] at (beta) {$b_\beta'$};

\node[circle,fill=black,inner sep=1.2pt] (one) at (0,-6) {};
\node[right] at (one) {$1$};

\draw (btp1)--(btp2);
\draw[dashed] (btp2)--(beta);
\draw (beta)--(one);

\node[draw,rectangle, minimum width=1cm, minimum height=0.5cm] (Fbtp1) at (-3,0) {$\F_{(b_{t+1}')}(\widetilde T^2)$};
\draw (btp1)--(Fbtp1);

\node[draw,rectangle, minimum width=1cm, minimum height=0.5cm] (Fbtp2) at (-3,-2) {$\F_{(b_{t+2}')}(\widetilde T^2)$};
\draw (btp2)--(Fbtp2);

\node[draw,rectangle, minimum width=1cm, minimum height=0.5cm] (Fbeta) at (-3,-5) {$\F_{(b_\beta')}(\widetilde T^2)$};
\draw (beta)--(Fbeta);

\node[draw,rectangle, minimum width=1cm, minimum height=0.5cm] (Fone) at (-3,-7) {$\F_{1}(\widetilde T^2)$};
\draw (one)--(Fone);

\end{tikzpicture}
\caption{The rooted subtree
$\overline{\F}_{(b_{t+1}')}(\widetilde T^2)$.}
\end{subfigure}

\caption{The structure of the tree $\widetilde T^5$.}
\label{fig:T5_structure}

\end{figure}

\noindent
If $\son_{\widetilde{T}}(1)\cap F_{\ell}=\emptyset$, then $\widetilde{T}^5\in\Cnllminustwo$, and we define $\psi(\widetilde{T})=\widetilde{T}^5$. 
Otherwise, suppose that $\son_{\widetilde{T}}(1)\cap F_{\ell}\neq \emptyset$.
In this case, we consider $\widetilde{T}^6=\widetilde{T}^5\bigl[
\F_1(\widetilde{T}^2)\to b_t'\bigr]$,
obtained by attaching descendants of $1$ to the leaf $b_t'$ in $\widetilde{T}^{5}$.
Clearly, $\widetilde{T}^{6}$ satisfies the hypotheses of Proposition~\ref{prop:T*_in_Cnlp}. 
Thus, by applying Proposition~\ref{prop:T*_in_Cnlp}, we obtain a tree $(\widetilde{T}^6)^*\in\Cnllminustwo$, and we define $\psi(\widetilde{T})=(\widetilde{T}^6)^*$.

\medskip
\noindent 
$\textbf{Subcase III(b).}$ 
Let $p\in[t]=\{1,2,\dots , t\}$ be the least integer such that the largest son of $b_p'$ in $\F_{(b_p')}(\widetilde{T}^2)$ is a vertex $c>b_p'$. 
Proceeding as in Subcase~III(a), we construct $\widetilde{T}^3$.
The rooted subtree $\overline{\F}_{(b_p')}(\widetilde{T}^2) =b_p'\vee\F_{(b_p')}(\widetilde{T}^2)$ has the form illustrated in Figure~\ref{fig:subcaseIIIb_structure}.

\begin{figure}[ht]
\centering
\begin{tikzpicture}[scale=1,
    every node/.style={font=\small}]

\node[circle, fill=black, inner sep=1.2pt] (bp) at (2,0) {};
\node[above] at (bp) {$b_p'$};

\node[draw, rectangle, minimum width=1cm, minimum height=0.5cm]
(Fbp) at (2,-1.8) {$\F_{(b_p')}(\widetilde{T}^2)$};
\draw (bp)--(Fbp);

\node at (2,-3.6)
{$\overline{\F}_{(b_p')}(\widetilde{T}^2)$};

\node at (4,-1)
{$=$};

\node[circle, fill=black, inner sep=1.2pt] (r) at (8,0) {};
\node[above] at (r) {$b_p'$};

\node[circle, fill=black, inner sep=1.2pt] (c) at (6,-1.6) {};
\node[left] at (c) {$c$};

\node[draw, rectangle,
      minimum width=1cm,
      minimum height=0.5cm]
(Fc) at (6,-3.2)
{$\F_{c}(\widetilde{T}^2)$};

\draw (c)--(Fc);

\node[draw, rectangle,
      minimum width=1cm,
      minimum height=0.5cm]
(Fprime) at (9.3,-1.7)
{$\F'_{(b_p')}(\widetilde{T}^2)$};

\draw (r)--(c);
\draw (r)--(Fprime);


\node[
draw,
dashed,
rounded corners,
inner sep=0.35cm,
fit=(c) (Fc) (Fprime)
] (Fbox) {};

\node[right=0.25cm] at (Fbox.east)
{$\F_{(b_p')}(\widetilde{T}^2)$};

\end{tikzpicture}
\caption{Illustration of the rooted subtree $\overline{\F}_{(b_p')}(\widetilde{T}^2) =b_p'\vee\F_{(b_p')}(\widetilde{T}^2)$.} 
\label{fig:subcaseIIIb_structure}
\end{figure}

\noindent
In $\widetilde{T}^3$, we now remove the forests $\F_{(b_j')}(\widetilde{T}^2)$ for $1\leq j\leq p$. 
Then attach $\F_{(b_1')}(\widetilde{T}^2)$ to $n-\ell+1$, and, for each $1<j\leq p$, attach $\F_{(b_j')}(\widetilde{T}^2)$ to $b_{j-1}'$. 
Finally, remove $\F_{c}(\widetilde{T}^2)$ and attach it to $b_p'$, thereby obtaining $\widetilde{T}^7$ (see Figure~\ref{fig:T7_structure}).

\begin{figure}[ht]
\centering

\begin{subfigure}[t]{0.48\textwidth}
\centering
\begin{tikzpicture}[scale=0.75,
    every node/.style={font=\small}]


\node[circle,fill=black,inner sep=1.2pt] (r) at (0,0) {};
\node[above] at (r) {$r=n-\ell+2$};

\node[circle,fill=black,inner sep=1.2pt] (nl1) at (-1,-2) {};
\node[left] at (nl1) {$n-\ell+1$};

\node[circle,fill=black,inner sep=1.2pt] (b1) at (-1,-4) {};
\node[right] at (b1) {$b_1'$};

\node[circle,fill=black,inner sep=1.2pt] (bp2) at (-1,-7) {};
\node[right] at (bp2) {$b_{p-2}'$};

\node[circle,fill=black,inner sep=1.2pt] (bpm1) at (-1,-9) {};
\node[right] at (bpm1) {$b_{p-1}'$};

\node[circle,fill=black,inner sep=1.2pt] (c) at (-1,-11) {};
\node[right] at (c) {$c$};

\draw (r)--(nl1)--(b1);
\draw[dashed] (b1)--(bp2);
\draw (bp2)--(bpm1)--(c);

\node[draw,rectangle,minimum width=1cm,minimum height=0.5cm]
(Fr) at (3,-2)
{$\F_{(r)}(\widetilde T^2)$};
\draw (r)--(Fr);

\node[draw,rectangle,minimum width=1cm,minimum height=0.5cm]
(Fnl1) at (-3.5,-3)
{$\F_{(b_1')}(\widetilde T^2)$};
\draw (nl1)--(Fnl1);

\node[draw,rectangle,minimum width=1cm,minimum height=0.5cm]
(Fb1) at (-3.5,-5)
{$\F_{(b_2')}(\widetilde T^2)$};
\draw (b1)--(Fb1);

\node[draw,rectangle,minimum width=1cm,minimum height=0.5cm]
(Fp2) at (-3.5,-8)
{$\F_{(b_{p-1}')}(\widetilde T^2)$};
\draw (bp2)--(Fp2);

\node[draw,rectangle,minimum width=1cm,minimum height=0.5cm]
(Fprime) at (-3.5,-11)
{$\F'_{(b_p')}(\widetilde T^2)$};
\draw (bpm1)--(Fprime);

\node[draw,rectangle,minimum width=1cm,minimum height=0.5cm]
(Tbp) at (2,-11)
{$\widetilde T(b_p')$};
\draw (bpm1)--(Tbp);

\end{tikzpicture}
\caption{The tree $\widetilde T^7$.}
\end{subfigure}
\hfill
\begin{subfigure}[t]{0.48\textwidth}
\centering
\begin{tikzpicture}[scale=0.75,
    every node/.style={font=\small}]


\node[circle,fill=black,inner sep=1.2pt] (bp) at (0,1) {};
\node[above] at (bp) {$b_p'$};

\node[circle,fill=black,inner sep=1.2pt] (bpp1) at (0,-1) {};
\node[right] at (bpp1) {$b_{p+1}'$};

\node[circle,fill=black,inner sep=1.2pt] (beta) at (0,-4) {};
\node[right] at (beta) {$b_\beta'$};

\node[circle,fill=black,inner sep=1.2pt] (one) at (0,-6) {};
\node[right] at (one) {$1$};

\draw (bp)--(bpp1);
\draw[dashed] (bpp1)--(beta);
\draw (beta)--(one);

\node[draw,rectangle,minimum width=1cm,minimum height=0.5cm]
(Fc) at (-3,0)
{$\F_c(\widetilde T^2)$};
\draw (bp)--(Fc);

\node[draw,rectangle,minimum width=1cm,minimum height=0.5cm]
(Fbpp1) at (-3,-2)
{$\F_{(b_{p+1}')}(\widetilde T^2)$};
\draw (bpp1)--(Fbpp1);

\node[draw,rectangle,minimum width=1cm,minimum height=0.5cm]
(Fbeta) at (-3,-5)
{$\F_{(b_\beta')}(\widetilde T^2)$};
\draw (beta)--(Fbeta);

\node[draw,rectangle,minimum width=1cm,minimum height=0.5cm]
(F1) at (-3,-7)
{$\F_{1}(\widetilde T^2)$};
\draw (one)--(F1);

\end{tikzpicture}
\caption{The tree $\widetilde T(b_p')$.}
\end{subfigure}

\caption{The structure of the tree $\widetilde T^7$.}
\label{fig:T7_structure}

\end{figure}

\noindent
Observe that $\F_{1}(\widetilde{T}^{7})=\F_{1}(\widetilde{T}^{2})$.
If $\son_{\widetilde{T}^7}(1)\cap F_{\ell}=\emptyset$, then $\widetilde{T}^7\in\Cnllminustwo$, and we define $\psi(\widetilde{T})=\widetilde{T}^7$. 
Otherwise, if $\son_{\widetilde{T}^{7}}(1)\cap F_{\ell}\neq \emptyset$, then we consider the tree $\widetilde{T}^8=\widetilde{T}^7\bigl[\F_1(\widetilde{T}^7)\to c\bigr],$
obtained by attaching descendants of $1$ in $\widetilde{T}^7$ to the leaf $c$. 
Since $\widetilde{T}^8$ satisfies the hypotheses of Proposition~\ref{prop:T*_in_Cnlp}, we obtain $(\widetilde{T}^8)^*\in\Cnllminustwo$.
In this case, we define $\psi(\widetilde{T})=(\widetilde{T}^8)^*$.

\medskip
\noindent
We now characterize the trees $T\in \Cnllminustwo$ that arise from Case~III.
Let $T\in \Cnllminustwo$, and let $\Pmax(T)$  
\[
r=n-\ell+2=a_0\to a_1\to \cdots \to a_{\alpha},
\]
be the maximum path of $T$ in $F_{\ell}$ as in~\eqref{eq:Pmax_path}.
Suppose that $1$ is a descendant of $a_t$ for some $1\leq t<\alpha$ but is not a descendant of $a_{t+1}$.
Since $1$ is a descendant of $a_t$, there exists a unique path $Q'$ of the form
\begin{equation}\label{eq:QprimeT}
Q': \quad a_t=a_0'\to a_1'\to \cdots \to a_q'\to 1,
\qquad (q\geq 1),
\end{equation}
in $T$.
Recall that, for the maximum path $\Pmax(T)$, $\mathcal{F}_{a_i}(T)$ denotes the rooted forest consisting of all descendants of $a_i$ in $T$. Then $\mathcal{F}_{(a_i)}(T)$ denotes the rooted forest obtained from $\mathcal{F}_{a_i}(T)$ by removing the subtree rooted at $a_{i+1}$.
We divide the analysis into two subclasses.

\medskip
\noindent\textbf{Subclass (1).}
Suppose that the son of $a_t$ lying on the path $Q'$ does not belong to $F_{\ell}$, i.e., $a_1'\notin F_{\ell}$.
Then we consider the tree $T^{3}= T\bigl[\F_1(T)\to a_{\alpha}\bigr],$
obtained by attaching descendants of $1$ to the leaf $a_{\alpha}$ in $T$. 
Suppose that the rooted subtree $\overline{\F}_{(a_t)}(T^{3})=a_t\vee \mathcal{F}_{(a_t)}(T^{3})$ has the form illustrated in Figure~\ref{fig:subclass1_structure}.

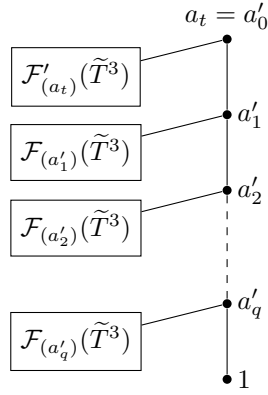
\begin{figure}[ht]
\centering
\begin{tikzpicture}[scale=0.5,
    every node/.style={font=\small}]

\node[circle,fill=black,inner sep=1.2pt] (at) at (0,0) {};
\node[above] at (at) {$a_t=a_0'$};

\node[circle,fill=black,inner sep=1.2pt] (a1) at (0,-2) {};
\node[right] at (a1) {$a_1'$};

\node[circle,fill=black,inner sep=1.2pt] (a2) at (0,-4) {};
\node[right] at (a2) {$a_2'$};

\node[circle,fill=black,inner sep=1.2pt] (aq) at (0,-7) {};
\node[right] at (aq) {$a_q'$};

\node[circle,fill=black,inner sep=1.2pt] (one) at (0,-9) {};
\node[right] at (one) {$1$};

\draw (at)--(a1)--(a2);
\draw[dashed] (a2)--(aq);
\draw (aq)--(one);

\node[draw,rectangle,
      minimum width=1cm,
      minimum height=0.5cm]
(Fat) at (-4,-1)
{$\F'_{(a_t)}(\widetilde T^3)$};
\draw (at)--(Fat);

\node[draw,rectangle,
      minimum width=1cm,
      minimum height=0.5cm]
(Fa1) at (-4,-3)
{$\F_{(a_1')}(\widetilde T^3)$};
\draw (a1)--(Fa1);

\node[draw,rectangle,
      minimum width=1cm,
      minimum height=0.5cm]
(Fa2) at (-4,-5)
{$\F_{(a_2')}(\widetilde T^3)$};
\draw (a2)--(Fa2);

\node[draw,rectangle,
      minimum width=1cm,
      minimum height=0.5cm]
(Faq) at (-4,-8)
{$\F_{(a_q')}(\widetilde T^3)$};
\draw (aq)--(Faq);

\end{tikzpicture}
\caption{Structure of the rooted subtree $\overline{\F}_{(a_t)}(T^{3})$ in $T^{3}$.}
\label{fig:subclass1_structure}
\end{figure}

\noindent
Let $\mathcal{F}_{a_{t+1}}(T^{3})$ denote the rooted forest consisting of all descendants of $a_{t+1}$ in $T^{3}$.
Now consider the tree $T^{4}= T^{3}\bigl[\F_{a_{t+1}}(T^3)\to 1\bigr],$
obtained by attaching all descendants of $a_{t+1}$ to the leaf $1$ in $T^{3}$. 
The tree $T^{4}$ has the form as shown in Figure~\ref{fig:T4_caseIII}.

\begin{figure}[ht]
\centering

\begin{subfigure}[t]{0.48\textwidth}
\centering
\begin{tikzpicture}[scale=0.66,
    every node/.style={font=\small}]


\node[circle,fill=black,inner sep=1.2pt] (r) at (0,0) {};
\node[above] at (r) {$r=n-\ell+2$};

\node[circle,fill=black,inner sep=1.2pt] (a1) at (-1,-2) {};
\node[left] at (a1) {$a_1=n-\ell+1$};

\node[circle,fill=black,inner sep=1.2pt] (a2) at (-1,-4) {};
\node[right] at (a2) {$a_2$};

\node[circle,fill=black,inner sep=1.2pt] (atm1) at (-1,-7) {};
\node[right] at (atm1) {$a_{t-1}$};

\node[circle,fill=black,inner sep=1.2pt] (at) at (-1,-9) {};
\node[right] at (at) {$a_t$};

\node[circle,fill=black,inner sep=1.2pt] (atp1) at (-1,-11) {};
\node[right] at (atp1) {$a_{t+1}$};

\draw (r)--(a1)--(a2);
\draw[dashed] (a2)--(atm1);
\draw (atm1)--(at)--(atp1);


\node[draw,rectangle,
minimum width=1cm,
minimum height=0.5cm]
(Fr) at (3,-2)
{$\F_{(r)}(T^3)$};
\draw (r)--(Fr);

\node[draw,rectangle,
minimum width=1cm,
minimum height=0.5cm]
(Fa1) at (-3,-3.5)
{$\F_{(a_1)}(T^3)$};
\draw (a1)--(Fa1);

\node[draw,rectangle,
minimum width=1cm,
minimum height=0.5cm]
(Fa2) at (-3,-5.5)
{$\F_{(a_2)}(T^3)$};
\draw (a2)--(Fa2);

\node[draw,rectangle,
minimum width=1cm,
minimum height=0.5cm]
(Fatm1) at (-3,-8.5)
{$\F_{(a_{t-1})}(T^3)$};
\draw (atm1)--(Fatm1);

\node[draw,rectangle,
minimum width=1cm,
minimum height=0.5cm]
(Fprime) at (-3,-10.5)
{$\F'_{(a_t)}(T^3)$};
\draw (at)--(Fprime);

\node[draw,rectangle,
minimum width=1cm,
minimum height=0.5cm]
(Fbar) at (2,-10.5)
{$\overline{\F}_{(a_1')}(\widetilde T^3)$};
\draw (at)--(Fbar);

\end{tikzpicture}
\caption{The tree $T^4$.}
\end{subfigure}
\hfill
\begin{subfigure}[t]{0.48\textwidth}
\centering
\begin{tikzpicture}[scale=0.55,
    every node/.style={font=\small}]


\node[circle,fill=black,inner sep=1.2pt] (a1) at (0,0) {};
\node[above] at (a1) {$a_1'$};

\node[circle,fill=black,inner sep=1.2pt] (a2) at (0,-2) {};
\node[right] at (a2) {$a_2'$};

\node[circle,fill=black,inner sep=1.2pt] (aq) at (0,-5) {};
\node[right] at (aq) {$a_q'$};

\node[circle,fill=black,inner sep=1.2pt] (one) at (0,-7) {};
\node[right] at (one) {$1$};

\draw (a1)--(a2);
\draw[dashed] (a2)--(aq);
\draw (aq)--(one);

\node[draw,rectangle,
minimum width=1cm,
minimum height=0.5cm]
(Fa1p) at (-3,-1)
{$\F_{(a_1')}(\widetilde T^3)$};
\draw (a1)--(Fa1p);

\node[draw,rectangle,
minimum width=1cm,
minimum height=0.5cm]
(Fa2p) at (-3,-3)
{$\F_{(a_2')}(\widetilde T^3)$};
\draw (a2)--(Fa2p);

\node[draw,rectangle,
minimum width=1cm,
minimum height=0.5cm]
(Faqp) at (-3,-6)
{$\F_{(a_q')}(\widetilde T^3)$};
\draw (aq)--(Faqp);

\node[draw,rectangle,
minimum width=1cm,
minimum height=0.5cm]
(F1) at (-3,-8)
{$\F_{a_{t+1}}(\widetilde T^3)$};
\draw (one)--(F1);

\end{tikzpicture}
\caption{The rooted subtree $\overline{\F}_{(a_1')}(\widetilde T^3)$.}
\end{subfigure}

\caption{Illustration of the structure of $T^4$.}
\label{fig:T4_caseIII}

\end{figure}
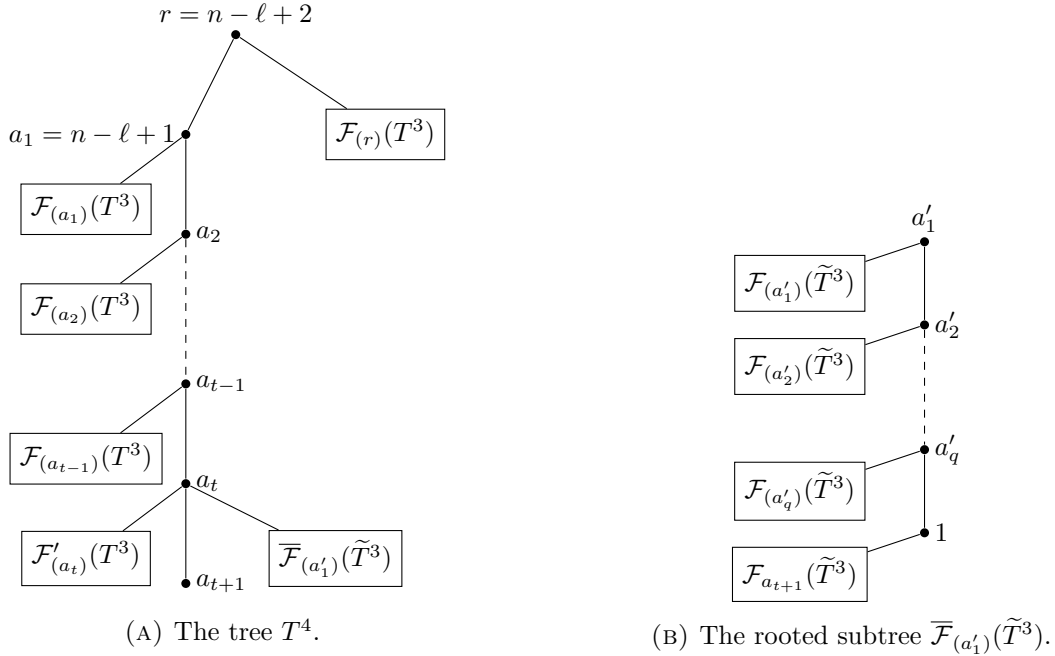

\noindent

Next, we remove the forests $\mathcal{F}'_{(a_t)}(T^{3})$, $\overline{\mathcal{F}}_{a_1'}(T^{3})$, and $\mathcal{F}_{(a_i)}(T^{3})$ for $1\leq i<t$ from $T^{4}$. Then attach $\mathcal{F}'_{(a_t)}(T^{3})$ and $\overline{\mathcal{F}}_{a_1'}(T^{3})$ to $a_{t+1}$. 
Then, for each $1\leq i<t$, attach $\mathcal{F}_{(a_i)}(T^{3})$ to $a_{i+1}$. The resulting tree $T^{5}$ is illustrated in Figure~\ref{fig:T5_caseIII}.

\begin{figure}[ht]
\centering
\begin{tikzpicture}[scale=0.6,
    every node/.style={font=\small}]

\node[circle,fill=black,inner sep=1.2pt] (r) at (0,0) {};
\node[above] at (r) {$r=n-\ell+2$};

\node[circle,fill=black,inner sep=1.2pt] (a1) at (-2,-2) {};
\node[left] at (a1) {$a_1=n-\ell+1$};

\node[circle,fill=black,inner sep=1.2pt] (a2) at (-2,-4) {};
\node[left] at (a2) {$a_2$};

\node[circle,fill=black,inner sep=1.2pt] (atm1) at (-2,-7) {};
\node[left] at (atm1) {$a_{t-1}$};

\node[circle,fill=black,inner sep=1.2pt] (at) at (-2,-9) {};
\node[left] at (at) {$a_t$};

\node[circle,fill=black,inner sep=1.2pt] (atp1) at (-2,-11) {};
\node[right] at (atp1) {$a_{t+1}$};

\draw (r)--(a1);
\draw (a1)--(a2);

\draw[dashed] (a2)--(atm1);

\draw (atm1)--(at);
\draw (at)--(atp1);

\node[draw,rectangle,minimum width=1cm, minimum height=0.5cm]
(Fr) at (3,-1.5)
{$\F_{(r)}(T^3)$};
\draw (r)--(Fr);

\node[draw,rectangle,minimum width=1cm, minimum height=0.5cm]
(Fa1) at (-5,-5.5)
{$\F_{(a_1)}(T^3)$};
\draw (a2)--(Fa1);

\node[draw,rectangle,minimum width=1cm, minimum height=0.5cm]
(Fatm2) at (-5,-8.5)
{$\F_{(a_{t-2})}(T^3)$};
\draw (atm1)--(Fatm2);

\node[draw,rectangle,minimum width=1cm, minimum height=0.5cm]
(Fatm1) at (-5,-10.5)
{$\F_{(a_{t-1})}(T^3)$};
\draw (at)--(Fatm1);

\node[draw,rectangle,minimum width=1cm, minimum height=0.5cm]
(Fprime) at (-6,-13)
{$\F'_{(a_t)}(T^3)$};
\draw (atp1)--(Fprime);

\node[draw,rectangle,minimum width=1cm, minimum height=0.5cm]
(Fbar) at (3,-13)
{$\overline{\F}_{(a'_1)}(\widetilde{T}^3)$};
\draw (atp1)--(Fbar);

\end{tikzpicture}
\caption{The tree $T^{5}$.}
\label{fig:T5_caseIII}
\end{figure}
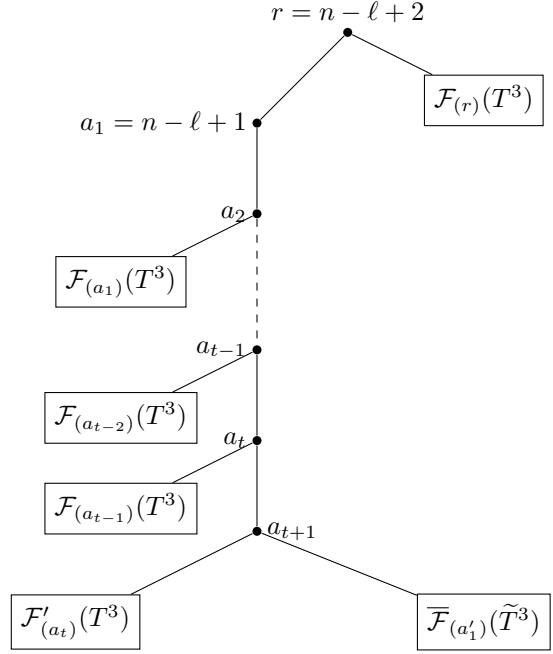

\noindent
Now, consider the rooted subtree $\overline{\mathcal{F}}_{a_2}(T^{5}),$ consisting of the vertex $a_2$ together with all its descendants in $T^{5}$. 
By removing $\overline{\mathcal{F}}_{a_2}(T^{5})$ from $T^{5}$ and attaching it to the root $r=n-\ell+2$ in $T^{5}$, we obtain $T^{6}$, illustrated in Figure~\ref{fig:T6_caseIII}.

\begin{figure}[ht]
\centering
\begin{tikzpicture}[scale=0.8,
    every node/.style={font=\small}]

\node[circle, fill=black, inner sep=1.2pt] (r) at (8,0) {};
\node[above] at (r) {$r=n-\ell+2$};

\node[circle, fill=black, inner sep=1.2pt] (b) at (4,-1.6) {};
\node[left] at (b) {$a_1=n-\ell+1$};

\node[draw, rectangle, minimum width=1cm, minimum height=0.5cm]
(F) at (8,-1.7) {$\F_{(r)}(T^3)$};
\draw (r)--(F);

\node[draw, rectangle, minimum width=1cm, minimum height=0.5cm]
(F) at (12,-1.7) {$\overline{\F}_{a_2}(T^5)$};

\draw (r)--(b);
\draw (r)--(F);

\end{tikzpicture}
\caption{ The tree $T^{6}$.} 
\label{fig:T6_caseIII}
\end{figure}
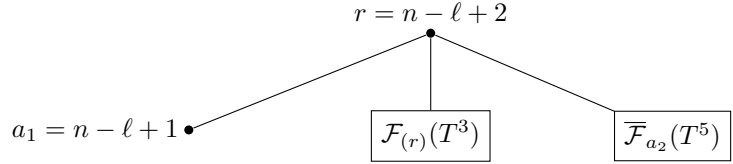

\noindent
Finally, by interchanging the vertices $a_2$ and $a_q'$ in $T^{6}$, we obtain a tree $T^{7}\in \Tnl$.
Reversing the above constructions, we see that $\psi(T^{7})=T$.

\medskip
\noindent\textbf{Subclass (2).}
Suppose that the son of $a_t=a_0'$ lying on the path $Q'$ belongs to $F_{\ell}$; that is, $a_1'\in F_{\ell}$.
Proceeding as in Subclass~(1), we first construct the tree $T^{4}$. From $T^{4}$, we then construct a tree $T^{8}$ of the form illustrated in Figure~\ref{fig:T8_caseIII}.

\begin{figure}[ht]
\centering
\begin{tikzpicture}[scale=0.6,
    every node/.style={font=\small}]

\node[circle,fill=black,inner sep=1.2pt] (r) at (0,0) {};
\node[above] at (r) {$r$};

\node[circle,fill=black,inner sep=1.2pt] (a1) at (-2,-2) {};
\node[left] at (a1) {$n-\ell+1$};

\node[circle,fill=black,inner sep=1.2pt] (a2) at (-2,-4) {};
\node[left] at (a2) {$a_2$};

\node[circle,fill=black,inner sep=1.2pt] (am) at (-2,-7) {};
\node[left] at (am) {$a_{t-1}$};

\node[circle,fill=black,inner sep=1.2pt] (at) at (-2,-9) {};
\node[left] at (at) {$a_t$};

\node[circle,fill=black,inner sep=1.2pt] (ap1) at (-2,-11) {};
\node[left] at (ap1) {$a_1'$};

\node[circle,fill=black,inner sep=1.2pt] (ap2) at (-4,-13) {};
\node[right] at (ap2) {$a_2'$};

\node[circle,fill=black,inner sep=1.2pt] (apq) at (-4,-16) {};
\node[right] at (apq) {$a_q'$};

\node[circle,fill=black,inner sep=1.2pt] (one) at (-4,-18) {};
\node[right] at (one) {$1$};

\node[circle,fill=black,inner sep=1.2pt] (atp1) at (-0.9,-13.3) {};
\node[right] at (atp1) {$a_{t+1}$};

\draw (r)--(a1);
\draw (a1)--(a2);

\draw[dashed] (a2)--(am);

\draw (am)--(at)--(ap1);

\draw (ap1)--(ap2);
\draw[dashed] (ap2)--(apq);
\draw (apq)--(one);

\draw (ap1)--(atp1);

\node[draw,rectangle,minimum width=1cm, minimum height=0.5cm]
(Fr) at (2.5,-1.5)
{$\F_{(r)}(T^3)$};
\draw (r)--(Fr);

\node[draw,rectangle,minimum width=1cm, minimum height=0.5cm]
(Fa1) at (0.5,-5)
{$\F_{(a_1)}(T^3)$};
\draw (a2)--(Fa1);

\node[draw,rectangle,minimum width=1cm, minimum height=0.5cm]
(Fam) at (0.5,-8)
{$\F_{(a_{t-2})}(T^3)$};
\draw (am)--(Fam);

\node[draw,rectangle,minimum width=1cm, minimum height=0.5cm]
(Fat) at (0.5,-10)
{$\F_{(a_{t-1})}(T^3)$};
\draw (at)--(Fat);

\node[draw,rectangle,minimum width=1cm, minimum height=0.5cm]
(Fprime) at (0.5,-12)
{$\F'_{(a_t)}(T^3)$};
\draw (ap1)--(Fprime);

\node[draw,rectangle,minimum width=1cm, minimum height=0.5cm]
(Fa2p) at (-6.5,-14)
{$\F_{(a_2')}(T^3)$};
\draw (ap2)--(Fa2p);

\node[draw,rectangle,minimum width=1cm, minimum height=0.5cm]
(Faq) at (-6.5,-17)
{$\F_{(a_q')}(T^3)$};
\draw (apq)--(Faq);

\node[draw,rectangle,minimum width=1cm, minimum height=0.5cm]
(Fone) at (-6.5,-19)
{$\F_{a_{t+1}}(T^3)$};
\draw (one)--(Fone);

\node[draw,rectangle,minimum width=1cm, minimum height=0.5cm]
(Fatp1) at (2.3,-14.3)
{$\F_{(a_1')}(T^3)$};
\draw (atp1)--(Fatp1);

\end{tikzpicture}
\caption{The tree $T^{8}$.}
\label{fig:T8_caseIII}
\end{figure}
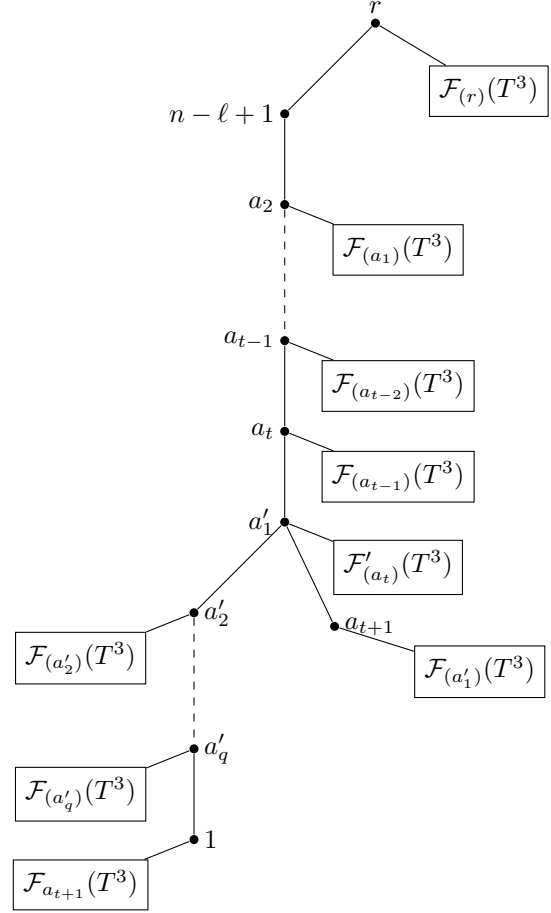

\noindent
Let $\overline{\mathcal{F}}_{a_2}(T^{8})$ denote the rooted subtree consisting of the vertex $a_2$ together with all its descendants in $T^{8}$. By removing $\overline{\mathcal{F}}_{a_2}(T^{8})$ from $T^{8}$ and attaching it to the root $r=n-\ell+2$, we obtain a tree $T^{9}$ from $T^{8}$.
Let $T^{10}$ be the tree obtained from $T^{9}$ by interchanging the vertices $a_2$ and $a_q'$. 
Clearly, $T^{10}\in \Tnl$.
Reversing the above constructions, we obtain $\psi(T^{10})=T.$

\medskip
\noindent
This shows that the map $\psi\colon \Tnl\longrightarrow \Cnllminustwo$ is well-defined and surjective.
We now briefly characterize the trees $T\in \Cnllminustwo$ arising in the above three cases. 
Let $\Pmax(T)\colon
n-\ell+2=a_0\to a_1\to \cdots \to a_{\alpha}$ be the maximum path of $T$ in $F_{\ell}$ as in \eqref{eq:Pmax_path}.

\begin{enumerate}
    \item If $\alpha=1$, then $T$ arises from Case~I.
    \item If $\alpha>1$ and $1$ is a descendant of $a_1=n-\ell+1$, then $T$ arises from Case~II.
    \item Suppose $\alpha>1$ and there exists $1\leq t<\alpha$ such that $1$ is a descendant of $a_t$ but not of $a_{t+1}$, then $T$ arises from Case~III. 
    Let $Q'\colon a_t=a_0'\to a_1'\to\cdots\to a_q'\to 1$ be the unique path from $a_t$ to $1$. 
    Then:
    \begin{itemize}
        \item if $a_1'\notin F_{\ell}$, then $T$ arises from 
        Subcase~III(a),
        \item if $a_1'\in F_{\ell}$, then $T$ arises from 
        Subcase~III(b).
    \end{itemize}
\end{enumerate}

Since the cases are mutually exclusive and exhaustive, and each 
$T\in\Cnllminustwo$ arises from exactly one $\widetilde{T}\in\Tnl$, 
the map $\psi$ is a bijection.
Therefore, $|\Cnllminustwo|= |\Tnl|$, and the required formula follows from Proposition~\ref{prop:Tnl_count}.
\end{proof}

Proposition~\ref{prop:Tnl_bijection} determines the cardinality of 
$\Cnllminustwo$. 
We now turn to the enumeration of $\Cnlp$ for $0\leq p<\ell-2$. 
For this purpose, we consider the sets $\Cnlminusonep$ and $\Cnpplustwop$.

By Proposition~\ref{prop:Cnlp_characterization}, a tree $T\in\Cnlminusonep$ if and only if $T\in\Un^{1\not\sim F_{\ell-1}}$ and conditions \ref{cond1} and \ref{cond2} are satisfied with 
$F_{\ell-1}=\{n-\ell+2,n-\ell+3,\ldots,n\}$.
In particular, the vertex $1$ is not adjacent to any vertex of $F_{\ell-1}$, but it is still possible that $1$ is adjacent to the vertex $n-\ell+1$ in $T$. 
Now, we define the subset
\[
\Tnlminusonep=\{T\in\Cnlminusonep : \Par_T(1)\notin F_{\ell}\}.
\]
Since $T\in\Cnlminusonep$ implies $\Par_T(1)\notin F_{\ell-1}$, the only possibility excluded by the condition $\Par_T(1)\notin F_{\ell}$ is $\Par_T(1)=n-\ell+1$. 
Thus, $\Tnlminusonep\subseteq\Cnlminusonep$ consists precisely of those trees $T$ for which $\Par_T(1)\neq n-\ell+1$.
Note that the vertex $n-\ell+1$ may still be a son of $1$ in such a tree.

Similarly, a tree $T\in\Cnpplustwop$ if and only if $T\in\mathcal{U}_n^{1\not\sim F_{p+2}}$ and conditions \ref{cond1} and \ref{cond2} of Proposition~\ref{prop:Cnlp_characterization} hold, with $F_{p+2}=\{n-p-1, n-p, \ldots,n\}$.
So, $\Par_T(1)\notin F_{p+2}$. 
We define
\[
\Tnpplustwop=\{T\in\Cnpplustwop : \Par_T(1)\notin F_{\ell}\}.
\]

For each $T\in\Tnpplustwop$, let $\widetilde{T}$ be the tree obtained from $T$ by interchanging the vertices $n-p-1$ and $n-\ell+1$ in $T$. 
We define
\[
\TildeTnpplustwop=\{\widetilde{T} : T\in\Tnpplustwop\}.
\]

Since $T\mapsto\widetilde{T}$ is a bijection from $\Tnpplustwop$ to $\TildeTnpplustwop$, we have $|\Tnpplustwop|=|\TildeTnpplustwop|$. 
Moreover, $\Tnlminusonep\cap\TildeTnpplustwop=\emptyset$.
Indeed, in any tree in $\Tnlminusonep$, the vertex $n-\ell+1$ is never a son of the root $n-p$. 
On the other hand, in every tree $\widetilde{T}\in\TildeTnpplustwop$, the vertex $n-\ell+1$ is a son of the root $n-p$. 
This follows from the fact that $\widetilde{T}$ is obtained by interchanging the vertices $n-p-1$ and $n-\ell+1$ in a tree of $\Tnpplustwop$, where $n-p-1$ is the largest son of the root $n-p$. Therefore, no tree can belong to both $\Tnlminusonep$ and $\TildeTnpplustwop$.

Now, we shall see that these subsets play a key role in the enumeration of $\Cnlp$.

\begin{proposition}\label{prop:Phi_p}
For $0\leq p<\ell-2$, there exists a bijection
\[
\Phi_p\colon \Tnlminusonep\coprod \TildeTnpplustwop \longrightarrow \Cnlp.
\]
Consequently, $|\Cnlp|=|\Tnlminusonep|+|\Tnpplustwop|.$
\end{proposition}

\begin{proof}
We define the map $\Phi_p$ separately on the two disjoint subsets $\Tnlminusonep$ and  $\TildeTnpplustwop$, and characterize the image of each part together with the corresponding inverse construction.

\medskip
\noindent\textbf{The map on $\Tnlminusonep$.}
Let $\widetilde{T}\in \Tnlminusonep$.
Suppose that the maximum path $\Pmaximum^{\ell-1}(\widetilde{T})$ of $\widetilde{T}$ in $F_{\ell-1}$ is of the form
\[
\Pmaximum^{\ell-1}(\widetilde{T})\colon \quad 
r=n-p\to a_1'\to \cdots \to a_{\alpha-1}'\to a_{\alpha}', \qquad (\alpha\geq 1),
\]
where $a_i'\in F_{\ell-1}$.
Since $\widetilde{T}\in \Tnlminusonep$, we have $\Par_{\widetilde{T}}(1)\notin F_{\ell}$.

If $\son_{\widetilde{T}}(1)\cap F_{\ell}=\emptyset$, then $\widetilde{T}\in \Cnlp$, and we define $\Phi_p(\widetilde{T})=\widetilde{T}$.
Now suppose that $\son_{\widetilde{T}}(1)\cap F_{\ell}\neq \emptyset$.
That is, $\son_{\widetilde{T}}(1)\cap F_{\ell}=\{n-\ell+1\}$.
Consider the tree $\widetilde{T}'= \widetilde{T}\bigl[\F_{1}(\widetilde{T})\to a_{\alpha}'\bigr],$
obtained by attaching the descendants of $1$ to the leaf $a_{\alpha}'$  in $\widetilde{T}$.
Clearly, $\widetilde{T}'$ satisfies the hypotheses of Proposition~\ref{prop:T*_in_Cnlp}. 
Moreover, the maximum path $\Pmax(\widetilde{T}')$ in $F_{\ell}$ is of the form
\[
\Pmax(\widetilde{T}')\colon \quad r=n-p\to a_1'\to \cdots\to a_{\alpha}'\to a_{\alpha+1}'\to \cdots\to a_{\alpha+t}', \qquad (t\geq 1),
\]
where $a_{\alpha+1}'=n-\ell+1$.
Applying Proposition~\ref{prop:T*_in_Cnlp}, we obtain a tree $(\widetilde{T}')^*\in \Cnlp$.
We define $\Phi_p(\widetilde{T})=(\widetilde{T}')^*$.

In fact, we shall show that $\Phi_p(\Tnlminusonep)$ consists precisely of those trees $T\in \Cnlp$ such that the maximum path $\Pmax(T)$ is of the form
\[
\Pmax(T)\colon \quad
r=n-p\to a_1\to a_2\to \cdots \to a_{\alpha-1}\to a_{\alpha},
\qquad (\alpha\geq 1),
\]
where $a_i\in F_{\ell}$ for all $1\leq i\leq \alpha$, $a_{\alpha}$ is a leaf and $n-\ell+1<a_1<r=n-p$.
Moreover, either $a_i\in F_{\ell-1}$ for all $1\leq i\leq \alpha$, or there exists an index $t$ with $1<t\leq \alpha$ such that $a_t=n-\ell+1$ and the vertex $1$ is not a descendant of $\Par_T(a_t)$ in $T$.

In the former case, we clearly have $T\in \Tnlminusonep$ and $\Phi_p(T)=T.$
In the latter case, consider the tree $T'= T\bigl[\F_{1}(T)\to a_{\alpha}\bigr],$
obtained by attaching all descendants of $1$ to the leaf $a_{\alpha}$ in $T$. 
Let $\overline{\mathcal{F}}_{a_t}(T')=a_t\vee \mathcal{F}_{a_t}(T')$ denote the rooted subtree consisting of the vertex $a_t$ together with all its descendants in $T'$.
We then construct a tree $T^2$ by removing $\overline{\mathcal{F}}_{a_t}(T')$ from $T'$ and attaching it to the leaf $1$.
By construction, $T^2\in \Tnlminusonep$, and $\son_{T^2}(1)\cap F_{\ell}=\{n-\ell+1\}$.
Reversing the above constructions, we recover $T$ from $T^2$, and therefore $\Phi_p(T^2)=T$.
\medskip

\noindent\textbf{The map on $\TildeTnpplustwop$.}
Let $\widetilde{T}\in\TildeTnpplustwop$. 
Recall that every tree in $\TildeTnpplustwop$ is obtained from a tree in $\Tnpplustwop$ by interchanging the vertices $n-p-1$ and $n-\ell+1$.
Let
\[
\Pmaximum^{p+2}(T)\colon \quad
r=n-p\to \widetilde{a}_1\to \widetilde{a}_2\to \cdots \to \widetilde{a}_t,
\qquad (t\geq 1),
\]
be the maximum path of the corresponding tree $T\in \Tnpplustwop$ in $F_{p+2}$, where $\widetilde{a}_i\in F_{p+2}$, $\widetilde{a}_1=n-p-1$, and $\widetilde{a}_t$ is a leaf.
By abuse of notation, we write
\begin{equation}\label{eq:path_p+2}
\Pmaximum^{p+2}(\widetilde{T})\colon\quad
r=n-p\to a_1''\to a_2''\to \cdots \to a_t'',   
\end{equation}
where $a_1''=n-\ell+1$, $a_j''=\widetilde{a}_j\in F_{p+2}\setminus\{n-p-1\}$ $(2\leq j\leq t)$, and $a_t''=\widetilde{a}_t$ is a leaf of $\widetilde{T}$.
Note that $\Pmaximum^{p+2}(\widetilde{T})$ as defined in~\eqref{eq:path_p+2} need not be the maximum path of $\widetilde{T}$ in $F_{p+2}$. 
However, throughout this subsection, the notation $\Pmaximum^{p+2}(\widetilde{T})$ will always refer to the path obtained from $\Pmaximum^{p+2}(T)$ by interchanging the vertices $n-p-1$ and $n-\ell+1$.
We consider the following cases.

\medskip
\noindent\textbf{Case (i).}
Suppose that the vertex $1$ is not adjacent to any vertex of $F_{\ell}$ in $\widetilde{T}$, and that $n-\ell+1$ is the largest son of the root $r=n-p$.
Then $\widetilde{T}\in \Cnlp$, and we define $\Phi_p(\widetilde{T})=\widetilde{T}$.

\medskip
\noindent\textbf{Case (ii).}
Suppose that $n-\ell+1$ is the largest son of the root $r=n-p$ in $\widetilde{T}$, but $\son_{\widetilde{T}}(1)\cap F_{\ell}\neq\emptyset$. 
Recall that, by abuse of notation, $\Pmaximum^{p+2}(\widetilde{T})$ denotes the path obtained from the maximum path of the corresponding tree in $\Tnpplustwop$. 
In the present case, since $n-\ell+1$ is the largest son of the root $r=n-p$ in $\widetilde{T}$, this path is in fact the maximum path of $\widetilde{T}$ in $F_{p+2}$.
Thus,
\[
\Pmaximum^{p+2}(\widetilde{T})\colon\quad
r=n-p\to a_1''\to a_2''\to\cdots\to a_t'',
\qquad (t\geq 1),
\]
where $a_1''=n-\ell+1$, $a_i''\in F_{p+2}\setminus\{n-p-1\}$ for $2\leq i\leq t$, and $a_t''$ is a leaf of $\widetilde{T}$. 
Consider the tree  $\widetilde{T}^2=\widetilde{T}\bigl[\F_1(\widetilde{T})\to a_t''\bigr]$
obtained by attaching all descendants of $1$ to the leaf $a_t''$ in $\widetilde{T}$. 
Then the maximum path $\Pmax(\widetilde{T}^2)$ in $F_{\ell}$ has the form
\[
\Pmax(\widetilde{T}^2)\colon \quad 
r=n-p\to a_1''\to a_2''\to\cdots\to a_t''\to a_{t+1}''
\to\cdots\to a_{\alpha}'',
\qquad (\alpha>t).
\]
Clearly, the tree $\widetilde{T}^2$ satisfies the hypotheses of 
Proposition~\ref{prop:T*_in_Cnlp}, and applying it yields a tree $(\widetilde{T}^2)^*=\widetilde{T}^2\bigl[
\F_{a_{\alpha}''}(\widetilde{T}^2)\to 1\bigr]\in\Cnlp.$
Here, we define $\Phi_p(\widetilde{T})=(\widetilde{T}^2)^*$.

\medskip

We now identify the trees in $\Cnlp$ arising from Case~(ii). 
Let $T\in\Cnlp$ whose maximum path
\[
\Pmax(T)\colon \quad   r=n-p\to a_1\to a_2\to\cdots\to a_{\alpha}, \qquad (a_i\in F_{\ell}),
\]
contains an index $t$ with $1<t<\alpha$ such that $a_1=n-\ell+1$, 
$a_i\in F_{p+2}\setminus\{n-p-1\}$ for $1<i\leq t$, 
$a_{t+1}\notin F_{p+2}\setminus\{n-p-1\}$, and 
vertex $1$ is not a descendant of $a_t$. 
Consider the tree 
\[
T^3=T\bigl[\F_1(T)\to a_{\alpha}\bigr],
\]
obtained by attaching all descendants of $1$ to the leaf $a_{\alpha}$ in $T$. 
Next, consider the tree
\[
T^4=T^3\bigl[\F_{a_{t+1}}(T^3)\to 1\bigr].
\]
Then $T^{4}\in \Tnpplustwop$, and reversing the above construction gives $\Phi_p(T^{4})=T$.

\medskip
\noindent\textbf{Case (iii).}
Suppose that the largest son of the root $r=n-p$ in $\widetilde{T}$ is a vertex $b\neq n-\ell+1$.
Since $\widetilde{T}\in\TildeTnpplustwop$ is obtained from a tree in $\Tnpplustwop$ by interchanging the vertices $n-p-1$ and $n-\ell+1$, it follows that $n-\ell+1<b<n-p$.
Let $\Pmax(\widetilde{T})$ be the maximum path of $\widetilde{T}$ in $F_{\ell}$. 
Then
\[
\Pmax(\widetilde{T})\colon\quad
r=n-p\to b=b_1\to b_2\to \cdots \to b_{\beta-1}\to b_{\beta},
\qquad (\beta\geq 1),
\]
where $b_i\in F_{\ell}$ for all $1\leq i\leq \beta$.
We divide the analysis into the following three subcases.

\medskip
\noindent\textbf{Subcase (iiia).}
Suppose that the vertex $1$ is a descendant of $b_{\beta}$ in $\widetilde{T}$.
Recall that
\[
\Pmaximum^{p+2}(\widetilde{T})\colon \quad 
r=n-p\to a_1''\to a_2''\to \cdots \to a_t'',
\]
where $a_1''=n-\ell+1$, $a_i''\in F_{p+2}\setminus\{n-p-1\}$ and $a_t''$ is a leaf.
First, construct the tree $\widetilde{T}^{2}=\widetilde{T}
\bigl[\mathcal{F}_1(\widetilde{T})\to a_t''\bigr]$, by attaching the descendants of $1$ to the leaf $a_t''$ in $\widetilde{T}$.
Let $\overline{\mathcal{F}}_{n-\ell+1}(\widetilde{T}^{2})=(n-\ell+1)\vee\mathcal{F}_{n-\ell+1}(\widetilde{T}^{2})$ denote the rooted subtree consisting of the vertex $n-\ell+1$ together with all its descendants in $\widetilde{T}^{2}$. 
We obtain a tree $\widetilde{T}^{3}$ by removing $\overline{\mathcal{F}}_{n-\ell+1}(\widetilde{T}^{2})$ from $\widetilde{T}^{2}$ and attaching it to the vertex $b_{\beta}$.
Suppose that the maximum path $\Pmax(\widetilde{T}^{3})$ of $\widetilde{T}^{3}$ in $F_{\ell}$ is of the form
\[
\Pmax(\widetilde{T}^{3})\colon  \quad r=n-p\to b_1\to b_2\to \cdots\to b_{\beta}\to b_{\beta+1}\to \cdots\to b_{\beta+t},\qquad (t\geq 1),
\]
where $b_{\beta+1}=n-\ell+1$.
Then $\widetilde{T}^{3}$ satisfies the hypotheses of Proposition~\ref{prop:T*_in_Cnlp}. 
By applying Proposition~\ref{prop:T*_in_Cnlp}, we obtain the tree $(\widetilde{T}^{3})^*=\widetilde{T}^{3}\bigl[\mathcal{F}_{(b_{\beta+t})}(\widetilde{T}^{3})\to 1\bigr]\in \Cnlp$.
We define $\Phi_p(\widetilde{T})=(\widetilde{T}^{3})^*$.

\medskip
\noindent\textbf{Subcase (iiib).}
Suppose that the vertex $1$ is a descendant of $n-\ell+1$ in $\widetilde{T}$. Proceeding exactly as in Subcase~(iiia), we construct the trees $\widetilde{T}^{2}$, $\widetilde{T}^{3}$ and obtain a tree $(\widetilde{T}^{3})^*\in \Cnlp$.
We then define $\Phi_p(\widetilde{T})=(\widetilde{T}^{3})^*$.

\medskip
The images arising from Subcases~(iiia) and~(iiib) can be distinguished by the position of the vertex $1$. 
In Subcase~(iiia), the vertex $1$ is a descendant of $b_{\beta}$ but not of $n-\ell+1$ in $\Phi_p(\widetilde{T})$. 
On the other hand, in Subcase~(iiib), the vertex $1$ is a descendant of $n-\ell+1$ in $\Phi_p(\widetilde{T})$.

\medskip
\noindent\textbf{Subcase (iiic).}
Suppose that the vertex $1$ is neither a descendant of $n-\ell+1$ nor a descendant of $b_{\beta}$ in $\widetilde{T}$.
Consider the rooted forest $\mathcal{F}_1(\widetilde{T})$ consisting of all descendants of $1$ in $\widetilde{T}$, and let
$r\vee \mathcal{F}_1(\widetilde{T})$
be the rooted tree obtained by attaching $\mathcal{F}_1(\widetilde{T})$ to the root $r=n-p$.
Suppose that the maximum path $\Pmax(r\vee \mathcal{F}_1(\widetilde{T}))$ of $r\vee \mathcal{F}_1(\widetilde{T})$ in $F_{\ell}$ is
\[
r=n-p=c_0\to c_1\to c_2\to \cdots \to c_{\gamma},
\qquad (\gamma\geq 0),
\]
where $c_i\in F_{\ell}$ $(0\leq i\leq \gamma)$.
Note that if $1$ is a leaf of $\widetilde{T}$, or more generally if $\son_{\widetilde{T}}(1)\cap F_{\ell}=\emptyset$, then $\Pmax(r\vee \mathcal{F}_1(\widetilde{T}))=\{r\}$, that is, $\gamma=0$.
We construct a tree $\widetilde{T}^{4}$ from $\widetilde{T}$ by removing the forest $\mathcal{F}_1(\widetilde{T})$ from vertex $1$ and attaching it to the vertex $r$.
Let $\mathcal{F}'_{(r)}(\widetilde{T}^{4})$ denote the rooted forest obtained from the descendants of the root $r$ in $\widetilde{T}^{4}$ by removing the rooted subtree $\overline{\mathcal{F}}_{n-\ell+1}(\widetilde{T}^{4})$ and the rooted forest $\mathcal{F}_1(\widetilde{T})$,
which is now attached to $r$. 
The structure of $\widetilde{T}^{4}$ is illustrated in Figure~\ref{fig:subcaseIIIC_T4}.

\begin{figure}[ht]
\centering


\begin{subfigure}[t]{0.45\textwidth}
\centering
\begin{tikzpicture}[scale=0.8,
    every node/.style={font=\small}]

\node[circle,fill=black,inner sep=1.2pt] (r) at (0,3) {};
\node[above] at (r) {$r=n-p$};

\node[circle,fill=black,inner sep=1.2pt] (v) at (-3,1) {};
\node[left] at (v) {$n-\ell+1$};

\draw (r)--(v);

\node[draw,rectangle,minimum width=1cm,minimum height=0.5cm]
(Fr) at (0,0)
{$\F'_{(r)}(\widetilde T^4)$};

\draw (r)--(Fr);

\node[draw,rectangle,minimum width=1cm,minimum height=0.5cm]
(Fv) at (-3,-1)
{$\F_{n-\ell+1}(\widetilde T^4)$};

\draw (v)--(Fv);

\node[draw,rectangle,minimum width=1cm,minimum height=0.5cm]
(F1) at (3.5,1)
{$\F_1(\widetilde T)$};

\draw (r)--(F1);

\end{tikzpicture}
\caption{The tree $\widetilde T^4$ obtained from $\widetilde T$.}
\end{subfigure}

\vspace{0.5cm}

\begin{subfigure}[t]{0.45\textwidth}
\centering

\begin{tikzpicture}[scale=0.7,
    every node/.style={font=\small}]

\node[circle,fill=black,inner sep=1.2pt] (r) at (0,1) {};
\node[above] at (r) {$c_0=r=n-p$};

\node[circle,fill=black,inner sep=1.2pt] (c1) at (0,-2) {};
\node[right] at (c1) {$c_1$};

\node[circle,fill=black,inner sep=1.2pt] (c2) at (0,-4) {};
\node[right] at (c2) {$c_2$};

\node[circle,fill=black,inner sep=1.2pt] (cg) at (0,-7) {};
\node[right] at (cg) {$c_\gamma$};

\draw (r)--(c1)--(c2);
\draw[dashed] (c2)--(cg);

\node[draw,rectangle,
minimum width=1cm,
minimum height=0.5cm]
(Froot) at (-3,-1)
{$\F_{(1)}(\widetilde T)$};

\draw (r)--(Froot);

\node[draw,rectangle,
minimum width=1cm,
minimum height=0.5cm]
(Fc1) at (-3,-3)
{$\F_{(c_1)}(\widetilde T)$};

\draw (c1)--(Fc1);

\node[draw,rectangle,
minimum width=1cm,
minimum height=0.5cm]
(Fc2) at (-3,-5)
{$\F_{(c_2)}(\widetilde T)$};

\draw (c2)--(Fc2);

\node[draw,rectangle,
minimum width=1cm,
minimum height=0.5cm]
(Fcg) at (-3,-8)
{$\F_{(c_\gamma)}(\widetilde T)$};

\draw (cg)--(Fcg);


\node[
draw,
dashed,
rounded corners,
inner sep=0.6cm,
fit=(c1) (cg) (Froot) (Fc1) (Fc2) (Fcg)
] (Fonebox) {};

\node[right=0.5cm] at (Fonebox.east)
{$\F_1(\widetilde T)$};

\end{tikzpicture}
\caption{The rooted tree $r\vee \F_1(\widetilde T)$.}
\end{subfigure}
\hfill
\begin{subfigure}[t]{0.45\textwidth}
\centering

\begin{tikzpicture}[scale=0.7,
    every node/.style={font=\small}]

\node[circle,fill=black,inner sep=1.2pt] (r) at (0,1) {};
\node[above] at (r) {$r=n-p$};

\node[circle,fill=black,inner sep=1.2pt] (b1) at (0,-2) {};
\node[right] at (b1) {$b_1$};

\node[circle,fill=black,inner sep=1.2pt] (b2) at (0,-4) {};
\node[right] at (b2) {$b_2$};

\node[circle,fill=black,inner sep=1.2pt] (bbeta) at (0,-7) {};
\node[right] at (bbeta) {$b_\beta$};

\draw (r)--(b1)--(b2);
\draw[dashed] (b2)--(bbeta);

\node[draw,rectangle,
minimum width=1cm,
minimum height=0.5cm]
(Froot) at (-2.3,-1)
{$\F''_{(r)}(\widetilde T)$};

\draw (r)--(Froot);

\node[draw,rectangle,
minimum width=1cm,
minimum height=0.5cm]
(Fb1) at (-2.3,-4)
{$\F_{(b_1)}(\widetilde T)$};

\draw (b1)--(Fb1);

\node[draw,rectangle,
minimum width=1cm,
minimum height=0.5cm]
(Fb2) at (-2.3,-6.3)
{$\F_{(b_2)}(\widetilde T)$};

\draw (b2)--(Fb2);

\node[draw,rectangle,
minimum width=1cm,
minimum height=0.5cm]
(Fbbeta) at (-2.3,-8)
{$\F_{(b_\beta)}(\widetilde T)$};

\draw (bbeta)--(Fbbeta);







\node[
draw,
dashed,
rounded corners,
inner sep=0.6cm,
fit=(b1) (bbeta) (Froot) (Fb1) (Fb2) (Fbbeta) 
] (Fprimebox) {};

\node[right=0.5cm] at (Fprimebox.east)
{$\F'_{(r)}(\widetilde T^4)$};

\end{tikzpicture}
\caption{The rooted tree $r\vee \F'_{(r)}(\widetilde T^4)$ where $1$ is a leaf which is not a descendant of $b_{\beta}$.}
\end{subfigure}

\caption{Structure of the tree $\widetilde T^4$.}
\label{fig:subcaseIIIC_T4}

\end{figure}
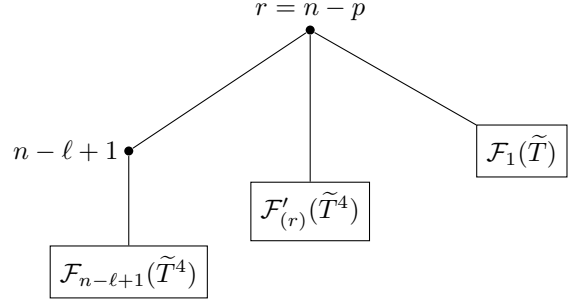
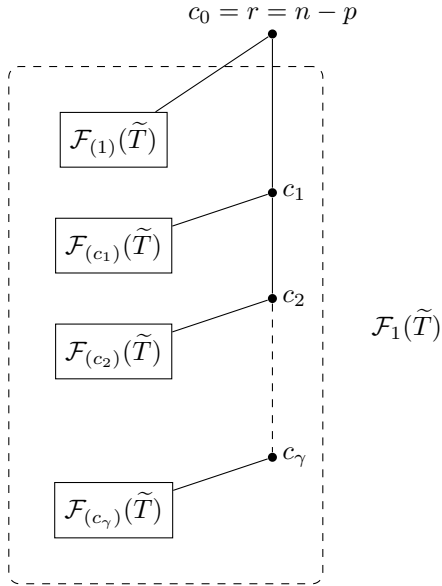
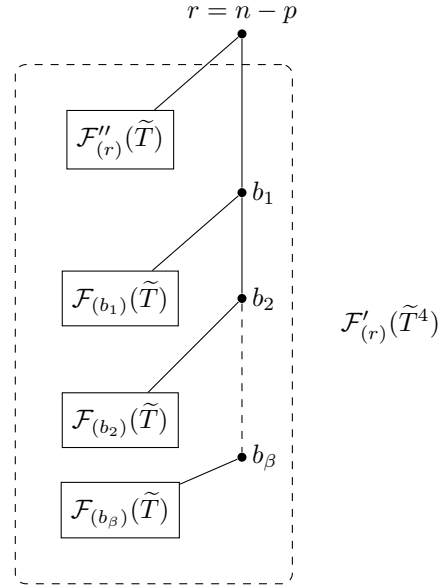

\noindent
Next, we construct the rooted tree $\widetilde{T}^{5}$ from $\widetilde{T}^{4}$ by removing the rooted subtree $\overline{\mathcal{F}}_{n-\ell+1}(\widetilde{T}^{4}) 
= (n-\ell+1)\vee \mathcal{F}_{n-\ell+1}(\widetilde{T}^{4})$ and attaching it to the vertex $c_{\gamma}$ in $\widetilde{T}^{4}$. 
Suppose that the maximum path $\Pmax(\widetilde{T}^{5})$ of $\widetilde{T}^{5}$ in $F_{\ell}$ is of the form
\[
\Pmax(\widetilde{T}^{5})\colon \quad r=n-p \to c_1 \to \cdots \to c_{\gamma} \to c_{\gamma+1} \to \cdots \to c_{\gamma+t}, 
\]
where $c_{\gamma+1}=n-\ell+1$, $c_{\gamma+i}\in F_{p+2}\setminus\{n-p-1\}$ $(1<i\leq t)$ and $c_{\gamma+t}$ is a leaf.
Observe that $\mathcal{F}'_{(r)}(\widetilde{T}^{5})=\mathcal{F}'_{(r)}(\widetilde{T}^{4})$.
We then construct $\widetilde{T}^{6}$ from $\widetilde{T}^{5}$ by removing $\mathcal{F}'_{(r)}(\widetilde{T}^{5})$ from $\widetilde{T}^{5}$ and attaching it to the leaf $c_{\gamma+t}$.
The tree $\widetilde{T}^{6}$ satisfies the hypotheses of Proposition~\ref{prop:T*_in_Cnlp} with the maximum path $\Pmax(\widetilde{T}^{6})$ of $\widetilde{T}^{6}$ in $F_{\ell}$ of the form
\[
\Pmax(\widetilde{T}^{6})\colon \quad r=n-p \to c_1 \to \cdots \to c_{\gamma} \to c_{\gamma+1} \to \cdots \to c_{\gamma+t} \to c_{\gamma+t+1} \to \cdots \to c_{\gamma+t+\beta},
\]
where $c_{\gamma+1}=n-\ell+1$, 
$c_{\gamma+i}\in F_{p+2}\setminus\{n-p-1\}$ $(1<i\leq t)$, 
$c_{\gamma+t+1}=b\in F_{\ell}\setminus F_{p+2}$,  and 
$c_{\gamma+t+j}\in F_{\ell}$ $(1\leq j\leq \beta)$.
Therefore, by Proposition~\ref{prop:T*_in_Cnlp}, we obtain $(\widetilde{T}^{6})^* = \widetilde{T}^{6}\bigl[ \mathcal{F}_{c_{\gamma+t+\beta}}(\widetilde{T}^{6}) \to 1\bigr] \in \Cnlp$.
We define $\Phi_p(\widetilde{T}) =(\widetilde{T}^{6})^*$.

\medskip

We now characterize the trees $T\in \Cnlp$ that arise from Case~(iii). 
Suppose that the maximum path $\Pmax(T)$ of $T$ in $F_{\ell}$ is of the form
\[
\Pmax(T)\colon\quad
r=n-p\to a_1\to a_2\to \cdots \to a_{\alpha-1}\to a_{\alpha},\qquad (\alpha>1),
\]
with $a_s=n-\ell+1$ for some $1<s\leq \alpha$.
%
If $s=\alpha$ or $a_{s+1}\notin F_{p+2}$, we set $t=0$. 
Otherwise, let $t$ be the largest positive integer such that $s+t\leq \alpha$ and $a_{s+j}\in F_{p+2}$ for $1\leq j\leq t$.

\medskip
If the vertex $1$ is a descendant of $a_{s-1}$ but not of $a_s$ in $T$, then $T$ arises from Subcase~(iiia). 
Indeed, consider the tree $T'=T\bigl[\mathcal{F}_1(T)\to a_{\alpha}\bigr].$
Next, consider the tree $T^{5} = T'\bigl[\mathcal{F}_{a_{s+t}}(T')\to 1\bigr]$,
obtained by attaching all the descendants of $a_{s+t}$ to the leaf $1$ in $T'$.
Let $\overline{\mathcal{F}}_{a_s}(T^{5}) = a_s\vee \mathcal{F}_{a_s}(T^{5})$ be the rooted subtree consisting of the vertex $a_s=n-\ell+1$ together with all of its descendants in $T^{5}$. 
Finally, we construct the tree $T^{6}$ from $T^{5}$ by removing the rooted subtree $\overline{\mathcal{F}}_{a_s}(T^{5})$ and attaching it to the root $r$ in $T^{5}$. 
Then $T^{6}\in \TildeTnpplustwop$, and reversing the above constructions, we recover $T$ from $T^{6}$. 
Hence $\Phi_p(T^{6})=T$, showing that $T$ arises from Subcase~(iiia).

\medskip
Suppose there exist integers $s\geq 1$ and $t\geq 0$ with $s+t<\alpha$ such that $a_s=n-\ell+1$, $a_{s+j}\in F_{p+2}$ for $1\leq j\leq t$, and $b:=a_{s+t+1}\in F_{\ell}\setminus F_{p+2}$.
Assume further that the vertex $1$ is a descendant of $a_s=n-\ell+1$ but is not a descendant of $b$ in $T$. 
Then $T$ arises from Subcase~(iiib). 
Indeed, the tree $T^{6}\in\TildeTnpplustwop$ constructed above satisfies the hypotheses of Subcase~(iiib), and reversing the construction yields $\Phi_p(T^{6})=T$.

\medskip
Finally, suppose there exists an index $s$ with $1<s\leq \alpha$ such that $a_s=n-\ell+1$, and let $t\geq 0$ be as defined above. Assume that $s+t<\alpha$ and set $b:=a_{s+t+1}\in F_{\ell}\setminus F_{p+2}$.
Suppose further that the vertex $1$ is a descendant of $a_{s+t}$ in $T$. 
Then $T$ arises from Subcase~(iiic).
Indeed, consider the rooted tree $T'=T\bigl[\mathcal{F}_1(T)\to a_{\alpha}\bigr]$. 
Let $\mathcal{F}_{(r)}(T')$ denote the rooted forest consisting of all descendants of the root $r$, excluding the rooted subtree $\overline{\mathcal{F}}_{a_1}(T') = a_1\vee\mathcal{F}_{a_1}(T')$.
Next, obtain the tree $T^7$ from $T'$ by removing $\mathcal{F}_{(r)}(T')$ from $T'$ and attaching it to $1$.
We then construct the rooted tree $T^{8}$ from $T^{7}$ by removing the rooted forest $\mathcal{F}_{a_{s+t}}(T^{7})$ and attaching it to the root $r$ in $T^{7}$. 
Subsequently, we construct the rooted tree $T^{9}$ from $T^{8}$ by removing the rooted subtree $\overline{\mathcal{F}}_{a_s}(T^{8})=a_s\vee\mathcal{F}_{a_s}(T^{8})$, consisting of $a_s=n-\ell+1$ together with all its descendants in $T^{8}$, and attaching it to the root $r$ in $T^{8}$. 
Finally, we construct the rooted tree $T^{10}$ from $T^{9}$ by removing the rooted subtree $\overline{\mathcal{F}}_{a_1}(T^{9})$ and attaching it to the vertex $1$ in $T^{9}$. 
Then $T^{10}\in\TildeTnpplustwop$, and $T^{10}$ satisfies the hypotheses of Subcase~(iiic). 
Reversing the above constructions, we recover $T$ from $T^{10}$, and hence $\Phi_p(T^{10})=T$.

\medskip

Since the above cases are mutually exclusive and exhaustive, and every tree $T\in\Cnlp$ arises from a unique element of $\Tnlminusonep$ or $\TildeTnpplustwop$, the map 
$\Phi_p\colon\Tnlminusonep\coprod\TildeTnpplustwop\longrightarrow\Cnlp$
is a bijection. 
This completes the proof.
\end{proof}

To determine the cardinality of $\Cnlp$, we compare the sizes of the subsets $\Tnlminusonep\subseteq\Cnlminusonep$ and $\Tnpplustwop\subseteq\Cnpplustwop$.
Our approach is based on Pr\"{u}fer codes of rooted trees, which allow us to relate $|\Tnlminusonep|$ to $|\Cnlminusonep|$, and  $|\Tnpplustwop|$ to $|\Cnpplustwop|$.

We recall that for a rooted tree $T_0$ with vertex set $[n]$,  where $n\geq 2$, the  \emph{Pr\"{u}fer code} $\pr(T_0)=(\lambda_1,\lambda_2,\ldots,\lambda_{n-1})$ is obtained by recursively removing the smallest non-root leaf and recording its parent in the remaining tree. 
The final entry $\lambda_{n-1}$ is the root of $T_0$. 
Moreover, every Pr\"{u}fer code uniquely determines a rooted tree with vertex set $[n]$.

\begin{proposition}\label{prop:T_vs_C}
For $0\leq p<\ell-2$, we have 
\[
|\Tnlminusonep|=|\Cnlminusonep|\left(\frac{n-\ell-1}{n-\ell}\right)
\qquad\text{and}\qquad
|\Tnpplustwop|=|\Cnpplustwop|\left(\frac{n-\ell-1}{n-p-3}\right).
\]
\end{proposition}

\begin{proof}
Let $T\in\Cnlminusonep$ or $T\in\Cnpplustwop$, and let $\Pmaximum^{\ell-1}(T)$ or $\Pmaximum^{p+2}(T)$, respectively, denote the maximum path of $T$ of the form
$r=n-p\to a_1\to a_2\to\cdots\to a_{\alpha}$.
Consider the tree $T^*=T\bigl[\F_1(T)\to a_{\alpha}\bigr]$, obtained by attaching the descendants of $1$ to the leaf $a_{\alpha}$ in $T$. 
By construction, vertex $1$ is a leaf of $T^*$. 
Let
\[
\hatCnl=\{T^* : T\in\Cnlminusonep\}
\qquad\text{and}\qquad
\hatCnp=\{T^* : T\in\Cnpplustwop\}.
\]
By Proposition~\ref{prop:T*_in_Cnlp}, the map $T\mapsto T^*$ is a bijection from $\Cnlminusonep$ to $\hatCnl$ and from $\Cnpplustwop$ to $\hatCnp$.
Thus,
\[
|\Cnlminusonep|=|\hatCnl| \qquad\text{and}\qquad |\Cnpplustwop|=|\hatCnp|.
\]
Let $\pr(\hatCnl)=\{\pr(T^*):T^*\in\hatCnl\}$ and $\pr(\hatCnp)=\{\pr(T^*):T^*\in\hatCnp\}$ denote the corresponding sets of Pr\"{u}fer codes. 
Since the map $T^*\mapsto\pr(T^*)$ is a bijection between rooted trees with vertex set $[n]$ and their Pr\"{u}fer codes, we have
\[
|\Cnlminusonep|=|\pr(\hatCnl)|
\qquad\text{and}\qquad
|\Cnpplustwop|=|\pr(\hatCnp)|.
\]
Let $A=\{2,3,\ldots,n-\ell+1\}$, $A'=A\setminus\{n-\ell+1\}$, and $B=\{2,3,\ldots,n-p-2\}$.

\medskip
\noindent\textbf{Enumeration of $\Tnlminusonep$.}
For every $T\in\Cnlminusonep$, we have $\Par_T(1)\in A$, so the 
first entry of $\pr(T^*)=(q_1(T^*),\ldots,q_{n-1}(T^*))$ satisfies 
$q_1(T^*)=\Par_T(1)\in A$. 
Let $\pi_1(\pr(T^*))=q_1(T^*)$ denote the first projection map, and define
\[
\varphi_1\colon\Cnlminusonep\longrightarrow A,
\quad \text{by} \quad
\varphi_1(T)=\pi_1(\pr(T^*))=q_1(T^*).
\]
Since $\Par_T(1)$ ranges over all of $A$, the map $\varphi_1$ is surjective, with fibers 
\[
\varphi_1^{-1}(a)=\{T\in\Cnlminusonep : \Par_T(1)=a\}.
\]
For any $a\in A$, let $T'$
be the tree obtained from $T$ by interchanging the labels $2$ and $a$.
Then $T\in\Cnlminusonep$ if and only if $T'\in\Cnlminusonep$.
Hence, the map $T\longmapsto T'$ is a bijection between the fibers $\varphi_1^{-1}(2)$ and $\varphi_1^{-1}(a)$.
Therefore, $|\varphi_1^{-1}(2)|=|\varphi_1^{-1}(a)|$, for any $a\in A$.
Hence, all fibers of $\varphi_1$ have equal cardinality.
Therefore, we have
\begin{equation}\label{eq:|Cnl-1p|}
|\Cnlminusonep|=\sum_{a\in A}|\varphi_1^{-1}(a)|
= \sum_{a\in A}|\varphi_1^{-1}(2)|
=(n-\ell)\,|\varphi_1^{-1}(2)|.
\end{equation}
Since we can express $\Tnlminusonep=\{T\in\Cnlminusonep : \Par_T(1)\in A'\}$, we have 
\begin{equation}\label{eq:|Tnl-1p|}
|\Tnlminusonep|=\sum_{a\in A'}|\varphi_1^{-1}(a)|
=(n-\ell-1)\,|\varphi_1^{-1}(2)|.
\end{equation}
By combining \eqref{eq:|Cnl-1p|} and \eqref{eq:|Tnl-1p|} we obtain
\[
|\Tnlminusonep|=\left(\frac{n-\ell-1}{n-\ell}\right)|\Cnlminusonep|.
\]

\medskip
\noindent\textbf{Enumeration of $\Tnpplustwop$.}
For every $T\in\Cnpplustwop$, we have $\Par_T(1)\in B$, so the 
same argument applies with the surjective map
\[
\psi_1\colon\Cnpplustwop\longrightarrow B,
\quad \text{given by} \quad \psi_1(T)=\pi_1(\pr(T^*))=q_1(T^*).
\]
Again, all fibers of $\psi_1$ have equal cardinality.
In particular, $|\psi_1^{-1}(2)| = |\psi_1^{-1}(b)|$, for any $b\in B$.
Thus,
\[
|\Cnpplustwop|
=\sum_{b\in B}|\psi_1^{-1}(b)|
= \sum_{b\in B}|\psi_1^{-1}(2)|
=(n-p-3)\,|\psi_1^{-1}(2)|.
\]
Since $\Tnpplustwop=\{T\in\Cnpplustwop : \Par_T(1)\in A'\}$, we obtain
\[
|\Tnpplustwop|
=\sum_{a\in A'}|\psi_1^{-1}(a)|
=(n-\ell-1)\,|\psi_1^{-1}(2)|.
\]
Therefore,
\[
|\Tnpplustwop|=\left(\frac{n-\ell-1}{n-p-3}\right)|\Cnpplustwop|.
\]
This completes the proof.
\end{proof}

Combining the results established so far, we can now determine the cardinality of the set $\Cnlp$.

\begin{theorem}\label{thm:Cnlp_count}
For $0\leq p\leq\ell-2$, we have
\[
|\Cnlp|=(\ell-p-1)(n-1)^{n-4-p}(n-2)^p(n-\ell-1).
\]
\end{theorem}

\begin{proof}
We prove the theorem by induction on $\ell$, for fixed $p$, where $\ell\geq p+2$. 
The base case is $\ell=p+2$, equivalently $p=\ell-2$.
By Proposition~\ref{prop:Tnl_bijection}, we have $|\Cnllminustwo|=(n-1)^{n-\ell-2}(n-2)^{\ell-2}(n-\ell-1)$, which 
agrees with the required formula upon substituting $p=\ell-2$.

Now we assume that $\ell>p+2$, and suppose the theorem holds for all $\mathcal{C}_{n,\ell'}^p$ with $\ell'<\ell$. 
Thus by the induction hypothesis, we have 
\[
|\Cnlminusonep|=(\ell-p-2)~(n-1)^{n-4-p}~(n-2)^p~(n-\ell).
\]
Moreover, since $p+2<\ell$, by the induction hypothesis we have
\[
|\Cnpplustwop|=(n-1)^{n-4-p}(n-2)^p(n-p-3).
\]
By Proposition~\ref{prop:T_vs_C}, we obtain
\[
|\Tnlminusonep|
=|\Cnlminusonep|\left(\frac{n-\ell-1}{n-\ell}\right)
=(\ell-p-2)(n-1)^{n-4-p}(n-2)^p(n-\ell-1),
\]
and 
\[
|\TildeTnpplustwop|
=|\Tnpplustwop|
=|\Cnpplustwop|\left(\frac{n-\ell-1}{n-p-3}\right)
=(n-1)^{n-4-p}(n-2)^p(n-\ell-1).
\]
Finally, by Proposition~\ref{prop:Phi_p}, the map $\Phi_p:\Tnlminusonep\coprod\TildeTnpplustwop\longrightarrow\Cnlp$ is a bijection, and therefore
\[
|\Cnlp|=|\Tnlminusonep|+|\TildeTnpplustwop|=(\ell-p-1)(n-1)^{n-4-p}(n-2)^p(n-\ell-1).
\]
This completes the induction and hence the proof.
\end{proof}

Since $\Cnl=\coprod_{p=0}^{\ell-2}\Cnlp$, the following corollary is immediate from Theorem~\ref{thm:Cnlp_count}.

\begin{corollary}\label{cor:Cnl_count}
We have
\[
|\Cnl|
=
\sum_{p=0}^{\ell-2}
(\ell-p-1)
(n-1)^{\,n-4-p}
(n-2)^p
(n-\ell-1).
\]
\end{corollary}

We have already shown in Theorem~\ref{thm:|Un|} that $|\UnFl|=(n-1)^{n-\ell-2}(n-2)^{\ell}(n-\ell-1)$.
Moreover, by Theorems~\ref{thm:SPFGl_to_UnFl} and~\ref{thm:Im(Psi)}, we have
\[
|\Im(\phi_{\Gl})|= |\SPF(\Gl)|=|\UnFl\setminus\Cnl|.
\]
One of the interesting applications of the preceding computations is the following explicit enumeration of spherical $\Gl$-parking functions.

\begin{theorem}\label{thm:|SPF|}
For $n\geq3$ and $2\leq\ell\leq n-2$,
\[
|\SPF(\Gl)|=(n-1)^{n-3}(n-\ell-1)^2.
\]
\end{theorem}

We first establish the following identity.

\begin{lemma}\label{lem:key_identity}
For $n\geq\ell+3$ and $\ell\geq 2$, we have
\[
(n-1)^{n-\ell-2}(n-2)^{\ell}(n-\ell-1)
-
(n-1)^{n-3}(n-\ell-1)^2
=
\sum_{p=0}^{\ell-2}
(\ell-p-1)(n-1)^{n-4-p}(n-2)^p(n-\ell-1).
\]
\end{lemma}

\begin{proof}
The identity is established by a direct computation. Starting from 
the left-hand side, we compute
\begin{equation*}
\begin{split}
&\hspace{-0.7cm}(n-1)^{n-\ell-2}(n-2)^{\ell}(n-\ell-1)
-(n-1)^{n-3}(n-\ell-1)^2\\
&\hspace{0.5cm} =(n-1)^{n-\ell-2}(n-\ell-1)
\Bigl[(n-2)^{\ell}-(n-1)^{\ell-1}(n-\ell-1)\Bigr]\\
& \hspace{0.5cm} = (n-1)^{n-\ell-2}(n-\ell-1)
\Bigl[(n-2)^{\ell}-(n-1)^{\ell}+\ell(n-1)^{\ell-1}\Bigr]\\
& \hspace{0.5cm} =(n-1)^{n-\ell-2}(n-\ell-1)\sum_{j=0}^{\ell-1}\Bigl[(n-1)^{\ell-1} - (n-2)^j(n-1)^{\ell-1-j}\Bigr]\\
& \hspace{0.5cm} =(n-1)^{n-\ell-2}(n-\ell-1)
\sum_{j=1}^{\ell-1}
(n-1)^{\ell-1-j}\Bigl[ (n-1)^j-(n-2)^j \Bigr]\\
& \hspace{0.5cm} =(n-1)^{n-\ell-2}(n-\ell-1)
\sum_{j=1}^{\ell-1}
(n-1)^{\ell-1-j}\Biggl[ \sum_{p=0}^{j-1}(n-1)^{j-1-p}(n-2)^p \Biggr]\\
& \hspace{0.5cm}  =(n-1)^{n-\ell-2}(n-\ell-1)
\Biggl(\sum_{p=0}^{\ell-2}\sum_{j=p+1}^{\ell-1}
(n-1)^{\ell-2-p}(n-2)^p\Biggr)\\
& \hspace{0.5cm}=(n-1)^{n-\ell-2}(n-\ell-1)
\Biggl(\sum_{p=0}^{\ell-2}(\ell-p-1)(n-1)^{\ell-2-p}(n-2)^p\Biggr)\\
& \hspace{0.5cm} =\sum_{p=0}^{\ell-2}(\ell-p-1)(n-1)^{n-4-p}(n-2)^p(n-\ell-1).
\end{split}
\end{equation*}
This proves the lemma.
\end{proof}

Combining Theorem~\ref{thm:|Un|}, Corollary~\ref{cor:Cnl_count}, and
Lemma~\ref{lem:key_identity}, we obtain the following corollary.

\begin{corollary}\label{cor:image_count} We have 
\begin{equation*}\label{eq:image_count}
\begin{split}
|\Im(\phi_{\Gl})| & = |\UnFl|-|\Cnl| \\
&  =  (n-1)^{\,n-\ell-2}(n-2)^{\,\ell}(n-\ell-1) - \sum_{p=0}^{\ell-2} (\ell-p-1)(n-1)^{\,n-4-p}(n-2)^p(n-\ell-1)\\
&  =  (n-1)^{n-3}(n-\ell-1)^2.
\end{split}
\end{equation*}
\end{corollary}

\begin{remark}
Lemma~\ref{lem:key_identity} admits a natural combinatorial  interpretation. 
The identity reflects the decomposition
\[
\UnFl=\Bigl(\Im(\phi_{\Gl})\Bigr) \coprod \left(\coprod_{p=0}^{\ell-2}\Cnlp\right),
\]
that is, every tree in $\UnFl$ belongs either to $\Im(\phi_{\Gl})$ or to exactly one of the subsets $\Cnlp$. 
Rearranging the identity of Lemma~\ref{lem:key_identity}, we also obtain a decomposition of the elements of $\Im(\phi_{\Gl})$ according to the choice of root in the corresponding trees.
\end{remark}

We are now in a position to prove Theorem~\ref{thm:|SPF|}.

\begin{proof}[Proof of Theorem~\ref{thm:|SPF|}]
By Theorem~\ref{thm:SPFGl_to_UnFl}, the spherical $\Gl$-parking 
functions are in bijection with $\Im(\phi_{\Gl})$. Hence, by 
Theorem~\ref{thm:Im(Psi)} and Corollary~\ref{cor:image_count}, 
we obtain the stated formula.
\end{proof}

\begin{remark}
An \emph{$\ell$-claw} is a graph with $\ell+1$ vertices consisting  of a central vertex adjacent to $\ell$ other vertices, with no edges among the outer vertices. 
Equivalently, it is the complete bipartite graph $K_{1,\ell}$. 
The graph $\Gl$ is obtained from $\Knn$ by deleting an $\ell$-claw at vertex $1$, that is, by removing the $\ell$ edges joining vertex $1$ to the vertices of $F_{\ell}$.
More generally, let $\widetilde{G}_{\ell}$ be a graph with vertex set $[n]\cup\{0\}$ and root $0$, obtained from $\Knn$ by deleting an $\ell$-claw at some vertex $i\in[n]\setminus\{1\}$, 
that is, by removing the $\ell$ edges joining $i$ to distinct 
vertices $j_1,\ldots,j_{\ell}\in[n]\setminus\{i\}$. 
By relabeling the vertices, we may assume $i=1$ and 
$\{j_1,\ldots,j_{\ell}\}=F_{\ell}$, so $\widetilde{G}_{\ell}$ is isomorphic to $\Gl$. 
Hence their skeleton ideals are isomorphic, and in particular their $(n-2)$-skeleton ideals have the same number of standard monomials. Therefore,
\[
|\SPF(\widetilde{G}_{\ell})|=|\SPF(\Gl)|.
\]
Thus, Theorem~\ref{thm:|SPF|} computes the number of spherical 
$G$-parking functions for every graph $G$ obtained from $\Knn$ by 
deleting an $\ell$-claw.
\end{remark}

\begin{remark}
Although the previous remark shows that every graph obtained from $\Knn$ by deleting an $\ell$-claw has the same number of spherical $G$-parking functions, a combinatorial description of spherical parking functions in terms of uprooted spanning  is not known in general.
For such a graph $\widetilde{G}_{\ell}$, the image  of the map $\phi_{\widetilde{G}_{\ell}}: \SPF(\widetilde{G}_{\ell}) \longrightarrow \U(\widetilde{G}_{\ell}')$,  where $\widetilde{G}_{\ell}' = \widetilde{G}_{\ell}-\{0\}$,  need not admit a characterization analogous to that obtained for $\Gl$,
and the associated uprooted spanning trees may be considerably more difficult to describe.
For example, let us consider $\widetilde{G}=\Knn -\{e_{n-1,n}\}$.
For $n\geq 3$, the map $\phi_{\widetilde{G}}:\SPF(\widetilde{G})\longrightarrow \mathcal U(\widetilde{G}-\{0\})$ is injective but not surjective. 
In particular, when $n=4$, $|\SPF(K_{5}-\{e_{3,4}\})|=12$, whereas the number of uprooted trees with vertex set $[4]$ having no edge between $3$ and $4$ is $17$.
\end{remark}

The results obtained in this paper suggest several natural directions for further investigation. 
While we have obtained explicit enumeration formulas and combinatorial descriptions for spherical $\Gl$-parking functions and uprooted spanning trees, many questions remain open.
We conclude with some problems and directions for future research.

\noindent\textbf{Question 1.}
\emph{The formula for $|\UnFl|$ in Theorem~\ref{thm:|Un|} is given by 
the simple product
\[
|\UnFl|=(n-1)^{n-\ell-2}(n-2)^{\ell}(n-\ell-1).
\]
Does there exist a direct combinatorial proof of this formula, 
perhaps via a Pr\"{u}fer-type encoding or a bijection analogous 
to that of Chauve, Dulucq, and Guibert~\cite{CDG}?}

\noindent\textbf{Question 2.}
\emph{Determining $|\SPF(G)|$ for an arbitrary graph $G$ appears to be difficult, and explicit formulas are known only in a few cases. 
Can such formulas be obtained for other natural families of graphs, such as complete bipartite graphs? 
In this direction, Kumar, Lather, and Sonica~\cite{CGS} established that for $m\geq1$,
\[
|\SPF(K_{m+1,2})|=(m-1)2^{m} + 1.
\]
It would be interesting to determine $|\SPF(K_{m,n})|$ for general 
$m$ and $n$.}